\documentclass[11pt]{amsart}
\usepackage{amssymb, amscd, amsmath, amsthm, epsf,  latexsym,color ,mathtools} 

\usepackage[lmargin=1.1in,rmargin=1.1in,tmargin=1in,bmargin=1in]{geometry}
\usepackage{setspace} 
\usepackage[pdfpagelabels,bookmarksdepth=3]{hyperref}
\hypersetup{colorlinks=true,citecolor=black,urlcolor=black,linkcolor=black}
\usepackage{pinlabel}
\usepackage{pb-diagram}
\usepackage{soul}
 \usepackage{tikz-cd}
\usepackage{comment}
\usepackage{float}
\usepackage{graphicx}
\usepackage{mathrsfs}
\usepackage{wrapfig}
\restylefloat{figure}
\usepackage[font=small,labelfont=bf]{caption}
\usepackage{subcaption}
\usepackage{relsize}

 \usepackage{color}
 \usepackage{xcolor}

\newtheorem{theoremABC}{Theorem}

\newtheorem{conjectureABC}[theoremABC]{Conjecture}

\numberwithin{equation}{section}
\newtheorem{theorem}{Theorem}[section]
\newtheorem*{theorem*}{Theorem}
\newtheorem{corollary}[theorem]{Corollary}
\newtheorem{lemma}[theorem]{Lemma}

\newtheorem{proposition}[theorem]{Proposition}

\newtheorem{conjecture}[theorem]{Conjecture}

\theoremstyle{definition}
\newtheorem{definition}[theorem]{Definition}
\newtheorem{example}[theorem]{Example}

\newtheorem*{remark*}{Remark}

\newcommand{\Symp}{\mathit{\mathcal{S}ymp}}

\newcommand{\ep}{\epsilon}

\newcommand{\bbi}{{{\bf i}}}
\newcommand{\bbj}{{{\bf j}}}
\newcommand{\bbk}{{{\bf k}}}

\newcommand{\FF}{{\mathbb F}}
\newcommand{\ZZ}{{\mathbb Z}}
\newcommand{\RR}{{\mathbb R}}
\newcommand{\HH} {{\mathbb H}}
\newcommand{\CC}{{\mathbb C}}

\newcommand{\cA}{{\mathcal A}}

\newcommand{\cB}{{\B}}

\newcommand{\cM}{{\mathcal M}}

\newcommand{\rin}{{\operatorname{1}}}
\newcommand{\rout}{{\operatorname{0}}}

\def\co{\colon\thinspace}

\newcommand{\RN}[1]{%
  \textup{\uppercase\expandafter{\romannumeral#1}}%
}

\newcommand{\nat}{\natural}

\DeclareMathOperator{\Real}{Re}
\DeclareMathOperator{\Ima}{Im}

\DeclareMathOperator{\Hom}{Hom}

\graphicspath{ {figures/} }

 \renewcommand{\qed}{\hfill$\square$}

\newcommand{\tP}{\widetilde{P}}

\newcommand{\ba}{\bar{a}}
\newcommand{\bb}{\bar{b}}

\newcommand{\be}{\bar{e}}
\newcommand{\bff}{\bar{f}}
\newcommand{\bq}{\bar{q}}
\newcommand{\bp}{\bar{p}}
\newcommand{\bh}{\bar{h}}
\newcommand{\bg}{\bar{g}}

\newcommand{\pertSolidToriInPreliminaries}{{Q}} 
\newcommand{\pertCurveInPreliminaries}{{q}} 
\newcommand{\tracelessTwoSphere}{{C_\bbi}} 
\newcommand{\coreMap}{u} 
\newcommand{\modelMap}[1]{v_{#1}} 

\newcommand{\allLoopImmersions}[1]{\mathrm{LoopLag}_{#1}(P^*)} 
\newcommand{\allArcImmersions}[1]{\mathrm{ArcLag}_{#1}(P^*)} 

\newcommand{\allLoopImmersionsout}[1]{\mathrm{LoopLag}_{#1}((P^\rout)^*)} 
\newcommand{\allLoopImmersionsin}[1]{\mathrm{LoopLag}_{#1}((P^\rin)^*)}

\renewcommand{\r}{r} 
\newcommand{\rrout}{r^\rout} 
\newcommand{\rrin}{r^\rin} 

\newcommand{\iperpcirc}{S^1 _{\bbi^\perp}}
\newcommand{\Xupstairs}{\widetilde{X}}  
\newcommand{\Xdownstairs}{X}
\newcommand{\Vupstairs}{{\widehat{V}}} 
\newcommand{\V}{{V}} 
\newcommand{\mapVR}{{\Phi}} 
\newcommand{\mapVRmodInvolution}{{\phi}} 
\newcommand{\DehnTwistMap}{{\xi}} 
\newcommand{\genusFiveSurface}{{\widehat{\Sigma}}} 
\newcommand{\genusFiveSurfaceFamily}{{\widehat{\mathbf\Sigma}}} 
\newcommand{\genusThreeSurface}{{\Sigma}} 
\newcommand{\genusThreeSurfaceFamily}{{\mathbf\Sigma}} 
\newcommand{\FPS}{{(S^2, 4)}} 
\newcommand{\groupPrelims}{{\Gamma}} 
\newcommand{\link}{\mathcal{L}} 

\newcommand{\preScaledNatural}[2]{\raisebox{1pt}{\scalebox{2.3}[1]{$#1\natural$}}}
\newcommand{\scaledNatural}{{\mathpalette\preScaledNatural\relax}}
\DeclareMathOperator{\NAT}{{\scaledNatural}}
\newcommand{\tNAT}{\widetilde{\scaledNatural}}

\renewcommand{\l}{l}
\renewcommand{\emptyset}{\varnothing}

\newcommand{\genInt}{\cap}
\newcommand{\Fuk}{\mathcal{F}}
\newcommand{\wrFuk}{{\mathcal{W}}}
\newcommand{\HFhat}{\widehat{\mathit{HF}}}
\newcommand{\HFKhat}{\widehat{\mathit{HFK}}}
\newcommand{\CFDhat}{\widehat{\mathit{CFD}}}
\newcommand{\CFAhat}{\widehat{\mathit{CFA}}}
\newcommand{\HFT}{\mathit{HFT}}
\newcommand{\Kh}{\mathit{Kh}}
\newcommand{\Khr}{\widetilde{\mathit{Kh}}}
\newcommand{\BNr}{\widetilde{\mathit{BN}}}
\newcommand{\HF}{\mathit{HF}}
\newcommand{\CF}{\mathit{CF}}
\newcommand{\pt}{\text{pt}}
\DeclareMathOperator{\Tw}{Tw}
\newcommand{\X}{\mathcal{X}}
\newcommand{\arcHor}{\mathbf{a}^\bullet}
\newcommand{\arcVer}{\mathbf{a}^\circ}
\newcommand{\id}{\text{id}}
\newcommand{\B}{\mathcal{B}}
\newcommand{\mcH}{[\mathbf{I}\to\mathbf{I}]}
\newcommand{\N}{\mathcal N}
\DeclareMathOperator{\Lag}{\mathrm{Lag}}

\thanks{AK is supported by an AMS-Simons travel grant. CH and PK were supported by Simons Collaboration Grants for Mathematicians.}

\author{Guillem Cazassus}
\address{Department of Mathematics, Indiana University, Bloomington, IN 47405 }
\email{g.cazassus@gmail.com}

\author{Christopher Herald}
\address{Department of Mathematics and Statistics, University of Nevada,  Reno, NV 89557} 
\email{herald@unr.edu}

 \author{Paul Kirk}
\address{Department of Mathematics, Indiana University, Bloomington, IN 47405} 
\email{pkirk@indiana.edu}

 \author{Artem Kotelskiy}
\address{Department of Mathematics, Indiana University, Bloomington, IN 47405} 
\email{artofkot@gmail.com}

\subjclass[2010]{Primary 57K18, 57K31, 57R58; Secondary 81T13} 
\keywords{Pillowcase, holonomy perturbation, instanton homology, Floer homology, flat moduli space, traceless character variety, quilts, Lagrangian correspondence, figure eight bubble}

\begin{document}

\title{The correspondence induced on the pillowcase by the earring tangle}

\begin{abstract} 


The earring tangle consists of four strands $4\text{pt} \times I \subset S^2 \times I$ and one meridian around one of the strands. Equipping this tangle with a nontrivial $SO(3)$ bundle, we show that its traceless $SU(2)$ flat moduli space is topologically a smooth genus three surface. We also show that the restriction map from this surface to the traceless flat moduli space of the boundary of the earring tangle is a particular Lagrangian immersion into the product of two pillowcases. The latter computation suggests that  figure eight bubbling---a subtle degeneration phenomenon predicted by Bottman and Wehrheim---appears in the context of traceless character varieties. 
\end{abstract}

\maketitle



\section{Introduction}

 \nopagecolor

\begin{wrapfigure}{r}{0.3\textwidth}
\vspace{-0.2cm}
\labellist 
\pinlabel $\color{blue}w$ at 258 238
\endlabellist
\centering
\includegraphics[width=0.3\textwidth]{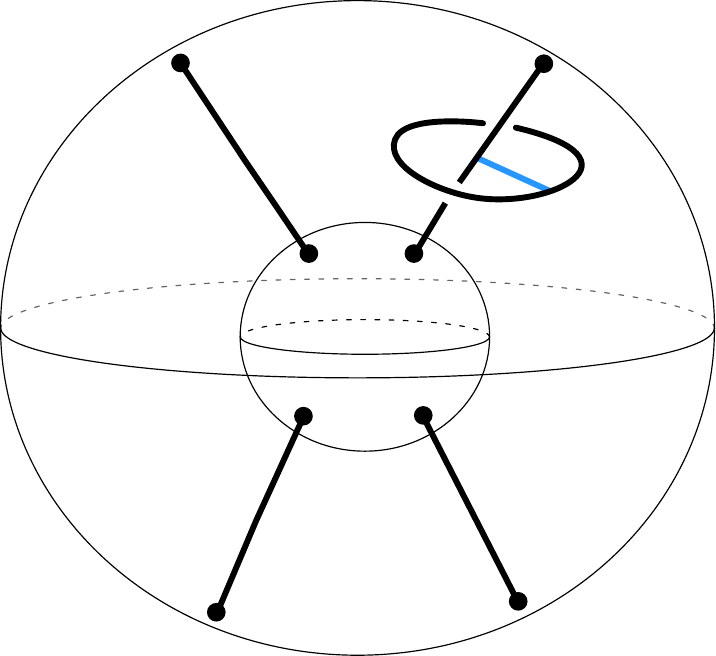}
\caption{The earring tangle}
\label{cob_new3fig}
\end{wrapfigure}

\subsection*{Overview} The main results of this article, Theorems~\ref{thmA} and~\ref{thmB} stated below,  identify the holonomy perturbed traceless $SU(2)$ flat moduli space 
of the {\em earring tangle}, depicted on the right in Figure~\ref{cob_new3fig}, as a smooth genus three surface. In addition, the restriction map from this genus three surface to the (symplectic) flat moduli space of the tangle's boundary is identified, up to a small homotopy, as a certain Lagrangian immersion of a genus three surface into $P\times P$, in the manner described in Figure~\ref{fig:description_of_the_map}.    

Here, $P$ denotes the {\em pillowcase,} a 2-dimensional symplectic orbifold, homeomorphic to $S^2$, obtained as the quotient of the torus by the elliptic involution $\iota$.  The pillowcase $P$ is a real semi-algebraic variety isomorphic to the  traceless flat $SU(2)$ moduli space of $(S^2,4)$, a 2-sphere equipped with four punctures (see Lemma~\ref{pillowcase}).

The theorem holds for certain  arbitrarily small holonomy perturbation data, supported in tubular neighborhoods of the (red and green) curves with meridians labeled $p$ and $q$ depicted on the right in  Figure~\ref{cobordism3fig}. 

In  Example~\ref{ex:Rzero},  we show that the  {\em unperturbed} flat moduli space  of the earring tangle  (denoted $\NAT_0$) is a singular 3-dimensional space obtained by collapsing to points each of the four singular fibers of the Seifert fibered 3-manifold $\Sigma(S^2,2,2,2,2)$.  The restriction map to the two pillowcases is a composition $\NAT_0 \to \triangle_P \to P\times P$ where the second map is the diagonal, and the first map is a circle fibration away from the corners and has single-point fibers over the corners. The neighborhoods near those four singularities are cones over a torus.  

  In order to apply Floer-theoretic constructions, one must therefore properly perturb the flatness equation  in order
for   the restriction map from the perturbed moduli space of the tangle  to the moduli space of its boundary be {\em a Lagrangian immersion into the smooth top stratum}.  Theorem \ref{thmA} accomplishes this, using an explicit  holonomy perturbation supported along the two curves.
 
Before the differential topological analysis that led to the proof of Theorem \ref{thmA}, the topology of the perturbed moduli space $\NAT_s$ and its restriction map to $P\times P$ were not known, but there were two clues.  First, considering straight corner-to-corner arcs $A_0,A_1,A_2$ from the bottom left corner of the pillowcase with slopes $0, +1, \infty$ respectively (see Figure~\ref{pillowgensfig}),  the effect of adding an earring to a trivial 2-component tangle in a ball analyzed in \cite{HHK1} suggested that the composition of the correspondence $\NAT_s$  with either of the arcs $A_k$ is a figure eight curve close to the arc $A_k$, slightly missing the corresponding corners. This can be interpreted as indicating how $\NAT_s \to P\times P$ should meet $ (A_0 \bigcup A_1 \bigcup A_2)\times P\subset P\times P$.  Second, since under perturbation  a Bott-Morse function with critical circle can be turned into a Morse function with two critical points, there is some reason to expect that the Seifert fibration  $\NAT _0$  in the previous paragraph would resolve into a surface that double covers $\triangle_P$, at least  away from the corners.  The genus three surface described in Figure~\ref{fig:description_of_the_map} may be viewed as (a smooth manifestation of) two copies of $\triangle_P$ with open neighborhoods of the corners removed, with annuli connecting the boundary circles in such a way to be consistent with the intersection information above.

\medskip

 The earring  tangle is obtained from the product tangle $(S^2,4)\times [0,1]$ by a process introduced by Kronheimer-Mrowka which we call {\em placing an earring  at a base point on a tangle}, in this instance on the top right component in Figure~\ref{cob_new3fig}.  The process  first adds an extra closed component  to   a tangle, namely a small meridian at the base point. Then, an $SO(3)$ bundle over the tangle complement is determined by requiring its second Stiefel-Whitney class to be Poincar\'e dual to an  arc joining the extra component to the strand of the original tangle (the arc labeled $w$ in Figure~\ref{cob_new3fig}). In the language of gauge theory, adding an earring defines an {\em admissible} $SO(3)$ bundle. 
 
 \medskip

  Considering the earring tangle as a cobordism from $(S^2,4)\times\{0\} $ to
 $(S^2,4)\times\{1\} $  leads to the flexible perspective of Floer field theory~\cite{WW}, which has its origins in Floer's work on instanton homology~\cite{Flo}, and falls under the wider framework of the Atiyah-Floer conjecture~\cite{MR974342}.  From this perspective (which we review in Section~\ref{sec:fft}), the moduli space of the earring tangle induces an endomorphism of the set of (immersed, unobstructed, suitably transverse, etc.) Lagrangian curves in $P$, that is, the set of objects of the wrapped Fukaya category of $P$.

  Theorem~\ref{thmB}, whose statement is geometrically
 described in Figure~\ref{example1fig}, essentially reformulates Theorem~\ref{thmA} in this language.  It provides a {\em geometric} description of  the action of this endomorphism on objects in the wrapped Fukaya category of $P$.   Admissibility in the sense of gauge theory  manifests itself via Theorem~\ref{thmB}  in an interesting way:   the earring endomorphism
 maps the wrapped Fukaya category of the pillowcase into the (sub-$A_\infty$) ordinary Fukaya category of   \emph{compact}  immersed curves in $P$  which miss the orbifold points.

\medskip

The second part of this article places our results in context,  viewing Theorem~\ref{thmB} as a partial calculation of this endomorphism.  We examine how the notions of   Floer field theories,   quilts and  composition for  immersed generalized Lagrangian correspondences, figure eight bubbling  and bounding cochains, and the relationship between the wrapped Fukaya category of the pillowcase  and instanton and Khovanov homology, can be explored in the concrete context of this interesting and nontrivial endomorphism.

\subsection*{Statement of results}

In preparation for the statements of Theorems~\ref{thmA} and~\ref{thmB},  note that the boundary of the earring tangle is the 
disjoint union of two copies of $(S^2,4)$, and hence its traceless flat moduli space is isomorphic to the   product  $P\times P$.  (Definitions of traceless and holonomy-perturbed flat moduli spaces are provided in Section~\ref{prelim}.)  Denote by  $P^*$ the smooth top (2-dimensional) stratum of the pillowcase $P$, obtained by removing the four orbifold points (called the {\em corners}) from $P=T^2/\iota$.  The smooth manifold $P^*$ is equipped with the symplectic form descended from the torus. 

 Throughout this article, if $M= (M,\omega)$ denotes a symplectic manifold,  $M^-= (M,-\omega)$ denotes the same smooth manifold with the negative symplectic form. In addition, the two boundary components of the earring tangle are distinguished by the superscripts $\rout$ or $\rin$.  Hence $P^\rin$ is the traceless flat moduli space of $(S^2,4)^\rin=(S^2,4)\times\{1\}$.  The notation $P\times P$ is used as shorthand for $P^\rout\times P^\rin$ and $P^*\times P^*$ for the top smooth symplectic stratum $((P^\rout)^*)^- \times (P^\rin)^*$, when clear from context.    Let $\Sigma$ denote a fixed closed surface of genus three.

\begin{theoremABC}
 \label{thmA}
 Given $s\in \RR$, let  $\NAT_s$ denote the holonomy-perturbed $SU(2)$  traceless flat moduli space of the earring tangle perturbed along the two perturbation curves indicated in Figure~\ref{cobordism3fig}.

\smallskip

The spaces $\NAT_s, ~0<|s|<\ep$ form a smooth bundle over   $(-\ep,0)\cup (0,\ep)$  for 
$\ep>0$ small enough with fiber a smooth orientable genus 3 surface $\Sigma$.  
This family extends to a smooth  bundle over $(-\ep, \ep)$, which we identify with  $\Sigma\times(-\ep,\ep)$, mapping to $P\times P$.  For $s\ne 0$ small enough, the restriction map   $r \co \NAT_s\to P\times P$ has image in $P^*\times P^*$ and is a Lagrangian immersion $r \co \NAT_s \rightarrow (P^*)^-\times P^*$.  

The Lagrangian immersion  $r \co \NAT_s \rightarrow (P^*)^-\times P^*$ is   $\ep$-homotopic to the Lagrangian immersion  $\modelMap{\delta}=\modelMap{\delta}^\rout \times \modelMap{\delta}^\rin \co \genusThreeSurface  \rightarrow (P^*)^-\times P^*$  with $\modelMap{\delta}^\rin=\modelMap{\delta}^\rout\circ D $, where $D $ is the composite of the four Dehn twists along disjoint circles which separate $\Sigma\cong\NAT_s$ into two four-punctured 2-spheres $\Sigma_+,~\Sigma_-$,   as illustrated in Figure~\ref{fig:description_of_the_map}. 
\end{theoremABC}

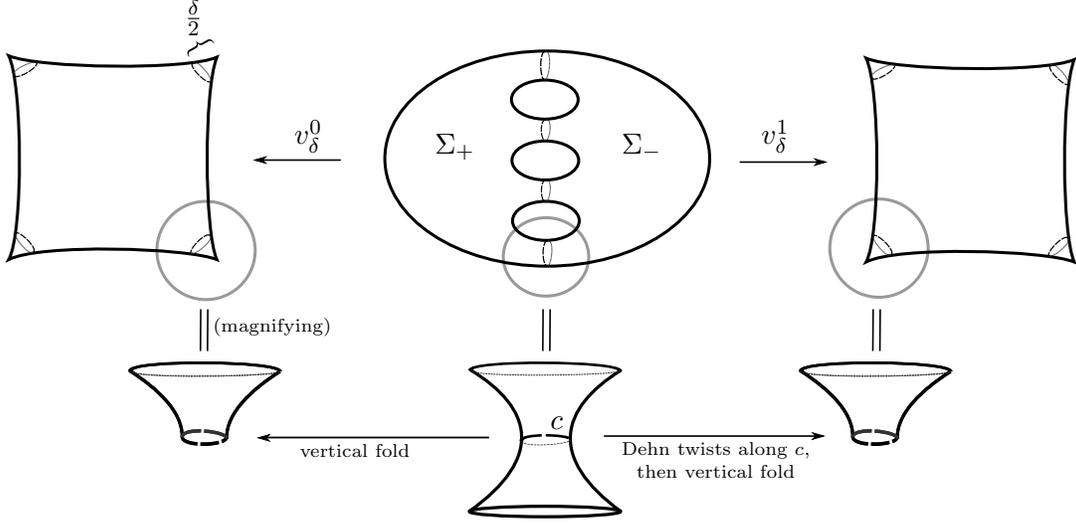
\begin{figure}[ht] 
\labellist 
\pinlabel $c$ at 260 45
\pinlabel $\tfrac \delta 2$ at 89 236
\pinlabel {\tiny vertical fold} at 165 32
\pinlabel {\tiny (magnifying)} at 126 90
\pinlabel {\tiny Dehn twists along $c$,} at 335 32
\pinlabel {\tiny then vertical fold} at 335 22
\pinlabel $\modelMap{\delta}^\rout$ at 143 180
\pinlabel $\modelMap{\delta}^\rin$ at 363 180
\pinlabel $\genusThreeSurface_+$ at 212 175
\pinlabel $\genusThreeSurface_-$ at 300 175
\endlabellist
\centering
\includegraphics[scale=0.8]{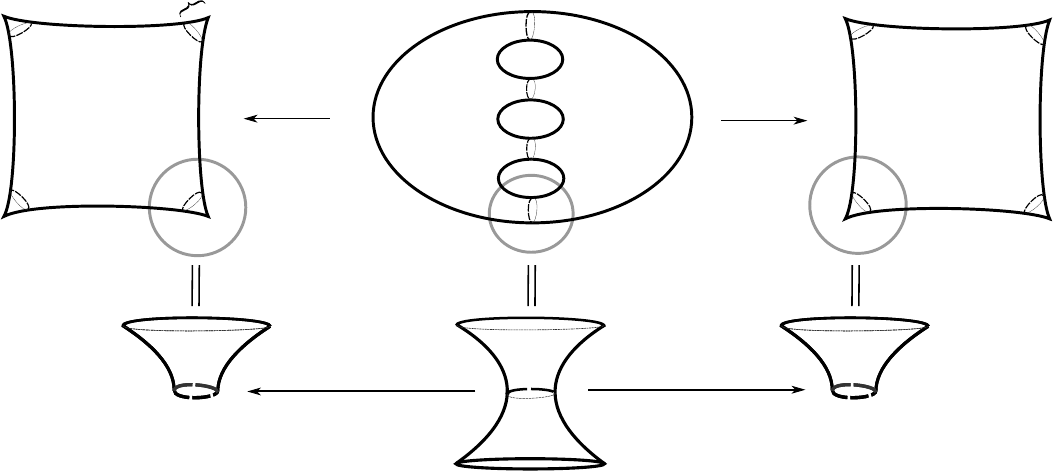}
\caption{
The map $\modelMap{\delta}^\rout$ sends both $\genusThreeSurface_-$ and $\genusThreeSurface_+$ bijectively to $P^*$ with corners cut, and glues these two maps along the dashed circles. The map $\modelMap{\delta}^\rin$ is the same but pre-composed with the Dehn twists along $c$, one for each corner, as is indicated in the bottom; see also Figure~\ref{correspondtwfig}.}\label{fig:description_of_the_map} 

\end{figure}

A Lagrangian immersion into a product $L\to M^-\times N$ induces a {\em Lagrangian correspondence} from the collection of (appropriately transverse) immersed Lagrangians in $M$ to immersed Lagrangians in $N$ (see
Definition~\ref{def:lagr_corr}).  In this manner, the   Lagrangian immersion $r\co\NAT_s\to (P^*)^-\times P^*$ induces correspondence on $P^*$, as follows. (Recall that in a symplectic 2-manifold, a curve is Lagrangian if and only if it is immersed).

\begin{theoremABC}\label{thmB} For any $\delta>0$ small enough, then for all small enough $s>0$, 
the correspondence induced by $r\co \NAT_s\to (P^*)^-\times P^*$ acts on immersed curves in $P^*$ in the following way (see Figure~\ref{example1fig}):
 \begin{itemize}
\item Immersed circles  which stay some $\delta>0$ away from the corners of the pillowcase    are  doubled. That is, an immersion  $L\co S^1\to P^*$ is sent to 
$$S^1\times \{0,1\}\xrightarrow{{\rm proj}_{S^1}}S^1\xrightarrow{L} P^*.$$

\item  Immersed arcs  which are linear in the $\delta>0$ neighborhood of the corners  are replaced by homology figure eight curves which lie in a small tubular neighborhood of the   immersed arc and stay a bounded distance from the corners.  
\end{itemize}

 \begin{figure}[ht]
\centering
\labellist 
\pinlabel $(\NAT_s)_*$ at 266 90
\endlabellist
\includegraphics[scale=0.7]{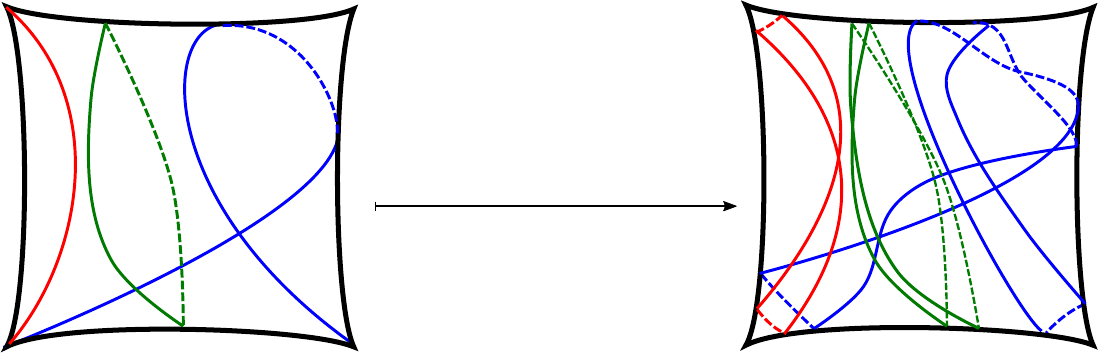}
 \caption{The action of $\NAT_s$, induced by the earring tangle, on immersed curves in the pillowcase, up to a small regular homotopy.}\label{example1fig}  
\end{figure}

\end{theoremABC}
\noindent More careful statements can be found in Theorems~\ref{thm:close},~\ref{looptype}, and Section~\ref{arc}  below. Homology figure eight curves are defined in Definition~\ref{HF8}.  We were unable to establish   the stronger conclusion that arcs are taken to (honest) figure eight curves, but this seems likely and is borne out by several calculations; see Conjecture~\ref{truefig8}.

\medskip

This work is motivated by the Atiyah-Floer conjecture~\cite{MR974342,MR1078014}, which establishes a bridge between (various versions of) Floer's instanton homology  and Lagrangian Floer theory  of suitable flat moduli spaces. There are many versions of this conjecture; some of which  have been established~\cite{DS}, but the one relevant to this paper, corresponding to Kronheimer-Mrowka's singular instanton homology, is still open. Some progress on this conjecture has been made~\cite{salamon2008instanton,lipyanskiy2014gromovuhlenbeck,duncan2019compactness,DFLAtiyahFloer,DFLmixed}, and  an equivariant approach,  suggested in~\cite{ManWood},  is being worked out by Daemi and Fukaya~\cite{Daemi_Fukaya}.

Holonomy perturbations~\cite{Taubes,Flo} provide a general framework, consistent with both (in the sense of the Atiyah-Floer conjecture) gauge theoretic Floer homology and the Lagrangian Floer homology  of representation varieties, to alter the differential equations defining the flat moduli space to produce a {\em holonomy perturbed}  flat moduli space. In favorable circumstances this produces a Lagrangian immersion into the symplectic smooth stratum of the traceless flat moduli space of   the  boundary punctured surface.   We   accomplish this for the earring tangle using two fairly simple perturbation curves, in order to retain the strongest connection between the perturbed flat moduli space and the fundamental group of the original tangle complement, and to allow an explicit description of the perturbed flat moduli space.  In particular, we make use of the  identification of the perturbed traceless flat moduli space with the {\em perturbed traceless character variety} (see Section~\ref{subsubsec:perturbations_prelims}).

  \medskip

The proofs of Theorems~\ref{thmA} and~\ref{thmB} involve a number of reduction steps to   obtain a smooth moduli space and identify the restriction map to $P^*\times P^*$.  The tools used are   fundamental group calculations, the algebraic description of holonomy perturbations, and the geometry of unit quaternions.    
  
\subsection*{Some context/history} 
From the character variety perspective, our results are inspired by  Klassen's investigation of the (unperturbed) $SU(2)$ character variety of the Whitehead link complement, which he uses   to study the flat moduli spaces of Whitehead doubles of knots~\cite{MR1205249}. 
{\em Perturbed flat} character varieties are real semi-analytic varieties and few explicit calculations are known.    Calculations for 2-stranded tangles in the 3-ball, including one arising from a tangle decomposition of the $(4,5)$ torus knot which we explore in 
Section~\ref{sec:PillHom_n_quilts}, can be found in~\cite{HHK2}. A few other known examples are discussed in Sections~\ref{TCV} and~\ref{CLC}.

From an abstract perspective, Theorem~\ref{thmB}
partially computes the value of  a putative Floer field functor, expected to be the Atiyah-Floer counterpart to Kronheimer-Mrowka's singular instanton homology,  on the earring tangle, a simple but useful and nontrivial cobordism from $(S^2,4)$ to itself. 

Roughly speaking, a Floer field functor is a functor from a cobordism category to 
Weinstein's symplectic ``category''~\cite{Weinstein} whose objects are symplectic manifolds and whose
morphisms are Lagrangian correspondences. Two well-known difficulties prevent this from being a category:   lack of transversality and the failure of the composition of two embedded correspondences to be again embedded. The first issue is typically dealt with using small perturbations.   Wehrheim and Woodward introduced the notion of quilts to circumvent the second difficulty.   They defined a category 
whose morphisms are {\em generalized} Lagrangian correspondences, in which composition is always defined. They proved an embedded composition theorem (see Appendix~\ref{Appendix}) which asserts that in special cases, geometric composition of Lagrangian correspondences preserves the Lagrangian Floer homology.

The Wehrheim-Woodward approach,   combined with Cerf theory techniques, leads to the construction of 3-manifold invariants in terms of the Lagrangian Floer homology of character varieties, by studying the character varieties associated to decompositions of 3-manifolds along surfaces  (as well as tangle decompositions of links in 3-manifolds along punctured surfaces). This is a thriving research area, and some interesting invariants have been defined and some of their structural properties worked out~\cite{FF_tangles, FF_coprime, ManWood, Guillem, MR4033514, Henry,Rezazadegan}.

Our focus on the pillowcase is motivated foremost by the fact that it is 2-dimensional, and hence many aspects of its Lagrangian Floer theory,  in particular associated Fukaya categories,  can be studied with differential-topological rather than $J$-holomorphic techniques~\cite{Abouzaid_Fuk_surface, SRS}.  This approach has already borne fruit by revealing a useful Floer-theoretic description of Khovanov homology~\cite{HHHK,KWZ,comparison}.   However, the price we pay is that, even when transversality is achieved (for example, the genus 3 surface in our main theorem), the Lagrangians correspondences associated to 3-manifolds with boundary are immersed rather than embedded.   Immersed Lagrangian Floer theory is an active research area, with foundational work including~\cite{AJ,Fuk_Ymaps}. More recently, Bottman and Wehrheim~\cite{BW} have 
proposed using figure eight bubbling to produce bounding cochains  associated with immersed composition.  Our informal exploration of the $(4,5)$ torus knot in Section~\ref{sec:PillHom_n_quilts} explains how the  strip-shrinking technique used by Wehrheim and Woodward in the proof of embedded composition theorem,  and by Bottman and Wehrheim to produce   figure eight bubbles, manifests itself in the context of character varieties.

\subsection*{Organization} We  outline the contents of the article.  We begin with a preliminary section which provides background on traceless $SU(2)$ moduli spaces and holonomy perturbations in a stripped down form sufficient for the arguments in this article.  In particular we introduce these concepts in the language of $SU(2)$ character varieties.

In Section~\ref{earring_tangle}  the   fundamental group, peripheral subgroups,  and meridians of the strands are computed. These  are used to determine the analytic equations   which cut out $\NAT_s$. Technical arguments involving elimination of variables and partial gauge fixing are carried out in Section~\ref{eliminating}, leading to a tractable description of $\NAT_s$. In particular, Example~\ref{ex:Rzero}  identifies the singular 3-dimensional $\NAT_0$.

In Section~\ref{sec:misses_corners}, we prove a  key lemma, Theorem~\ref{nogotheorem}, which asserts that the image of the restriction $\r_s\co \NAT_{s}\to P\times P$ lies in $P^*\times P^*$ when $s>0$.  In other words, whatever the moduli space $\NAT_s$ is, its restriction to the boundary lands inside the smooth stratum of the singular symplectic moduli space of  the  boundary surface.  This is fortunate, since the space $P\times P$,  although homeomorphic to $S^2\times S^2$, is a singular  variety along a codimension two subset.
Theorem~\ref{nogotheorem} makes possible the use, in  Section~\ref{sec:genus_three}, of elementary differential topology arguments to establish that the perturbed flat traceless moduli space $\NAT_s$ is a smooth surface, which is then identified
with a genus three surface by analyzing its restriction to the flat moduli space of one of the two boundary components.

The proof of Theorem~\ref{thmA} is completed in Sections~\ref{sec:restriction} and~\ref{homotopyclass}    by calculating first the homology class, and  then the $\ep$-homotopy class of the Lagrangian immersion $\NAT_s\to (P^*)^-\times P^*$.
 Theorem~\ref{thmB} is deduced from Theorem~\ref{thmA} in Section~\ref{sec:corr}.

\medskip

The second, less formal part of the article begins with Section~\ref{sec:PillHom_n_quilts}. We outline some consequences of Theorems~\ref{thmA} and~\ref{thmB} in the context of immersed Lagrangian Floer homology,  Floer field theory, the wrapped Fukaya category of the 4 punctured 2-sphere, and Khovanov homology.

Section~\ref{sec:PillHom_n_quilts} focuses on the explicit example of a tangle decomposition of the $(4,5)$ torus knot, and explores how Theorem~\ref{thmA}, combined with recent advances in immersed Floer theory, shed light on the Lagrangian Floer counterpart to Kronheimer-Mrowka's singular instanton  homology $I^\nat(\link)$.  In particular we outline how composition with the earring tangle necessitates the introduction of bounding cochains, and we describe how the approach of  Bottman-Wehrheim~\cite{BW}  gives a recipe, in terms of {\em figure eight bubbles}, which   produces the appropriate bounding cochain associated to composition with the genus 3 surface $\NAT_s$ of Theorem~\ref{thmA}.

In Section~\ref{sec:future_directions} we extrapolate from this example and outline some conjectures and future directions to which our main result points. We address the problem of building a well-defined Lagrangian Floer theory for decompositions of tangles along 4-punctured spheres (pillowcase homology) by incorporating bounding cochains.   We discuss an alternative  {\em bordered}  perspective  which aims to align, in the instanton case,   with the recent geometric interpretations of bordered Heegaard Floer and Khovanov homology  by Hanselman-Rasmussen-Watson~\cite{HRW} and Kotelskiy-Watson-Zibrowius~\cite{KWZ}.

Finally Appendix~\ref{Appendix} reviews, for the convenience of the reader,  the basic notions of  generalized Lagrangian correspondences, quilts, composition of correspondences, curved $A_\infty$ categories and the emergence of bounding cochains in the immersed setting.


\subsection*{Acknowledgments} The authors thank N. Bottman, A. Daemi, K. Fukaya, M. Hedden, L. Jeffrey,  and C. Woodward  for helpful comments. The third author thanks E. Klassen for introducing him to these topics.  We also thank the referee for their careful reading and valuable suggestions.

\section{Preliminaries}\label{prelim}
\subsection{Holonomy perturbed traceless SU(2) character varieties}

\subsubsection{Tangles}
A \emph{tangle} is a pair $(Y,\link)$, consisting of a compact oriented 3-manifold with boundary $Y$ and $\link\subset Y$ a smoothly and properly embedded compact 1-manifold (which may have closed components).  A link $\link$ with $2n$ boundary  points is called an $n$-tangle.   A tangle in the 3-ball is called the {\em trivial}  tangle if it is homeomorphic the inclusion of $n$ vertical segments.  Three views of the trivial 2-tangle are given in the middle of Figure~\ref{fig:basic_lagrangians}.

\begin{wrapfigure}{r}{0.17\textwidth}
\vspace{-0.5cm}
\centering
\includegraphics[width=0.17\textwidth]{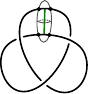}
\caption{}
\label{arc_svk}
\end{wrapfigure} 
A {\em 2-tangle decomposition of a link $(S^3,\link)$} refers to a  decomposition of the link along a 4-punctured 2-sphere $(S^2,4)$:
\begin{equation}\label{SVK} (S^3,\link)=(D^3,T_1) \cup_{(S^2,4)} (D^3,T_2). \end{equation}
Examples of such decompositions can be produced for any link $(S^3,\link)$   by taking $(D^3, T_2)$ to be a neighborhood of any embedded arc which meets $\link$ (transversely) only at its endpoints, and setting $(D^3, T_1)$ to be the complementary tangle. A 2-tangle decomposition of  the trefoil knot is depicted in Figure~\ref{arc_svk}.

More generally, one can view a tangle $(Y,T)$ as a morphism in a $(2,0)+1$ dimensional cobordism category whose objects are pairs $(F,k)$ consisting of an oriented surface $S$ and a choice of $k$ marked points on $S$, for example, $(S^2,4)$. A morphism  from $(F^\rout,k^\rout)$ to $(F^\rin,k^\rin)$ is a \emph{tangle cobordism}, i.e., a tangle $(Y,T)$ such that $\partial (Y,T)=(-F^\rout,k^\rout)\cup (F^\rin,k^\rin)$. For instance, a morphism from the empty pair $(\emptyset,\emptyset)$ to itself is a link in a closed 3-manifold. The cobordism category might be enhanced with auxiliary data, for example, a choice of base point on the link, or the addition of $w_2$ arcs (as explained  in Section~\ref{w2}).

\subsubsection{Traceless-SU(2)-character-variety functor}

The term {\em $SU(2)$ character variety} refers to  the contravariant functor which takes a discrete  group $\Gamma$ to the space of conjugacy classes of $SU(2)$ representations of $\Gamma$. 
Taking $\Gamma$ to be the fundamental group defines the character variety functor on the category of  based topological spaces.   
This notion extends to codimension two pairs of manifolds via traceless representations, as we next explain.

\medskip

Identify $SU(2)$ with the group of unit quaternions, and its Lie algebra $su(2)$ with the vector space of purely imaginary quaternions.  Equip $su(2)$ with the positive definite inner product $-\Real(vw)$.    Let $\Ima:SU(2)\to su(2)$ denote the projection to the imaginary part: $\Ima(g)=g-\Real(g)$. Denote the 2-sphere of purely imaginary unit quaternions by $\tracelessTwoSphere$.  Equivalently,
 \begin{equation} \label{Sigma}
 \begin{split}
 \tracelessTwoSphere&= SU(2)\cap su(2)\\
 &=  \{g\in SU(2)~|~ \Real(g)=0\}    \\
 &=  \{v\in su(2)~|~ \|v\|^2=1\}    \\
 &=   \{g\in SU(2)~|~ g^2=-1\}      \\
 &=  \{ g\in SU(2) ~|~ g \mbox{ is conjugate to  $\bbi$}\}
 \end{split}
 \end{equation}
 
 \begin{definition}\label{traceless} Given a discrete group $\groupPrelims$ and a subset  $M \subset \groupPrelims$, define the {\em set of traceless representations of $\groupPrelims$} to be
$$\widetilde{R}(\groupPrelims,M) \coloneqq \{ \rho: \groupPrelims\to SU(2) ~|~ \rho(m)\in \tracelessTwoSphere \text{ for all } m\in M \}.$$   
The group $SU(2)$ acts via pointwise conjugation on $\widetilde{R}(\groupPrelims,M) $.  Define the {\em traceless character variety of $(\groupPrelims,M)$} to be the orbit space, 
$$R(\groupPrelims,M) \coloneqq \widetilde{R}(\groupPrelims,M) / SU(2).$$
Given a codimension two pair of manifolds $(Y,\link)$, define its {\em traceless character variety} by 
$$R(Y,\link)\coloneqq R(\pi_1(Y\setminus \link), M),$$
where $M$ is the set  of {\em meridians} of $\link$  in $\pi_1(Y\setminus \link)$. 
 \end{definition}
A set of $n$ generators for a group $\groupPrelims$ defines an embedding 
$$\widetilde{R}(\groupPrelims,M) \subset \Hom(\groupPrelims,SU(2))\subset SU(2)^n,$$ 
by evaluating a homomorphism on the generators. This exhibits $\widetilde{R}(\groupPrelims,M)$ as a topological space as well as an affine real-algebraic set.    A different choice of generators determines an isomorphic algebraic set.   As the  image of an algebraic set under an algebraic map, the orbit space of conjugation action ${R}(\groupPrelims,M)$ is a real semi-algebraic set
\cite[Theorem 2.2.1]{MR1659509}.

Clearly $R$ is functorial\footnote{Note that when applying the fundamental group functor, the surfaces should carry a basepoint, and tangles should contain a distinguished arc connecting those basepoints. This well-known subtlety is suppressed in this article, as our main calculation uses the Wirtinger presentation, which implicitly includes choices of distinguished basepoints and arcs.} 
with respect to inclusion of boundary of tangles $$\partial (Y,\link) \hookrightarrow (Y,\link)  \implies R(\partial (Y,\link)) \leftarrow R(Y,\link)$$
  Roughly speaking, the functor $R$ takes 2-manifolds to  symplectic manifolds, and tangles to  Lagrangian submanifolds.  More specifically, the work of Atiyah-Bott and Goldman
\cite{Atiyah_Bott,Goldman} shows that for a {\em closed} surface $S$, the top smooth stratum of $R(S)=R(S,\emptyset)$ admits a symplectic structure which corresponds, via Weil's  identification of the (Zariski) tangent space $T_\rho R(S )$ with $H^1(S;su(2)_{\text{ad}   \rho})$, to the intersection pairing. Poincar\'e-Lefschetz duality shows that if a  compact 3-manifold $Y$ has boundary the closed surface $S$, then the differential of the restriction map $  R(Y ) \to R(S) $ takes each tangent space $T_\rho R(Y)$ to a Lagrangian subspace of $T_\rho R(S)$.   With care, these two facts extend to  punctured surfaces $(S,F)$ and tangles $(Y,\link)$, as one can see from many different points of view  (\cite{MR1078014, MR1460627}).
One approach is to use a symplectic reduction argument to pass from the Lagrangian immersion  ($U$ denotes a small tubular neighborhood)
  $$R(Y\setminus U(\link))\to R(\partial(Y\setminus U(\link)))$$ to its symplectic quotient 
  $$R(Y,\link)=R(Y\setminus U(\link)) ~{/\mkern-6mu/} ~T^n \to R(\partial (Y,\link))=R(\partial(Y\setminus U(\link)))  ~{/\mkern-6mu/} ~ T^n$$ 
  using a Hamiltonian torus action defined in terms of the meridians of the tangle components.  The details for $\partial Y=S^2$, can be found in Theorem 4.4 and Corollary 4.5 of~\cite{HK}.  An alternative approach, which applies to any tangle in a compact 3-manifold with boundary, can be found in~\cite{CHK}.

\subsubsection{Character varieties and flat moduli spaces} 
The assignment of  holonomy to a smooth connection in a principal $G$ bundle $P$ over a manifold $X$ descends to a homeomorphism from  the moduli space of flat connections in $P\to X$  to a union of path components of the $G$ character variety of $X$.   In particular, the holonomy determines  a homeomorphism (and even analytic isomorphism, properly interpreted) between Kronheimer-Mrowka's moduli space of singular flat connections on $(Y,\link)$ and the corresponding traceless character variety $R(Y,\link)$, as well as variants defined using nontrivial $SO(3)$ bundles, touched upon in Section \ref{w2}  (see~\cite{KM_sing_connections} and  \cite[Section~7]{HHK1}).

In brief \begin{equation*}
\text{hol} \colon \cM(Y,\link ) \cong  R(Y,\link).
\end{equation*}
We do not make use of the details of this identification in this article beyond the fact that smoothness  of the perturbed flat moduli space at a connection is equivalent to smoothness of the perturbed character variety at its corresponding holonomy,  a consequence of the Kuranishi method~\cite{MR0176496,MR1883043}. We henceforth work  entirely with character varieties rather than flat moduli spaces.

\subsubsection{Holonomy perturbed representations} \label{subsubsec:perturbations_prelims}

\begin{definition}\label{def:pert}
 Given a group $\groupPrelims$,  a pair  of    elements $\pertCurveInPreliminaries , \lambda_\pertCurveInPreliminaries\in \groupPrelims$, and a real number $s$,   a representation  $\rho:\groupPrelims\to SU(2)$ is said to {\em satisfy the perturbation condition for the pair $(\pertCurveInPreliminaries , \lambda_\pertCurveInPreliminaries)$} provided 
 $$\rho(\pertCurveInPreliminaries )=e^{s\Ima(\rho(\lambda_\pertCurveInPreliminaries))}.$$ 
 More generally, given a finite list of pairs $(\pertCurveInPreliminaries_1,\lambda_{\pertCurveInPreliminaries_1}),\ldots, (\pertCurveInPreliminaries_n,\lambda_{\pertCurveInPreliminaries_n})$,  we say $\rho$ satisfies the perturbation condition if $\rho$ satisfies it for each pair in the list. 
\end{definition}

Note that the perturbation condition is preserved by conjugation, and implies that  $\rho(\pertCurveInPreliminaries )$ and $\rho(\lambda_\pertCurveInPreliminaries)$ commute. Also, when $s=0$, a representation  satisfies the perturbation condition with respect to to pairs $(\pertCurveInPreliminaries_1,\lambda_{\pertCurveInPreliminaries_1}), \ldots, (\pertCurveInPreliminaries_n,\lambda_{\pertCurveInPreliminaries_n})$ if and only if it sends each $\pertCurveInPreliminaries_i$ to the identity. 

When $(Y,\link)$ is a 3-manifold containing a tangle $\link$, such pairs $(\pertCurveInPreliminaries_i ,\lambda_{\pertCurveInPreliminaries_i})$ are produced starting with an embedding of a disjoint union of solid tori $\pertSolidToriInPreliminaries_i:S^1\times D^2\to Y\setminus \link$.   Set $\pertCurveInPreliminaries_i=\pertSolidToriInPreliminaries_i(\{1\}\times\partial D^2)$  and $\lambda_{\pertCurveInPreliminaries_i}=\pertSolidToriInPreliminaries_i(S^1\times\{1\})$ in $\pi_1(Y\setminus(\link\cup_i   \pertSolidToriInPreliminaries_i(S^1\times D^2))$. 
 The choice of $\pertSolidToriInPreliminaries_i$ and $s\in \RR$ is called {\em perturbation data}, and denoted here by $\pi$. The   real semi-analytic set of conjugacy classes of holonomy perturbed traceless representations of $(Y,\link)$ is denoted by $R_\pi(Y,\link)$. When $Y=D^3$ or $S^3$, we abbreviate this to $R_\pi(\link)$ for convenience.

\subsubsection{The second Stiefel-Whitney class and $w_2$-arcs}\label{w2}

\begin{wrapfigure}{r}{0.17\textwidth}
\vspace{-0.5cm}
\centering
\includegraphics[width=0.17\textwidth]{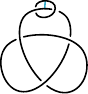}
\caption{}
\label{fig:earring}
\vspace{-0.3cm}
\end{wrapfigure}

Given a tangle $(Y,\link)$ and a point $p\in \link$, an \emph{earring} is a meridian $E$ around the point $p$, together with an arc $W$ connecting $p$ and $\link$. Figure~\ref{fig:earring} illustrates the trefoil knot with an earring placed on it. In~\cite{KM}, the earring determines a nontrivial $SO(3)$ bundle on which the gauge theory takes place. We choose to work in $SU(2)$ rather than $SO(3)$, as the notation and geometry are more familiar. 
To this end we introduce the following notion  (an alternative approach using projective representations is explained in~\cite{PraSav}).  

A {\em set of $w_2$-arcs for a tangle $(Y,\link)$} is a smoothly embedded  disjoint union of arcs  $W$ in the interior of $Y$,  whose interior misses $\link$ and whose boundary points map to $\link$, transversely.   
\begin{definition}
Given a tangle $(Y,\link)$ and a set of $w_2$-arcs W, define the \emph{traceless character variety of $(Y,\link)$ satisfying the $w_2$-condition} to be 
$$
R(Y,\link,W)= \left\{
\substack{  \displaystyle
\rho: \pi_1(Y\setminus (\link \cup W))\to SU(2) ~|~ \rho(m_i)\in \tracelessTwoSphere,~ \rho(w_j)=-1, \\\displaystyle
\text{ where } m_i \text{ are meridians of }\link \text{ and } w_j \text{ are meridians of }W  
}
\right\} \Big/ \text{conjugation}
$$
Given the necessary data, the perturbed version $R_\pi(Y,\link,W)$ is defined according to Definition~\ref{def:pert}.
\end{definition}

Suppose a 2-tangle $(Y,\link)$ is equipped with a basepoint.   Place an earring $(E,W)$, and denote the traceless character variety of $(Y,\link \cup E)$ satisfying the $w_2$-condition with respect to $W$ by
\begin{equation}\label{defnat} R^\nat(Y,\link) \coloneqq R(Y,\link \cup E, W).\end{equation}
For example, the space $R^\nat(S^3,U)$ for an unknot is a single regular point~\cite[Proposition~4.1]{KM}.

\subsection{The 4-punctured sphere  and the pillowcase}

\subsubsection{The pillowcase}  \label{pillsect}
\begin{definition}\label{defnofpillow}
The {\em pillowcase} $P$ is the quotient of the torus $T$ by the {\em elliptic involution}, which we denote by $\iota:T\to T$: 
$$P \coloneqq  T/\iota.$$
\end{definition}  

In the coordinates $T=(\RR/2\pi \ZZ)^2$, the elliptic involution $\iota$ is given by 
\begin{equation}\label{ei}\iota(\gamma,\theta) = (-\gamma,-\theta).\end{equation} 
There are four fixed points, which we denote by 
$$C=\{(0,0), (0,\pi), (\pi, 0), (\pi,\pi)\},$$ 
and the quotient map $T\to P$ is a 2-fold branched cover, with the branch set $C$. We write $[\gamma,\theta] \in P$ for the equivalence class of $(\gamma,\theta) \in T$.  The fundamental domain is given by $\{(\gamma,\theta)\in [0,\pi]\times [0, 2\pi]\}$ (the middle of Figure~\ref{fig:pillowcase_coordinates}), and $P$ is homeomorphic to the 2-sphere (on the right in Figure~\ref{fig:pillowcase_coordinates}).  
We abuse notation slightly and identify $C$ with its image in $P$. When a distinction is needed, we refer to them as the {\em fixed points} in $T$ and the {\em corners} in $P$. Let 
$$P^* \coloneqq P\setminus C.$$

The pillowcase arises as the traceless character variety $R\FPS$ of the 4-punctured 2-sphere. Take as meridians  the four loops $a,b,e,f$ enclosing the punctures, oriented so that $\pi_1(S^2\setminus 4)=\langle a,b,e,f~|~ ba=ef\rangle,$ as illustrated on the left of Figure~\ref{fig:pillowcase_coordinates}. 

\begin{lemma}[\text{\cite[Proposition~3.1]{HHK1}}]\label{pillowcase} The traceless character variety of the pair $(S^2,4)$ is isomorphic to the pillowcase.  In fact, the composite of the map 
\begin{align*}
T=(\RR/2\pi\ZZ)^2 &\to \Hom(\pi_1(S^2\setminus 4),SU(2)) \\
(\gamma,\theta)& \mapsto \big(a\mapsto \bbi, ~b\mapsto e^{\gamma\bbk}\bbi,~ f\mapsto e^{\theta\bbk}\bbi\big)
\end{align*}
and the projection to conjugacy classes induces an isomorphism from the pillowcase $P=T/\iota$ to the traceless character variety $R \FPS$. \qed
\end{lemma} 
\begin{figure}[ht]
\labellist 
\pinlabel $a$ at 200 180
\pinlabel $b$ at 85 200
\pinlabel $f$ at 205 100
\pinlabel $e$ at 85 70
\pinlabel $\gamma$ at 523 0
\pinlabel $\theta$ at 375 280
\pinlabel {\scriptsize$ 2\pi$} at 360 253
\pinlabel {\scriptsize$ \pi$} at 366 130
\pinlabel {\scriptsize$0$} at 370 0
\pinlabel {\scriptsize$\pi$} at 497 -7
\pinlabel $R(-)$ at 298 150
\pinlabel $(S^2,4)$ at 110 0
\pinlabel $\text{The pillowcase }P$ at 665 40
\endlabellist
\centering
\includegraphics[width=0.8\textwidth]{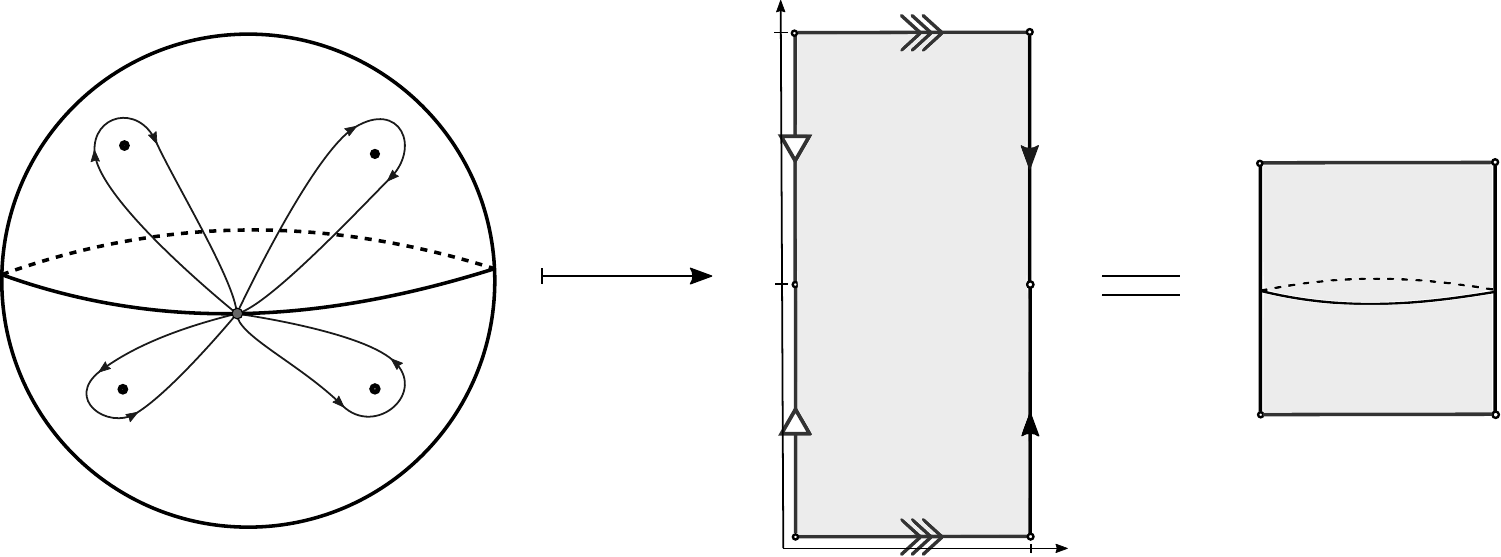}
\caption{Pillowcase coordinates used in this article: $(\gamma,\theta)\in \RR^2$ determines $\rho\in R(S^2,4)$ by $\rho(a)=\bbi, \rho(b)= e^{\gamma\bbk}\bbi, ~ \rho(f)= e^{\theta\bbk}\bbi $.}
\label{fig:pillowcase_coordinates}
\end{figure}
We call the pair $[\gamma,\theta]\in (\RR/2\pi\ZZ)^2/\iota$   {\em pillowcase coordinates} for a point in $P$. 
Endow the  top stratum $P^*$  with the Atiyah-Bott-Goldman symplectic form \cite{MR1460627, HK}. 
\begin{proposition}  The symplectic structure on $P^*$ is a constant multiple of the  one that descends from the standard symplectic structure $d\gamma \wedge d\theta$ on the torus. 
\end{proposition} 
See  \cite{CHK} for a proof.

\subsubsection{Traceless character varieties of 2-tangles as objects in the 
Fukaya category of the pillowcase}\label{TCV}
 
We single out two subfamilies of connected immersed curves in the pillowcase.  Fix a small $\delta>0$.
\begin{definition}\label{loops/arcs}
An immersed circle   $L:S^1\to P^*$  is called {\em $\delta$-loop-type} if the curve is an immersed circle  which misses the neighborhood $U_\delta(C)$.
The set of all $\delta$-loop-type curves is denoted by  $\allLoopImmersions{\delta}$.

A continuous map of an interval $L:[0,1]\to P$ is called {\em  $\delta$-arc-type} if   
\begin{enumerate}
\item $L(0), L(1)\in C$, and  $L(t)\not \in C$ for $0<t<1$,
\item there exist $0<a<b<1$ so that the restriction of $L$ to $[a,b]$ maps to $P\setminus U_\delta(C)$ and the restrictions of $L$ to $[0,a)$ and $(b,1]$ are embeddings into $U_\delta(C)$, and
\item $L$ admits a smooth immersed lift $\tilde L:[0,1]\to T$ (in particular $\tilde L'(0)\ne 0$ and $\tilde L'(1)\ne 0$).
\end{enumerate}
The set of all $\delta$-arc-type curves is denoted by
 $
 \allArcImmersions{\delta}.$

Call a curve {\em loop-type} (resp.~ {\em arc-type}) if there exists a $\delta>0$ so that the curve is $\delta$-loop-type (resp.~ $\delta$-arc-type). We remark that this terminology differs from the one in~\cite{HRW}, where loop-type curves are compact curves with the trivial local system.
 \end{definition}
The arcs $A_0,A_1,A_2$ in Figure~\ref{pillowgensfig}  lie in $\allArcImmersions{\delta}$ for any $\delta$.

\medskip

The importance of $\delta$-arc and $\delta$-loop-type curves for us stems from the following theorem, proved in~\cite{HK}.
\begin{theorem} \label{thm:immersed_curve_invariant}  Let $Y$ be a homology 3-ball, and $(Y,T)$ be a tangle with boundary $\FPS$. Then, for any sufficiently small $\delta>0$,   there exists a collection of disjoint embedded solid tori $N_i:D^2\times S^1 \to Y\setminus T$ and arbitrarily small holonomy perturbation data $\pi$ using these perturbation curves such that the perturbed traceless character variety $R_\pi(Y,T)$ and its restriction to the pillowcase is a finite union of $\delta$-arc-type and $\delta$-loop-type immersed Lagrangians.  Endpoints of arc components correspond bijectively to reducible traceless representations. 
\end{theorem}

To put these in a broader context, the loop-type curves represent objects in the Fukaya $A_\infty$ category $\Fuk(P^*)$ of immersed compact Lagrangians in $P^*$, and the arc-type   curves define objects in the (larger) wrapped Fukaya $A_\infty$ category $\wrFuk(P^*)$   of immersed proper Lagrangians in $P^*$.  (These notions are reviewed in Section \ref{ssec:fukayacat}.)

\medskip

Descriptions of $R_\pi(D^3,T)$ and $R^\nat_\pi(D^3,T)$  are known for various 2-tangles.  
When $T$ is a trivial 2-tangle, $R(D^3,T)$ is a linear embedded arc joining 2 corners of the pillowcase; see Lemma~\ref{lem8.2} and \cite[Section~10]{HHK1}. For example, in our pillowcase coordinates, the character varieties for three trivial tangles are depicted on the left side of Figure~\ref{fig:basic_lagrangians}. The trivial tangles are homeomorphic, but as   \emph{bordered} tangles have different image in the character variety of the boundary.   
Theorem 7.1 of~\cite{HHK1} identifies $R^\nat_\pi(D^3,T)$ for the trivial 2-tangle as the corresponding figure eight curve in $P^*$; see the right part of Figure~\ref{fig:basic_lagrangians}. 
 The variety $R(D^3,T)$ is identified for tangles  coming from 2-tangle decompositions of torus knots, holonomy perturbed along one of the exceptional fibers in \cite{HHK2,FKP}. This includes one  particularly interesting example, explored in detail in Section~\ref{sec:fig8bubble}, associated to a 2-tangle decomposition of the $(4,5)$ torus knot.

\begin{example}\label{example3} 
Figure~\ref{A2fig} shows a trivial tangle $T$ in the 3-ball equipped with one $w_2$-arc, labeled $W$.  
The four meridians $a,b,e,f$ of $T$, together with the meridian $w$ of $W$, generate 
$\pi_1(D^3\setminus(T\cup W))$. 
The space $R(D^3,T,W)$ (conjugacy classes of traceless representations satisfying the $w_2$ conditions) is an arc. In fact, an explicit slice of the conjugation action is given by the path of representations
$$
[0,\pi]\ni \theta\mapsto  \big(a\mapsto \bbi, ~b\mapsto \bbi,~ f\mapsto e^{\theta\bbk}\bbi , w\mapsto-1  \big)
$$
which restricts to the linear arc $\gamma=0$ in the pillowcase, labeled as $A_2$ in Figure~\ref{pillowgensfig}.  
\end{example}

\begin{figure}[ht]
  \centering
  \begin{subfigure}{0.3\textwidth}
    \labellist 
    \pinlabel $\color{blue}W$ at 27 28
    \pinlabel $w$ at 31 22
    \pinlabel $a$ at 37 33
    \pinlabel $b$ at 13 33
    \pinlabel $e$ at 13 12
    \pinlabel $f$ at 38 13
    \endlabellist
    \centering
    \includegraphics[width=0.7\textwidth]{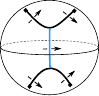} 
    \caption{}\label{A2fig}
  \end{subfigure}
  \begin{subfigure}{0.3\textwidth}
    \centering
    \includegraphics[width=0.7\textwidth]{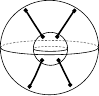} 
    \caption{}\label{fig:product_cobordism}
  \end{subfigure}
  \begin{subfigure}{0.3\textwidth}
    \centering
    \includegraphics[width=0.7\textwidth]{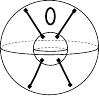} 
    \caption{}\label{fig:prod_circ}
  \end{subfigure}
  \caption{}
\end{figure}

\subsubsection{Cobordisms and the associated Lagrangian correspondences}\label{CLC}
 
The assignment of $SU(2)$ traceless character varieties $R(-)$ can be viewed as a first approximation of a \emph{Floer field theory}.    There are several ingredients which go into the setup of Floer field theory, and we review the main ideas in Section~\ref{sec:fft} and Appendix~\ref{Appendix}. 
The notions of Lagrangian correspondences and their composition 
(see Definitions~\ref{def:lagr_corr} and~\ref{def:compos})  play a key role in this article, starting from the following paragraph.

We give a few examples of Lagrangian correspondences induced by tangle cobordisms. 
The first example is provided by the product $(S^4, 4)\times [0,1]$, illustrated in Figure~\ref{fig:product_cobordism}. Since the inclusion of each boundary component induces a meridian preserving isomorphism, its traceless character variety serves as the identity (diagonal) correspondence $$P \xrightarrow{\triangle} P^- \times P.$$

A more complicated example, which is analyzed using explicit holonomy perturbations, is the correspondence obtained by adding a single split unknotted circle to the product cobordism, as depicted in Figure~\ref{fig:prod_circ}. The associated immersed Lagrangian correspondence is a trivial 2-fold cover of the diagonal~\cite[Corollary B.2]{HHHK}:
$$P\times S^0\xrightarrow{{\rm proj}_P}P\xrightarrow{\triangle} P^- \times P.$$
This induces the doubling correspondence on all Lagrangian immersions, which replaces any loop-type or  arc-type curve  by two copies of itself. From an algebraic point of view   (and taking gradings into consideration),  this correspondence takes the tensor product with the cohomology algebra $H^*(S^2)$ in the additive enlargement of the wrapped Fukaya category of the pillowcase. 

This calculation, combined with the fact (illustrated in Figures~\ref{fig:basic_lagrangians} and \ref{corrfig}) that the character varieties of the unoriented skein triple form an exact triple in the wrapped Fukaya category of the pillowcase~\cite[Lemma~5.4]{Abouzaid_Fuk_surface},  are the key ingredients behind the relationship between pillowcase homology and Khovanov homology described in~\cite{HHHK} and expanded and clarified in~\cite{KWZ}.

\begin{figure}[ht] 
\labellist
\pinlabel $R_\pi(-)$ at 550 220
\pinlabel $R^\nat_\pi(-)$ at 1210 220
\endlabellist
\centering
\includegraphics[width=0.9\textwidth]{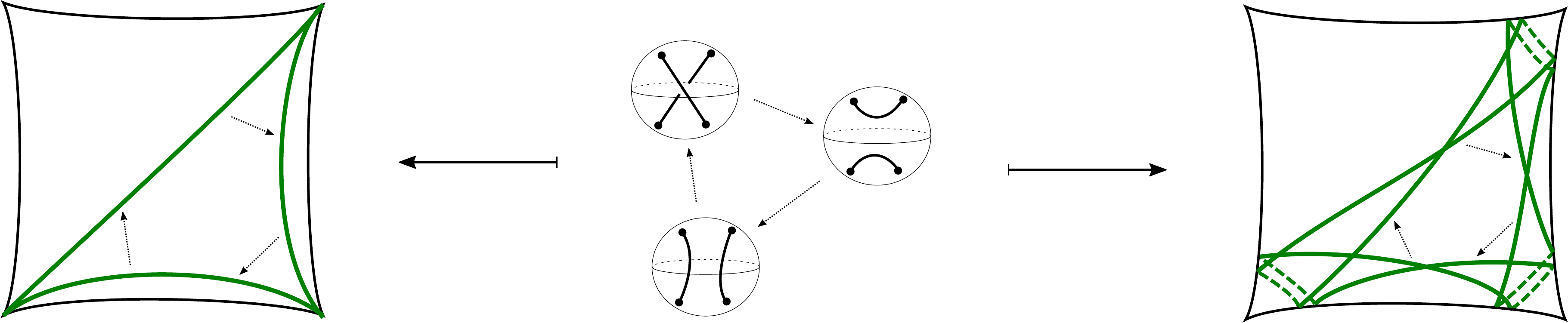}
\caption{Given a skein triple of tangles, both $R(-)$ and $R_\pi^\natural(-)$  produce an exact triple of Lagrangians\label{fig:basic_lagrangians}. (The pillowcase coordinates from Figure~\ref{fig:pillowcase_coordinates} are used.) }
\end{figure}

\medskip


 \section{The earring tangle}\label{earring_tangle}

 \begin{figure}[ht]
\begin{center}
\def\svgwidth{6in}
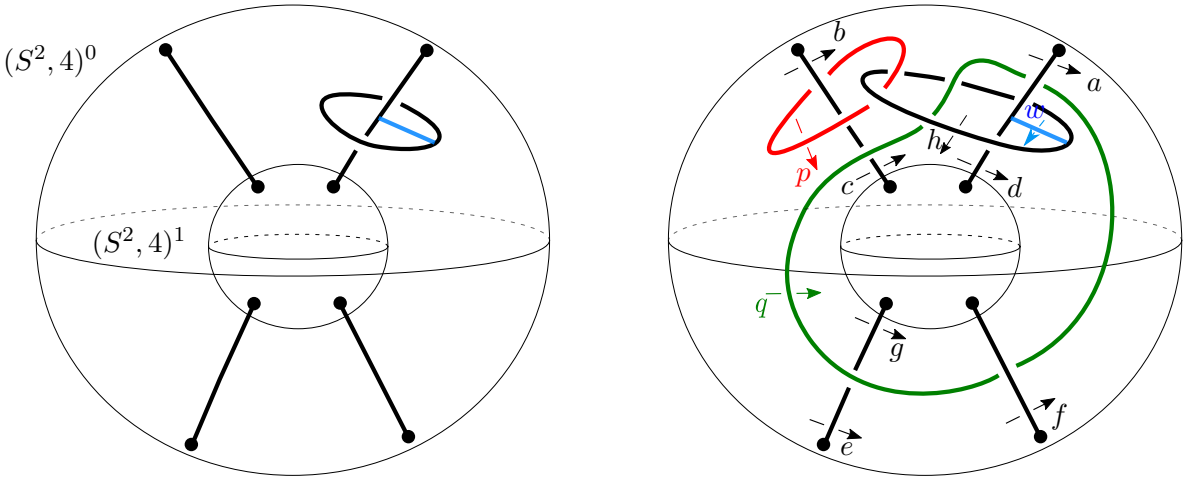
 \caption{The earring tangle cobordism on left. The right figure indicates two perturbation curves and labels the generators of the fundamental group.\label{cobordism3fig}  }
\end{center}
\end{figure}

Figure~\ref{cobordism3fig} on the left depicts the central character in this paper: the 3-manifold $Y=S^2\times I$ containing the four black arcs $\link$ and one circle $E$ which we call the {\em earring}, as described above. The pair $(Y,\link\cup E)$ has two boundary components $\partial(Y,\link\cup E)=\FPS^\rout \sqcup \FPS^\rin$. On the right side selected meridians of $\link$ are labeled by $a,b,c,d,e,f,g$. One of meridians of $E$ is denoted by $h$. A red and green pair of curves, with meridians labeled $p$ and $q$, are also shown. Denote   $Q_p,~Q_q:S^1\times D^2\to Y$  0-framed embeddings of the solid torus into a tubular neighborhood of these two curves. These are the {\em perturbation solid tori} along which we perform holonomy perturbations, as described in Section \ref{subsubsec:perturbations_prelims} above. We use the less precise terminology {\em perturbation curves}, with the implicit understanding that these are thickened to solid tori.  They give part of the data needed to deform flat moduli spaces, and hence play a  different role than the tangle $\link \cup E$.  
The  blue arc labeled $w$ is a $w_2$-arc, as defined in Section~\ref{w2}.

 Several considerations lead to this particular choice of perturbation curves.  To begin with,  they intersect any incompressible annulus, so the associated perturbed character varieties do not have   the same bending symmetry as the unperturbed character variety.  From a gauge theoretic point of view, the trace of the holonomy along each curve breaks the gluing parameter gauge symmetry of the Chern-Simons functional.  
A more subtle consideration is that  the {\em unperturbed} restriction to the boundary has the structure of a circle  bundle (away from the reducibles, see Example~\ref{ex:Rzero} below), and  the sum of the traces of the holonomy along the  two perturbation curves restricts to a Morse function with 2 critical points on each fiber of that bundle.    This is  a ``families'' analogue, for a 3-manifold with boundary,  of the process of choosing a perturbation function whose restriction to each Bott-Morse critical submanifold of the Chern-Simons functional to be a Morse function on that submanifold    (see, for example, \cite{bh}).  This choice of   perturbation curves was the simplest we found with the required properties.

\subsection{Fundamental group calculation}

The  Seifert-Van Kampen theorem and Wirtinger  presentation approach gives the following. Throughout this article, $\bar{g}$    indicates the inverse of a group element  $g$.

\begin{proposition}\label{wirtpato} With the notation $\lambda_p=bh$ and $\lambda_q=f\ba h$, a presentation of the fundamental group  of the complement 
$$\Pi\coloneqq  \pi_1\big((S^2\times I)\setminus (\link\cup W \cup E\cup Q_p\cup Q_q)\big)
$$ is given as follows. 
\begin{description}
  \item[Generators] The set 
$\{w, a,b, e, f,p,q, h\}$ generates the group, and, in terms of this generating set, we have  
$$c=\bq\bp b p q, ~ d=\bh a w h,~g= \bq e q.$$

\item[Relations] The following list of relations, divided into three types,  is complete. 
\begin{itemize}
\item[(I)]  $[p,\lambda_p]=1, ~ [q,\lambda_q]=1,$ 
\item[(II)] 
$ba=ef$ and $cd =gf,$
\item[(III)]
$[\bff \bq f \ba \bb \bp b,h]=w.$
 \end{itemize}
 Moreover, the relation $cd=gf$ is a consequence of the other relations.
\end{description}  
\qed
\end{proposition}
\noindent The last statement reflects the well-known fact that the last relation in a Wirtinger presentation is redundant.

 \subsection{The peripheral subgroups}
The inclusions of the boundary components $$ \FPS^i\subset (S^2\times I,(\link\cup E\cup W\cup Q_p\cup Q_q)),~ i=\rout,\rin$$ induce peripheral subgroups  which we denote by $\Pi^\rout, \Pi^\rin \subset \Pi$.  The  notation corresponds to Figure~\ref{cobordism3fig}, namely,  $\Pi^\rout$ denotes the fundamental group of the  outer  4-punctured 2-sphere    and $\Pi^\rin$ that of the  inner  punctured 2-sphere. Choose appropriate basepoints on the spheres, and a path between them, so that 
$\Pi^\rout$ and $\Pi^\rin$ have presentations
$$\Pi^\rout=\langle a,b,e,f~|~ ba=ef\rangle, \qquad \Pi^\rin=\langle c,d,g,f~|~ cd=gf\rangle,$$ 
and the inclusion homomorphisms $\Pi^\rout \to \Pi, ~ \Pi^\rin \to \Pi$ preserve the labeling of the generators.

Equip the    peripheral subgroup
$ \Pi^\rout$ with the four (geometric) meridians $M=\{a,b,e,f\}$ (see Definition~\ref{traceless}).  Its traceless character variety is then the pillowcase $P$, as explained in Lemma \ref{pillowcase}.  The   subgroup $ \Pi^\rin$ is equipped similarly, and its traceless character variety is likewise a pillowcase.

\subsection{\texorpdfstring{Holonomy perturbed traceless representations of the earring tangle satisfying the $w_2$ condition}{Holonomy perturbed traceless representations of the earring tangle satisfying the w-2 condition}}

The group $\Pi$ computed in Proposition~\ref{wirtpato}  is generated set $\{w,a,b,e,f,p,q,h\}$.   We take as meridians the subset 
$$M=\{a,b,c,d,e,f,h,g\}$$ of $\Pi$. The blue arc, with meridian $w$, is the only $w_2$-arc.

For perturbations we  use tubular neighborhoods (with the zero framing) of the green and red curves in Figure~\ref{cobordism3fig}. Formally, the perturbation data is given by the two pairs:
$$(p, \lambda_p)=(p, bh) \text { and } (q, \lambda_q)=(q, f\ba h).$$
 and a real parameter $s$.   (One could more generally use a pair of distinct real parameters $s_p$ and $s_q$, but this is unnecessary for the present article.)

\begin{definition}\label{ptrep}  A {\em perturbed traceless $SU(2)$ representation of $\Pi$  and perturbation parameter $s\in \RR$  satisfying the $w_2$ condition}  
 is a homomorphism $\rho:\Pi\to SU(2)$ which satisfies
\begin{align*}
&\text{(1)} &&\rho(w)=-1&& \text{($w_2$ condition)}\\
&\text{(2)} &&0=\Real(\rho(a))=\Real(\rho(b))=\Real(\rho(e))=\Real(\rho(f))=\Real(\rho(h))&& \text{(traceless meridians)}\\
&\text{(3)} &&\rho(p)=e^{s \Ima(\rho(\lambda_p))},~\rho(q)=e^{s \Ima(\rho(\lambda_q))}&& \\
&&&\text{\scriptsize(where $\lambda_p=bh,~ \lambda_q=f\ba h$)} && \text{(perturbation condition)}
\end{align*}
\end{definition}

The group $\Pi$ is generated by $\{ b,f,a,h,w, p, q, e\}$, and therefore evaluation on this list determines an embedding 
$$\Hom(\Pi,SU(2))\subset SU(2)^8$$
 as an algebraic set.  The subset of perturbed traceless representations is a real {\em analytic} subset, because the perturbation condition is not an algebraic but rather an analytic restriction.

\begin{definition}\label{defofR}
Denote the real-analytic space of  perturbed  traceless representations of $\Pi$ satisfying $\rho(w)=-1$ by 
$$\tNAT_s \coloneqq \widetilde{R}_s(S^2 \times I, \link \cup E, W).$$ 
We have $\tNAT_s\subset \Hom(\Pi,SU(2))\subset SU(2)^8.$ Denote its quotient 
by the $SU(2)$ conjugation action $\NAT_s$, i.e.
$$\NAT_s \coloneqq R_s(S^2 \times I, \link \cup E, W).$$
\end{definition}

\section{Elimination of variables and partial gauge fixing}\label{eliminating}
The analytic space $\NAT_s$ is defined above as  the orbit space of an $SU(2)$ action on an analytic subspace $\tNAT_s$ of the 24-dimensional manifold $SU(2)^8$.  In this section we carry out some  preliminary steps to give a    more tractable description  of $\NAT_s$.

\subsection{Elimination of variables}
Every representation  $\rho\in \tNAT_s$  sends the meridians $b,f,h, $ and $a$ into the 2-sphere $\tracelessTwoSphere$. Hence the statement of the following lemma makes sense.

\begin{lemma}\label{4vars} Evaluation on the generators $\{ b,f,a,h\}$ determines an analytic embedding $$\tNAT_s\hookrightarrow  \tracelessTwoSphere ^4\subset SU(2)^4.$$ 
The image can be described as follows:
$$
\{(b,f,h,a)\in \tracelessTwoSphere^4~|~ \Real(b a \bff)=0,~\Real(\bp a \bff q f)=0 \text { and } \Real(h\bp a \bff q f)=0\},
$$ where $p=e^{s\Ima(bh)}, ~ q=e^{s\Ima(f\ba h)}.$
\end{lemma}
 
\begin{proof}

By definition, every $\rho\in \tNAT_s$  satisfies 
$$\rho(w)=-1, \quad \rho(e)=\rho(ba\bff), \quad \rho(p)=e^{s\Ima(\rho(bh))}, \quad \rho(q)=e^{s\Ima(\rho(f\ba h))}.$$ 
This implies that evaluation on $\{b,f,a,h\}$ embeds $\tNAT_s$.

It remains to identify the image of the embedding $\tNAT_s\hookrightarrow \tracelessTwoSphere^4$, which  consists of all 4-tuples $(b,f,a,h)$ which satisfy the list of relations in Proposition~\ref{wirtpato}.  We consider each of these relations   below.  When convenient, we identify $\rho\in \tNAT_s$ with its image in $\tracelessTwoSphere^4$.

  The first two relations, $[p, \lambda_p]=1$ and $[q,\lambda_q]=1$ hold for all $\rho\in \tNAT_s$, as a consequence of the perturbation condition. 
The relation $cd=gf$ is a consequence of the others. 
The relations $c=\bq\bp b p q, ~ d=\bh a w h,$ and $g= \bq e q$ show that the traceless requirement on $c,d,$ and $g$ follow from those for $b,f,$ and $a$.
Finally, the relation $ba=ef$ can be used to eliminate $e$, replacing the traceless restriction on $e$ by the equivalent restriction
$$\Real(ba\bff)=0.$$

 The remaining relation
 $ [\bff \bq f \ba \bb \bp b,h]=w$ is treated as follows. Observe that 
 $$\Real (\Ima(bh)b)=\Real (bhb-\Real(bh)b )=0.
 $$ This implies that $\Ima(bh)b=-b\Ima(bh)$, and so
\begin{equation}\label{bppb}pb=e^{x\Ima(bh)}b=b e^{x\bb \Ima(bh) b}=be^{-x\Ima(bh)}=b\bp.
\end{equation}
 Since $\rho(w)=-1$ the remaining relation is now equivalent to 
 $$ [\bff \bq f \ba  p  ,h]=-1$$
 
We now digress to prove a useful statement.  Two unit quaternions $h_1, h_2$ satisfy $[h_1,h_2]=-1$ if and only if $\Real(h_1)=\Real(h_2)=\Real(h_1h_2)=0$. This is because $[h_1,h_2]=-1$ implies $h_1 h_2=- h_2 h_1$, and so $\Real(h_1 h_2)=-\Real(h_2h_1)=-\Real(h_1h_2)$, which proves $\Real(h_1h_2)=0$. Also $\Real(h_1)=-\Real(h_2 h_1 h_2^{-1})=-\Real(h_1)$, and so $\Real(h_1)=0$.

Armed with these facts, we see that
$  [\bff \bq f \ba  p  ,h]=-1$ is equivalent to 
$$\Real(h)=\Real(\bff \bq f \ba  p)= \Real(h\bff \bq f \ba  p) =0.
$$
The equation $\Real(h)=0$ is a consequence of the traceless requirement.  The remaining two equations are easily seen to be the same as the two in the statement of the lemma, completing the proof.
\end{proof}

\subsection{Partial gauge fixing}\label{pgf} 
Lemma~\ref{4vars} expresses $\tNAT_s$ as the preimage of 0 under the  map
\begin{equation}\label{june}
\begin{split}
 F&=(F_1,F_2,F_3): \tracelessTwoSphere^4 \to \RR^3, \\
(b,f,h,a)&\mapsto \big(\Real(b a \bff),\Real(\bp a \bff q f),\Real(h\bp a \bff q f)\big). 
\end{split} 
\end{equation} 
This map is invariant under the diagonal $SU(2)$ conjugation action on $SU(2)^4$, restricted to the invariant subspace $\tracelessTwoSphere^4$.   Set $\Xupstairs   =F_1 ^{-1} (0)$ and $\Xdownstairs    = \Xupstairs   /SU(2)$.      

We  first study $\Xupstairs $,  which is simpler because $F_1$ is independent of $h$, $p$ and $q$, and hence of the parameter $s$.  We'll postpone the examination of $\tNAT_s =  \Xupstairs  \cap F_2^{-1} (0) \cap F_3 ^{-1}(0)$ until we have developed a simpler model for $\Xdownstairs$.

To begin, consider the diagonal conjugation invariant subset 
 \begin{equation}
\label{teepee}\tP=\{(b,f,a)~|~\Real(ba\bff)=0\}\subset \tracelessTwoSphere^3. 
\end{equation}
It is well known that this equation does not cut $\tP$  out transversely~\cite{Lin, Heu, HeuKro, JR}.    We can describe the quotient $\tP/SU(2)$ more easily by considering the inclusion of the torus
$$
T\hookrightarrow \tP,  (\gamma, \theta) \mapsto (e^{\gamma\bbk}\bbi,e^{\theta\bbk}\bbi,\bbi).$$

Each $(b,f,a)\in \tP$ can be conjugated to a triple with third component $\bbi$, so the torus hits every $SU(2)$ orbit and  the composite map $T\hookrightarrow \tP \to \tP/SU(2)$ is surjective.  Furthermore, the composition is constant on the orbits of the elliptic involution $\iota:T\to T$, so it induces a map on $P=T/\iota$.  Finally, note that $(\gamma, \theta)$ and $(\gamma', \theta')$ in $T$ map to the same point in $\tP$  if and only if they lie in the same $\iota $ orbit, so the induced map is a homeomorphism  (see~\cite[Proposition~3.1]{HHK1} for details).

This is summarized in the following diagram:
$$
\begin{tikzcd}
T 
\arrow[r, hook]
\arrow[d, "/ \iota"]
& 
\tP 
\arrow[d, "/ SU(2)"] 
\\
P 
\arrow[r, "\cong"]
& 
\tP/SU(2) 
&&
\end{tikzcd}
$$
 
Next, note that $$\Xupstairs =F_1 ^{-1}(0) = \tP \times \tracelessTwoSphere \subset \tracelessTwoSphere^4.$$   It follows that 
\begin{equation}\label{defofPhi}
\begin{split}
\mapVR:T\times \tracelessTwoSphere &\to \tracelessTwoSphere^4, \\
(\gamma,\theta, h)&\mapsto (b,f,a,h)= (e^{\gamma\bbk}\bbi, e^{\theta\bbk}\bbi, \bbi, h),
\end{split}
\end{equation}
hits every $SU(2)$ orbit in $\Xupstairs $.    Furthermore, the composition $T\times \tracelessTwoSphere \hookrightarrow \Xupstairs  \to \Xdownstairs $ is constant on equivalence classes under the {\em extended elliptic involution}
\begin{equation}\label{eei}
\begin{split}
\hat\iota:T\times \tracelessTwoSphere&\to T\times \tracelessTwoSphere, \\
(\gamma,\theta,h) &\mapsto (-\gamma,-\theta,-\bbi h \bbi). 
\end{split}
\end{equation}

It is straightforward to check that the induced map $\mapVRmodInvolution: ( T\times \tracelessTwoSphere ) /\hat \iota \to \Xdownstairs $ restricted to $( (T\setminus C)\times \tracelessTwoSphere ) /\hat \iota$ is injective,  but the map on the entire quotient $ ( T\times \tracelessTwoSphere ) /\hat \iota$
 is not a homeomorphism.  Again, we summarize with a diagram: 

\begin{equation}\label{commie diag}
\begin{tikzcd}
T \times \tracelessTwoSphere \arrow[r, hook, "\mapVR"] 
\arrow[d, "/ \hat\iota" 
]
& 
\Xupstairs 
\arrow[d, "/ SU(2)" 
]  
\\
( T \times  \tracelessTwoSphere )/\hat\iota
\arrow[r, "\mapVRmodInvolution", two heads]
& 
\Xdownstairs  \end{tikzcd}
\end{equation}

Next, we begin our examination of the second and third equations cutting out $\tNAT_s$.    
Set 
\begin{equation}
\label{defofVs}
\Vupstairs_s=\mapVR^{-1}(\tNAT_s)= \{ (\gamma, \theta, h) \in  T\times \tracelessTwoSphere \mid F\circ \mapVR (\gamma, \theta, h) = (0,0,0) \}.
\end{equation}
Since the involution $\hat \iota$ on $T\times \tracelessTwoSphere$ corresponds to conjugation by $\bbi$ in $\tNAT_s$, $\Vupstairs_s$ is $\hat \iota$ invariant, and we define \begin{equation*} \V_s = \Vupstairs_s / \hat \iota. \end{equation*}.

\begin{lemma} \label{red1} The involution $\hat\iota$  acts freely on $\Vupstairs_s$ if $|s|<\tfrac\pi 4$.   Moreover,  the composition of $\mapVR |_{\Vupstairs_s} $ with the quotient by $SU(2)$ induces   a surjection $\mapVRmodInvolution:\V_s\to \NAT_s$ making the diagram  below commute.
 $$\begin{tikzcd}
\Vupstairs_s \arrow[r, "\mapVR|_{\Vupstairs_s}" ] \arrow[d, "/\hat \iota "] &  \tNAT_s \arrow[d,"/SU(2)"]\\
\V_s\arrow[r,"\mapVRmodInvolution |_{\V_s} ", two heads]                                 & \NAT_s \end{tikzcd}$$
\end{lemma}
\begin{proof}
Given   $( \gamma,\theta ,h)\in \Vupstairs_s\subset T\times \tracelessTwoSphere$, consider 
$$(b,f,a, h)=(e^{\gamma\bbk}\bbi,e^{\theta\bbk}\bbi,\bbi, h) =\mapVR(\gamma,\theta,h)\in \tNAT_s.$$
Set $p=e^{s\Ima(bh)}, ~ q=e^{s\Ima(f\ba h)}$.
 Using the Euler formula  descriptions for $p$ and $q$, one can check that the assumption  $|s|<\tfrac\pi 4$  and the inequalities $\| \Ima(bh)\|\leq 1, \| \Ima(f\ba h)\|\leq 1$ imply that  
the $\bbi$ component of $\bp \bbi \bff q f$ is non-zero.  Combined with the condition  that  
 $F_3\circ \mapVR(\gamma, \theta, h)=\Real(h\bp \bbi \bff q f)=0$, this implies that 
\begin{equation}\label{eq:h_not_i}
h\ne \pm \bbi.
\end{equation} 
Therefore $-\bbi h \bbi\ne h$. It follows that $\hat\iota$ acts freely on $\Vupstairs_s$.

The composition $\Vupstairs_s \to \tNAT_s \to \NAT_s$   is constant on $\hat \iota$ orbits, so it factors through the quotient $\Vupstairs_s\to \V_s=\Vupstairs_s/\hat\iota$.  
The fact that $\mapVRmodInvolution|_{\V_s} :\V_s \to \NAT_s$ is surjective is immediate from the fact that $V_s$ is the preimage of $\NAT_s$ under $\mapVRmodInvolution$, which is surjective.   
\end{proof}

\emph{Partial gauge fixing} refers to the fact that any representation  in $\tNAT_s$  may be conjugated so that  $a$ is sent to $\bbi$ and $b,f$ are sent to the $ (\bbi,\bbj)\text{-circle}$ in $\tracelessTwoSphere$. These are precisely the restrictions in the definition of space $\Vupstairs_s$. 

We prove in Sections~\ref{sec:misses_corners} and~\ref{sec:genus_three} that when $0<|s|<\frac \pi 4$,  the map $\mapVRmodInvolution|_{\V_s} $ is, in fact, an analytic isomorphism, and $\V_s \cong \NAT_s$ is a genus three surface.  By contrast, Example~\ref{ex:Rzero} shows that $\V_0$ is a smooth 3-manifold
and the map $ \mapVRmodInvolution:\V_0\to \NAT_0 $ is a quotient map  collapsing four circles to points.   
 
\subsection{Restriction to peripheral subgroups}
Since every $\rho\in \tNAT_{s}$ satisfies $\rho(ba\bff\be)=1=\rho(cd\bff\bg)$ (with
$c=\bq\bp b p q, ~ d=\bh a w h$), the inclusion of peripheral subgroups $\Pi^\rout, \Pi^\rin\subset \Pi $ induce restriction maps 
\begin{align*}
\rrout_s:\tNAT_{s}&\to \tP  & &\text{and} &\rrin_s:\tNAT_{s}&\to \tP\\
\rho & \mapsto ( \rho(b),\rho(f), \rho(a)) && & \rho & \mapsto(\rho(c),\rho(f),\rho(d))
\end{align*}
Passing to the quotients by the $SU(2)$ action defines the {\em restriction to the boundary correspondence} 
   \begin{equation}
\label{correspondence}
\r_s=\rrout_s\times \rrin_s:\NAT_{s}\to P^\rout\times P^\rin.
\end{equation}

\begin{example}[$s=0$]\label{ex:Rzero}  The {\em unperturbed} traceless character variety corresponds to taking $s=0$.  This implies that $p=1=q$ and hence $\bp\bbi\bff qf= \bbi$. It follows that the equation $\Real(\bp\bbi\bff qf)=0$   is redundant.    Therefore
$$\tNAT_{0}=\{(b, f, a,h)\in (\tracelessTwoSphere)^4~|~ 
 \Real(ba \bff)=0 =\Real(h  a)\}.$$
The set $\{h\in \tracelessTwoSphere\mid \Real(h\bbi)=0\}$
 is the circle $\tracelessTwoSphere \cap  \bbi^\perp$, which we denote by 
 $\iperpcirc$, or simply just $S^1$ when the context is clear.   Then, 
 $$\Vupstairs_0=T\times S^1, ~  \V_0 =( T\times S^1 )/ \hat \iota  \text{ and } \tNAT_0=\tP\times S^1 .$$
Explicitly, $V_0$ is the torus bundle over $S^1$ with monodromy the elliptic involution.  It is Seifert-fibered over $P$ with four exceptional fibers  of order 2 at the corners, corresponding to the four fixed points of $\iota$.   Then $\NAT_0$ is obtained from $V_0$ by crushing each of these four exceptional fibers to a point. It is therefore a singular 3-dimensional space, with four singularities, each of whose neighborhood is a cone on a torus.  The restriction map $\NAT_0\to P\times P$ is induced by the composite of the Seifert fibration $V_0\to P$ (crushing every fiber to a point), and the diagonal map $\triangle_P \to P\times P$ (since $(b,f,a,h)=(c,f,d,h)$ when $s=0$). 
In particular, $\NAT_0$ is not a smooth 2-dimensional manifold, and the image of the restriction $r_0:\NAT_0\to P\times P$ intersects the singular stratum.
\end{example}   

\section{\texorpdfstring{The space $\NAT_s$ misses the corners when $s\ne 0$}{The space misses the corners when s is not zero}}\label{sec:misses_corners}

Recall that $C\subset P$ denotes the set of four corners of the pillowcase, and $P^*=P\setminus C$.
The main result of this section is the following.

\begin{theorem}\label{nogotheorem} If $0<|s|<\frac \pi 4$, the image of the restriction $\r_s: \NAT_{s}\to P\times P$ lies in $P^*\times P^*$. 
\end{theorem}
\noindent This is an immediate consequence of the following two lemm\ae.  
\begin{lemma} \label{no-golemma} If $0<|s|<\frac \pi 4$, then $\rrout_s (\NAT_s )\subset P^*$.  \end{lemma}
\begin{proof}  Suppose to the contrary that the conjugacy class $[(b,f,a,h)]\in \NAT_{s}$ is sent by $\rrout_s$ into $C$.   Up to conjugation, we can assume $a=\bbi$.  Then there exist   $n_b,n_f\in\{0,1\}$ so that $b=(-1)^{n_b}\bbi$ and $f=(-1)^{n_f}\bbi$.
We may then conjugate by the appropriate $e^{\alpha\bbi}$ so that, in addition,  $h=e^{\tau\bbk}\bbi$ for some $0\leq \tau\leq \pi$.   

Then
 $bh= (-1)^{n_b}\bbi e^{\tau\bbk}\bbi= (-1)^{n_b}  (-e^{-\tau\bbk}),$ 
so that $\Ima(bh)=(-1)^{n_b}\sin\tau\bbk$, and hence
$$p=e^{(-1)^{n_b} s \sin\tau\bbk}.$$
Similarly,  $\Ima(f\ba h)=(-1)^{n_f}h=(-1)^{n_f}e^{\tau\bbk}\bbi,$  so that  
$$q=e^{(-1)^{n_f} s e^{\tau\bbk}\bbi}.$$
Therefore
$$\bp \bbi\bff q f=\bp q \bbi=e^{-(-1)^{n_b} s \sin\tau\bbk}e^{(-1)^{n_f} s e^{\tau\bbk}\bbi}\bbi. $$
  We compute
  \begin{align*}0=\Real(\bp \bbi\bff q f)&=\Real( e^{-(-1)^{n_b} s \sin\tau\bbk}e^{(-1)^{n_f} s e^{\tau\bbk}\bbi}\bbi)\\
&= -(-1)^{n_f}\sin s\cos(\tau- (-1)^{n_b} s\sin\tau ). \end{align*}
Since $\sin s \ne 0$ and $0<|s|<\tfrac\pi 4$,  $\tau\in [0,\pi]$ must satisfy 
\begin{equation}
\label{eqo}\tfrac\pi 2= \tau-(-1)^{n_b}s \sin \tau.\end{equation}
Next we compute
  \begin{align*}0=\Real(h\bp \bbi\bff q f)&=\Real(e^{\tau\bbk}\bbi e^{-(-1)^{n_b} s \sin\tau\bbk}e^{(-1)^{n_f} s e^{\tau\bbk}\bbi}\bbi)\\
&=-\cos s \cos(\tau+(-1)^{n_b}s \sin\tau).
\end{align*}
Since $\cos s \ne 0$ and $0<|s|<\tfrac\pi 4$,  $\tau\in [0,\pi]$ must also satisfy 
\begin{equation}
\label{eqt}\tfrac\pi 2= \tau+(-1)^{n_b}s \sin \tau.\end{equation} 
But since $0<|s|<\tfrac \pi 4$,  Equations~\eqref{eqo} and~\eqref{eqt} have no common solution $\tau$.   This contradiction finishes the proof.\end{proof}
 
\begin{lemma}
 \label{no-golemma2}
  If $0<|s|<\frac \pi 4$, then $\rrin_s( \NAT_{s})\subset P^*$.
 \end{lemma}
\begin{proof} Recall that the $SU(2)$ elements corresponding to the generators $c,d,g$ of the fundamental group of the inner 4-punctured 2-sphere are expressible in terms of $a,b,e,h,p,q$ by the equations 
$$c=\bq \bp b pq,~d=-\bh a h\text{, and }g=\bq e q.$$  

We'll argue by contradiction, so assume  that there exists a $(b,f,a,h)\in \tNAT_{s}$  and $n_c, n_d, n_g\in\{0,1\}$  satisfying $n_c+n_d+n_g\equiv 0 $ mod $2$,
 such that  
 $$f=(-1)^{n_c} c=(-1)^{n_d} d=(-1)^{n_g}g.$$ 
 After conjugation as in Section (\ref{pgf}), we may further assume that 
 $$ a=\bbi, b=e^{\gamma\bbk}\bbi, f=e^{\theta\bbk}\bbi,~\text{ and } e=e^{(\theta+\gamma)\bbk}\bbi.$$
 (We caution the reader that throughout this argument, $e$ with no exponent corresponds to the outer sphere generator $e=ba\bff$, where as all instances of  $e$ with an exponent refers to Euler notation for a quaternion.)       Thus 
\begin{equation}\label{eqn5.4}  d= (-1)^{n_d}  e^{\theta\bbk}\bbi, ~ c=  (-1)^{n_c} e^{\theta\bbk}\bbi,~g= (-1)^{n_g}e^{\theta\bbk}\bbi. \end{equation}
 
 Using the equations above for $c$ and $g$, and the consequence $c=(-1)^{n_c + n_g} g$ of Equation (\ref{eqn5.4}), we obtain $\bp b p=\pm e$.    
 Rewriting the left side as $bp^2$ using  Equation \eqref{bppb}, 
 we can solve for $p^2$ to obtain
 
 \begin{equation} \label{expo} 
 p^2 =e^{2s \Ima (bh)} =\pm \bb e = \pm e^{ -\theta \bbk} .
 \end{equation} 
 
 We now distinguish three cases $\theta =0$ (i.e., $e^{2s\Ima(bh)}=1$),  $\theta =\pi$ (i.e., $e^{2s\Ima(bh)}=-1$), and  $0<\theta <\pi$ (i.e., $e^{2s\Ima(bh)}$ is a non-central unit quaternion).   
If $e^{2s\Ima(bh)}=1$, then since $0<|s|<\frac \pi 4$ and $\| \Ima(bh) \| \leq 1$, $\Ima(bh)=0$.  Since $b,h\in \tracelessTwoSphere$, this implies that $h=\pm b$ is orthogonal to $\bbk$.  By those same inequalities, $e^{2s\Ima(bh)}$ cannot be $-1$.  If $e^{2s\Ima(bh)}$ is a non-central unit quaternion, then $\Ima(bh)$ is a non-zero multiple of $\bbk$, which again implies $h $ is orthogonal to $\bbk$.  
It follows that, for  some $\beta$, $h$ has the form
 $$h=e^{\beta\bbk}\bbi,$$
 from which we conclude that $f\ba h=e^{(\theta+\beta)\bbk} \bbi$ is also orthogonal to $\bbk$.

 The equation $g=\bq e q$ now can be expressed geometrically by stating that rotating $e=e^{(\theta + \beta)\bbk} \bbi \in \bbk^\perp$ by angle $-2s \| \Ima (f\ba h)\| = -2s$ about the axis $\Ima(f\ba h)=e^{(\theta + \beta) \bbk} \bbi \in \bbk^\perp$ yields $g=(-1)^{n_g} e\in \bbk^\perp$  This is impossible since $0<|s|<\tfrac \pi 4$, unless $g=e$ and  the axis $\Ima(f\ba h)$ is a scalar multiple of $g$.  In this case,  writing $d,g,e,$ and $\Ima(f\ba h) $ out in terms of $\theta, \gamma, \beta$ shows that $\gamma, \theta, \beta$ are all zero modulo $\pi$, a contradiction by Lemma \ref{no-golemma}.  \end{proof}

 Recall from Lemma~\ref{red1} that the map $\mapVR:\Vupstairs_s\to \tNAT_s$ induces a surjection $\mapVRmodInvolution:\V_s\to \NAT_s$.
\begin{corollary}
 \label{RequalsS}
 When $0<|s|<\tfrac\pi 4$, the surjection $\mapVRmodInvolution:\V_s\to \NAT_s$ is a homeomorphism.
\end{corollary} 
\begin{proof}
Suppose that $(\gamma,\theta,h),(\gamma',\theta',h')\in \Vupstairs_s$ satisfy $\mapVRmodInvolution([\gamma,\theta,h])=\mapVRmodInvolution([\gamma',\theta',h'])$. According to the commutative diagram (\ref{commie diag}), this means that
$$
\mapVR(\gamma,\theta,h)=(e^{\gamma\bbk}\bbi, e^{\theta\bbk}\bbi,  \bbi, h) \quad \text{ and }\quad
\mapVR(\gamma',\theta',h')=(e^{\gamma'\bbk}\bbi, e^{\theta'\bbk}\bbi,  \bbi, h')
$$
are conjugate by some $v\in SU(2)$.
From consideration of the third factor, $v=e^{\nu\bbi}$ for some $\nu$.
On the other hand, from the first two components, we see 
$$e^{\nu\bbi} e^{\gamma'\bbk}\bbi e^{-\nu \bbi}=e^{\gamma\bbk}\bbi \mbox{ and } e^{\nu \bbi}   e^{\theta'\bbk}\bbi e^{-\nu \bbi} =  e^{\theta\bbk}\bbi, $$ which simplify to 
$$ 
e^{\gamma' e^{2\nu \bbi}\bbk}  = e^{\gamma\bbk} \mbox{ and }  e^{\theta' e^{2\nu \bbi}\bbk}  = e^{\theta\bbk} .$$

Theorem~\ref{nogotheorem} states that 
$\sin^2\gamma+\sin^2\theta\ne 0$.  This implies that $e^{\nu\bbi}\in \{\pm 1,\pm \bbi\}$ and 
$$\text{either }(\gamma',\theta',h')=(\gamma,\theta,h),\quad \text{ or }(\gamma',\theta',h')=(-\gamma,-\theta,-\bbi h\bbi)=\hat \iota(\gamma,\theta,h).$$
This completes the proof of injectivity of the map $\mapVRmodInvolution$.

Both $\Vupstairs_s$ and $\tNAT_s$ are compact and Hausdorff and the quotient map $\tNAT_s\to \NAT_s$ is a closed map, so $\V_s$ is compact and $\NAT_s$ is Hausdorff and $\mapVRmodInvolution$ is a homeomorphism. 
\end{proof}

\section{\texorpdfstring{For small non-zero $s$, $\NAT_{s}$ is  a closed orientable surface of genus 3}{For small non-zero s, the space is  a closed orientable surface of genus 3}}\label{sec:genus_three}

The purpose of this section is to identify $\NAT_s$ for $s$ small and non-zero and to establish some properties of the restriction map $r_s: \NAT_s \to P^* \times P^*$.  The results are summarized in Theorem~\ref{thm1}.     In this paper, we only use the smoothness of $\NAT_s$, not the analyticity.  The following theorem is equivalent to the collection of statements in the second paragraph of Theorem \ref{thmA}.

\begin{theorem}
 \label{thm1} 
 For $s$ small enough and non-zero, $\NAT_s$ is a real-analytic, smooth, closed orientable surface of genus 3, and the restriction $\r_s:\NAT_s\to (P^*)^-\times P^*$ is a Lagrangian immersion. \end{theorem}

Denote the circle of imaginary unit quaternions orthogonal to $\bbi$ by 
$$S^1 _{\bbi ^\perp} = \tracelessTwoSphere \cap \bbi ^\perp = \{ h\in \tracelessTwoSphere \mid \Real(h \bbi)=0 \}.$$  Write elements of this circle in the form $e^{\nu \bbi} \bbk$.   The conjugation action of  $\bbi$ on $\tracelessTwoSphere$ restricts to a free action on  circle $S^1 _{\bbi ^\perp}  $, given by
 $$\bbi e^{\nu\bbi}\bbk(-\bbi)=-e^{\nu\bbi}\bbk=e^{(\nu+\pi)\bbi}\bbk.$$

 \subsection{Outline of the proof of Theorem~\ref{thm1}.}

 The first step in the proof is to replace the family of $s$-dependent functions $(F_2\circ \mapVR, F_3\circ \mapVR):T\times \tracelessTwoSphere \to \RR^2$ cutting out $\Vupstairs_s$ with a different family of $s$-dependent functions with the same zero set when $s$ is near but not equal to zero.  In doing so, we  obtain a real-analytic manifold in $T\times \tracelessTwoSphere \times (-\epsilon, \epsilon) $ for small $\epsilon>0$, which is a surface bundle over $(-\epsilon, \epsilon)$.  The fiber above $s=0$ is not the same as $\Vupstairs_0$, but rather an analytic manifold $\genusFiveSurface$ that is contained in $\Vupstairs_0$.  
 
 This first step is accomplished in Section \ref{step1} by replacing the first of the two functions by its quotient by $s$, extended suitably to $s=0$.   The new equations cut out $\genusFiveSurface$ when $s=0$, and hence they cut out diffeomorphic copies of $\genusFiveSurface$ for $s$ near zero.   (For our argument about the smoothness of the quotient $\frac{F_2}{s}$ and its zero set, we only need the functions $F_i$ to be smooth, but we note throughout that the functions $F_i$ are actually real-analytic.)  
 
In Section \ref{step2},  we identify $\genusFiveSurface$ as a closed, orientable surface of genus five, which is invariant under $\hat \iota$.  Furthermore, we identify the $\hat \iota $ action on it, and show $\genusFiveSurface/\hat \iota$ is orientable with genus three.  
 The proof of  Theorem~\ref{thm1} is then finished by establishing the necessary analytic and Lagrangian properties.

 \subsection{\texorpdfstring{Better behaved equations cutting out $\Vupstairs_s$ for $s\neq 0$}{Better behaved equations cutting out V-hat-s for non-zero s}} \label{step1} In this section, we replace the equations cutting out $\Vupstairs_s$ by different equations that behave better near $s=0$.  To describe the new equations, and to highlight the role of $s$ as a variable in the present argument, we define the function 
 \begin{align} \label{defH}
H &:T\times \tracelessTwoSphere\times \RR\to \RR^2\\
H ( \gamma,\theta,h,s)&=\left\{ \begin{array}{ccc} ((\tfrac 1 s\Real(\bp a \bff qf),\Real(h(\bp a \bff qf)))) &  \mbox{ for } & s\neq 0\\
(\Real(h(ba+f)), \Real(ha)) & \mbox{ for }& s=0. \end{array} \right.  
\nonumber
\end{align}
with $(b,f,a,h)=\mapVR(\gamma,\theta,h)$ (Equation~\eqref{defofPhi}) and $p,q$ given by the perturbation condition.

\begin{proposition}\label{analytics} When $s\neq 0$, 
$$\Vupstairs_s =\{ (\gamma, \theta , h) \mid H(\gamma, \theta, h, s) = (0,0) \}.$$  The function $H$ defined in equation (\ref{defH}) is analytic and invariant under the action of $\hat \iota$.  

When $s=0$, the derivative of $H$ maps the tangent space of $T\times \tracelessTwoSphere$ onto $\RR^2$ along the zero set.
Hence, the zero set of $H$ near $T\times \tracelessTwoSphere \times \{0\}$ is a 3-manifold.  The zero set meets $T\times \tracelessTwoSphere \times \{0\}$ transversely in  a compact, oriented surface $\genusFiveSurface\times \{0\}$, where  
$$ \genusFiveSurface=\{ (\gamma, \theta, e^{\nu \bbi} \bbk) \in T\times S^1 _{\bbi ^\perp} \mid \sin \theta \sin \nu + \sin \gamma \cos \nu = 0\}\subset T\times \tracelessTwoSphere.$$   The action of  $\hat \iota$ is free on $\genusFiveSurface$.  
 \end{proposition}  
 
 The first statement in the proposition is obvious.  Note that the second component of $H$ is simply $F_3\circ \mapVR$, with the $s$-dependence made explicit, and simplified when $s=0$ so that $p=q=1$.  This component is clearly analytic.
 
The analyticity of the first component of $H$ is immediate from the next lemma, once $b,f,$ and $a$ are replaced by $e^{\gamma \bbk} \bbi, e^{\theta \bbk} \bbi,$ and $\bbi$.  
\begin{lemma}
 \label{dividebys}
  Consider the analytic quaternion-valued function $\tracelessTwoSphere^4 \times \RR \to \HH$ given by $(b,f,a,h,s)\mapsto \bp a \bff q f,$ where 
  $p=e^{s\Ima(bh)}$ and $q=e^{s\Ima(f\ba h)}$.  Then the first order expansion of this function is 
  $$ a+s(-\Ima(bh)a +a(\Ima(\ba f h))) + O(s^2),$$
  
  and hence  
  $$  \Real(h\bp a \bff q f)=\Real(ha)+O(s)   \quad
  \text{ and } \quad       \Real(\bp a \bff q f)=s\big( \Real(h(ba+f))\big)+O(s^2).$$  \qedhere
\end{lemma}
\begin{proof}
We compute
\begin{align*}
\bp a \bff q f&=e^{-s\Ima(bh)} a \bff e^{s\Ima(f\ba h)}f=e^{-s\Ima(bh)} a  e^{s\Ima( \ba h f)} \\
&=(1-s\Ima(bh) +O(s^2))a(1+s\Ima(\ba h f)+O(s^2))\\
&=a +s(\Ima(hb)a +a\Ima(\ba h f)) +O(s^2).
\end{align*}
Therefore $\Real(h\bp a \bff q f)=\Real(ha)+O(s)$. Moreover, since $\Real(a)=\Real(h)=0$, 
\begin{align*}
\Real(\bp a \bff q f)&= s\Real(\Ima(hb)a +a\Ima(\ba h f)) +O(s^2)\\
&=s\Real(hba +a \ba h f) +O(s^2)\\
&=s\Real(h(ba+f))+O(s^2).
\end{align*}
\end{proof}
 
 \noindent{\em Proof of Proposition \ref{analytics}.}  With the analyticity of $H$ established, we complete the proof of the proposition as follows.  The action of $\hat \iota$ on $(\gamma, \theta, h)$ is the same as conjugation by $\bbi$ on $(b,f,a,h)$, and $p$ and $q$ also change by conjugation by $\bbi$,  so the real parts of any products of these quaternions and their inverses are unaffected. The final claim that $\hat \iota $ acts freely on $\genusFiveSurface$ is immediate from the fact that $h\neq \pm \bbi$ in $\genusFiveSurface$.  
 
 When $s=0$, the second component $\Real(h \bbi)$ of $H$ cuts out the 3-torus $T\times S^1 _{\bbi ^\perp} \subset T\times \tracelessTwoSphere$ transversely.    Plugging $a=\bbi$, $b=e^{\gamma \bbk} \bbi$, $f=e^{\theta \bbk} \bbi$, and $h=e^{\nu \bbi} \bbk$ into the formula for the first component leads to the condition $\sin \theta \sin \nu + \sin \gamma \cos \nu=0$ defining $\genusFiveSurface$.  Our next step is to show that $\genusFiveSurface$ is cut out transversely from the 3-torus by the first component of $H$.  
 
 The only solutions to $\frac{\partial}{\partial \nu} (\sin \gamma \cos \nu + \sin \theta \sin \nu) =0$ on $\genusFiveSurface$ are points where $\sin \gamma =\sin \theta =0$, but it is easy to see that at least one of the other partial derivatives, $\cos \gamma \cos \nu$ and $\cos \theta \sin \nu$,  is non-zero at these points.  This implies the transversality claim in the proposition.  Orientability follows from the fact that $\genusFiveSurface$ is realized  as a level set of a map to $\RR^2$.  This completes the proof.\qed
 
 Proposition \ref{analytics} has the following immediate consequence. 
 \begin{corollary} \label{small s cor} For some $\epsilon>0$, $\Vupstairs_s$ is an analytic manifold  diffeomorphic to $\genusFiveSurface$ whenever $0<|s|<\epsilon.$   \end{corollary} 
 \begin{proof}  Analyticity follows from the analyticity of $H$.  Since the zero set of $H$ is a manifold near $T\times \tracelessTwoSphere \times \{0\}$ intersecting this copy of $T\times \tracelessTwoSphere$ transversely, projection from the zero set to $\RR$ is a submersion, and the zero set fibers over $(-\epsilon, \epsilon)$ for $\epsilon$ sufficiently small.   \end{proof}

\subsection{\texorpdfstring{Identifying  the surfaces $\genusFiveSurface$  and $\genusFiveSurface/\hat\iota$}{Identifying  the surfaces Sigma-hat   and Sigma-hat/iota}}\label{step2}
We   next show  that  $\genusFiveSurface$ is an orientable, connected genus five surface and and that $\genusFiveSurface/\hat\iota$ is a genus three surface.

 \medskip

Let $T^*=T\setminus C$, the torus with the four fixed points of  $\iota $ removed.  Define a function   
\begin{equation}
\begin{split}
\label{defnofh} 
h:T^* &\to  S^1 _{\bbi ^\perp}, \\
 (\gamma,\theta) &   \mapsto
 \frac{  \sin\gamma\bbj  +\sin\theta\bbk}{\sqrt{
\sin^2\gamma +\sin^2\theta}  }.
\end{split}
\end{equation}
Use this function to define two maps:
\begin{equation}
\label{defnofsig} 
\sigma_\pm:T^*\to T^*\times \tracelessTwoSphere,~
\sigma_{\pm}(\gamma,\theta)
 = \big(\gamma,\theta, \pm h(\gamma,\theta)\big).
\end{equation}

\begin{lemma} \label{xlem} \hspace{1cm}

\begin{enumerate}
\item 
The images of $\sigma_+$ and $\sigma_-$ lie in $\genusFiveSurface$.
\item The maps $\sigma_\pm:T^*\to \genusFiveSurface$ are open embeddings which induce opposite orientations on $\genusFiveSurface$.
\item The complement of $\genusFiveSurface^*\coloneqq \sigma_+(T^*)\cup \sigma_-(T^*)$ in $\genusFiveSurface$  is $C\times S^1_{\bbi^\perp}$, a disjoint union of four circles.
\item The images of $\sigma_+$ and $\sigma_-$ are disjoint.
\item The maps $\sigma_+$ and $\sigma_-$ are equivariant with respect to the involutions $\iota,\hat\iota$, and hence define open embeddings (we use the same notation)
$$\sigma_\pm:P^*\to \genusFiveSurface/\hat\iota$$ 
with disjoint image and complement a disjoint union of four circles.  
\item  The surface $\genusFiveSurface$ is connected. In particular, $\genusFiveSurface$ has genus five, and  $\NAT_s\cong \genusFiveSurface/\hat\iota$ is an orientable closed surface of genus three for small non-zero $s$.
\end{enumerate}
\end{lemma}

\begin{figure}[ht]
\labellist 
\pinlabel $\sigma_+(T^*)$ at 40 90
\pinlabel $\sigma_-(T^*)$ at 114 90
\pinlabel $\sigma_+(P^*)$ at 276 90
\pinlabel $\sigma_-(P^*)$ at 354 90
\pinlabel $/\hat \iota$ at 193 80
\pinlabel $\genusFiveSurface/\hat\iota=\NAT_s$ at 317 8
\pinlabel $\genusFiveSurface$ at 80 8
\endlabellist
\centering
\includegraphics[width=0.8\textwidth]{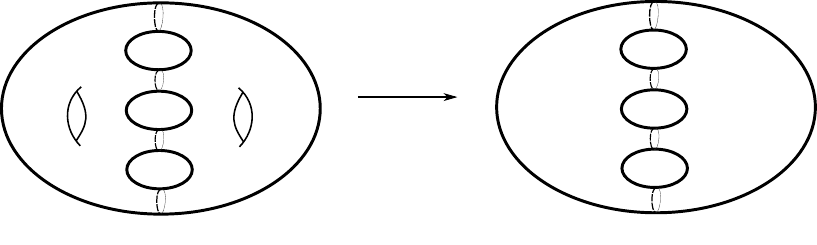}
\caption{}\label{fig:illustrating_lemma} 
\end{figure}

\begin{proof} 
Fix $(\gamma,\theta)\in T$. For any $h\in \tracelessTwoSphere$, $(\gamma,\theta,h)$ lies in $\genusFiveSurface$ if and only if $h$ has the form 
$h=e^{\nu\bbi}\bbk$, where $\nu$ which satisfies the equation  \begin{equation} \label{defining eq for g5S} \sin \theta \sin \nu + \sin \gamma \cos \nu=0.\end{equation} 
If $(\gamma,\theta)\in T^*$, there are precisely two possible choices for $h$, namely the one given in   Equation~\eqref{defnofh} and its opposite. On the other hand, if $(\gamma,\theta)\in C$, then Equation (\ref{defining eq for g5S}) holds for all choices of $\nu$.  This establishes the first and third assertions.

  It is straightforward to check that the closure of each image $\sigma_\pm (T^*)$ contains $C\times S^1_{\bbi ^\perp}$, which implies that $\genusFiveSurface$ is connected.  It follows that each of $\sigma_\pm (T^*) \cup (C\times S^1 _{\bbi ^\perp})$ is a compact genus one orientable surface with four boundary components.

 The fourth assertion --- that the images of $\sigma_+$ and $\sigma_-$ are disjoint --- follows from  $\pm h(\gamma,\theta)$ being two unit vectors pointing in opposite directions.

Referring to Equations~\eqref{ei} and~\eqref{eei} for the formulae for  $\iota,\hat\iota$, we have  
\begin{align*}
\sigma_+(\iota(\gamma,\theta))&=\sigma_+(-\gamma,-\theta)\\
&=(-\gamma,-\theta, h(-\gamma,-\theta))\\&=
(-\gamma,-\theta, -h(\gamma,\theta))\\&=
\hat\iota(\sigma_+( \gamma,\theta )).
\end{align*}
Similarly for $\sigma_-$.

 Since $\genusFiveSurface$ is a closed, orientable and connected surface formed by taking a union of two tori with four open disks removed by identifying their boundaries, $\genusFiveSurface$  has  genus five.

 The maps $\sigma_\pm:T^*\to \genusFiveSurface$
are smooth open embeddings, with disjoint image.  Since $\genusFiveSurface\subset T\times \tracelessTwoSphere$ disconnects the 3-torus, it is null homologous. This implies that the algebraic intersection number of $\genusFiveSurface$  with a fiber $\{(\gamma,\theta)\}\times   \tracelessTwoSphere$ is zero.  For $(\gamma,\theta)\in T^*$, the two embeddings $\sigma_\pm$ send $(\gamma,\theta)$ to the two intersection points of $\genusFiveSurface$ with the fiber $\{(\gamma,\theta)\}\times   \tracelessTwoSphere$, and hence these two points are transverse intersection points and so must have opposite sign.   It follows that $\sigma_\pm$ induce opposite orientations on $\genusFiveSurface$.

 The involution $\hat\iota$ acts freely on $\genusFiveSurface$ since $h\in\bbi^\perp$.   It   leaves each of the four circles invariant. Since the maps $\sigma_+$, $\sigma_-$ are equivariant, it follows that cutting $\genusFiveSurface/\hat\iota$ along the disjoint images of these four circles results in two disjoint copies of the closure of $P^*$ a 2-sphere with four open disks removed.  
  Hence $\genusFiveSurface/\hat\iota$ is orientable, connected, and has Euler characteristic $-4$, and thus has genus three.
\end{proof}

\subsection{Proof of Theorem~\ref{thm1}}
 The proof of Theorem~\ref{thm1} now follows quickly.  Corollary~\ref{small s cor} identifies $\NAT_s$ with $\genusFiveSurface$ when $0<|s|<\ep$.   Lemma~\ref{xlem}  then shows that $\NAT_s$ is a smooth genus 3 surface. That the restriction map has image in $P^*\times P^*\subset P\times P$ is the content of Theorem~\ref{nogotheorem}. That it  is a Lagrangian  immersion into the symplectic manifold $(P^*)^- \times P^*$ is a general feature of restriction  maps on perturbed flat moduli spaces,  as originally shown in~\cite{Herald}.  A complete proof can   be found in \cite{CHK}.
 \qed

\section{Restriction to the boundary}\label{sec:restriction}
In this section we introduce some auxiliary notation, and then collect some basic facts about the restriction-to-the-boundary map $\r_s:\NAT_s\to P^*\times P^*$,  in preparation for the completion of the proof of Theorem \ref{thmA} in Section \ref{homotopyclass}.   

\subsection{Notation}
Set
\begin{equation}\label{defofE}
\genusThreeSurface=\genusFiveSurface/\hat\iota.\end{equation}
Thus $\genusThreeSurface$ is a genus three surface diffeomorphic to $\NAT_s$ for small non-zero $s$.  

Proposition~\ref{analytics} implies that the parameterized solution set
$$\genusFiveSurfaceFamily_\ep=H^{-1}(0)\cap \left( T\times \tracelessTwoSphere\times (-\ep, \ep) \right) $$
 is diffeomorphic to a product $\genusFiveSurface \times (-\ep, \ep)$.  
Define 
\begin{equation}\label{EE}\genusThreeSurfaceFamily_\ep=\genusFiveSurfaceFamily_\ep/\hat\iota\subset  (T\times \tracelessTwoSphere)/ \hat\iota \times(-\ep, \ep).\end{equation}
Corollary~\ref{small s cor} and Theorem~\ref{thm1} imply that $\genusThreeSurfaceFamily_\ep$ is a surface bundle over $(-\ep, \ep)$ whose genus three fiber over $s$ is 
$$\genusThreeSurface_s=\begin{cases} \genusThreeSurface &\text{ if } s=0,\\
 \NAT_s & \text{ if } 0<|s|<\ep.\end{cases}$$
  Fix a bundle trivialization 
  $$\psi:\genusThreeSurface\times(-\ep,\ep)\xrightarrow{\cong} \genusThreeSurfaceFamily_\ep$$
  satisfying $$\psi(-,0)={\rm Id}_\genusThreeSurface,$$ and write $\psi_s=\psi(-,s)$.

\medskip

Recall that the   restriction map 
$$\rrout:( T\times  \tracelessTwoSphere)/ \hat\iota  \to P$$
is the map induced by the projection $T\times \tracelessTwoSphere \to T$. 
Define 
\begin{align*}
\rrin:\genusThreeSurfaceFamily_\ep &\to P=\tP/SU(2)\\
 ([(\gamma,\theta),h],s) &\mapsto [c,f,d]
\end{align*}
using $c=\bq\bp b p q, ~ d=\bh a w h,$ with $a=\bbi$, $b=e^{\gamma\bbk}\bbi$, $f=e^{\theta\bbk}\bbi$,  $h\in \tracelessTwoSphere$, etc. (a precise formula is not needed).
Hence $$\rrout\times\rrin:\genusThreeSurfaceFamily_\ep\to P^-\times P.$$
By design, 
the identification $\genusThreeSurface_s=\NAT_s$ for $s\ne 0$ identifies  the maps
  $$\r_s=\rrout_s\times \rrin_s:\NAT_s\to P\times P \quad \text{ and } \quad(\rrout \times \rrin)(-,s):\genusThreeSurface_s\to P \times P,$$ 

We now mark  the domain using $\psi$ to define a homotopy of the genus three surface $\genusThreeSurface$:\begin{equation}
\label{Lambda}
 \coreMap=\coreMap^\rout\times \coreMap^\rin\coloneqq (\rrout\times \rrin)\circ \psi: \genusThreeSurface\times(-\ep, \ep)\to P\times P.  \end{equation}
and set $ \coreMap_s= \coreMap(-,s)$.

Lemma~\ref{xlem} states that the maps $
\sigma_\pm:P^*\to \genusThreeSurface$ are open embeddings with disjoint images.  The complement of the union of their images is  a union of four circles, one corresponding to each corner.  By construction,
$\sigma_\pm$ are two sections over $P^*$ of the map $\coreMap^\rout(-,0):\genusThreeSurface\to P$.

Fix some $\delta>0$.  Let $U_\delta(C)\subset P$ denote a $\delta$ neighborhood  of the set of corners $C$.    For convenience, we introduce the notation 
\begin{equation}\label{deltas} P^\delta=P\setminus U_\delta(C) \quad \text{  and  } \quad \genusThreeSurface^\delta=\sigma_+(P^\delta)\sqcup \sigma_-(P^\delta).\end{equation}
The  sections $\sigma_\pm$ exhibit the restriction 
$$\coreMap^\rout(-,0)|_{\genusThreeSurface^\delta}:\genusThreeSurface^\delta\to P^\delta$$
as a trivial 2-fold covering space.

\subsection{Basic facts about restriction-to-the-boundary map}

The  following lemma summarizes the notation and the basic facts established above.
 
\begin{lemma}\label{notat} The homotopy $\coreMap:\genusThreeSurface\times(-\ep,\ep)\to P^-\times P$ satisfies:
\begin{enumerate}
\item 
 For non-zero $s$, the maps $\r_s:\NAT_s\to (P^*)^-\times P^*$ and $\coreMap_s:\genusThreeSurface\to P^-\times P$ differ  by an analytic reparameterization of their domain. 

 \item When $s=0$, the map $\coreMap_0:\genusThreeSurface\to P^-\times P$ is given by $$\coreMap_0[(\gamma,\theta),h]= ([\gamma,\theta],[\gamma,\theta]).$$
 This map restricts to a trivial 2-fold covering over the diagonal $\triangle_{P^*} \subset (P^*)^-\times P^*$, and maps the four remaining circles to each of the four corners of $\triangle_P \simeq P$.
 \item  For any $s$ small enough the restriction of 
 $
\coreMap_s^\rout: \genusThreeSurface \to P
$
to the preimage   $(\coreMap_s^\rout)^{-1}(P^{\delta})$,
$$
\coreMap_s^\rout:(\coreMap_s^\rout)^{-1}(P^{\delta})\to P^\delta,
$$  is a trivial 2-fold cover.
\end{enumerate}
 \qed
\end{lemma}

\section{\texorpdfstring{Calculation of the homotopy class of $\coreMap_s:\genusThreeSurface\to   (P^*)^-\times P^*$}{Calculation of the homotopy class of u-s}}
\label{homotopyclass}

In this section we determine the homotopy class of the map $\coreMap_s:\genusThreeSurface\to (P^*)^-\times P^*$,  and hence of the Lagrangian immersion $\r_s:\NAT_s\to (P^*)^-\times P^*$, for small positive $s$,  to complete the proof of Theorem \ref{thmA}.

\subsection{\texorpdfstring{Homology of $P^*\times P^*$}{Homology of P* x P*}}

\begin{wrapfigure}{r}{0.3\textwidth}
\vspace{-0.4cm}
\labellist 
\pinlabel $\color{red}\l_0$ at 300 90
\pinlabel $\color{red}\l_3$ at 62 90
\pinlabel $\color{red}\l_2$ at 62 270
\pinlabel $\color{red}\l_1$ at 280 300
\pinlabel $\color{blue}A_1$ at 155 190
\pinlabel $\color{blue}A_2$ at 50 200
\pinlabel $\color{blue}A_0$ at 185 45
\endlabellist
\centering
\includegraphics[width=0.23\textwidth]{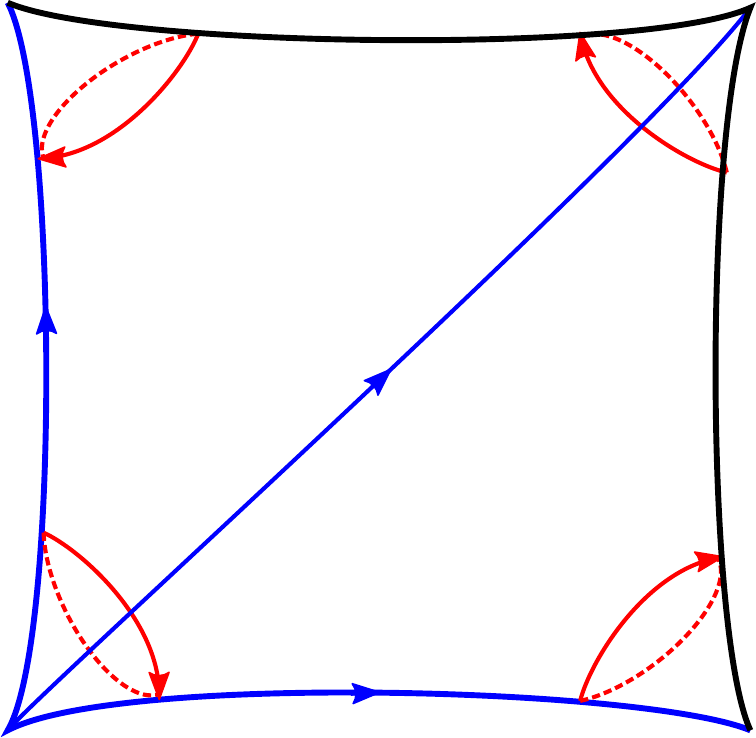}
\caption{ }
\label{pillowgensfig}
\end{wrapfigure} 

 The pillowcase $P$ is a 2-sphere, and $P^*$ is the complement of the four corner points $C$.  Hence $P^*$ is homotopy equivalent to a wedge of three circles, so that $H_1(P^*)\cong\ZZ^3\cong H_1(P,C)$.  Figure~\ref{pillowgensfig} illustrates a collection of oriented cycles $\l_0, \l_1, \l_2, \l_3\in H_1(P^*)$. The subset $\{\l_0, \l_1,\l_2\}$ forms a basis of $H_1(P^*)$, and $\l_3=-(\l_0+\l_1+\l_2)$.  Also shown   are three relative cycles $A_0,A_1,A_2$, representing a basis of $H_1(P,C)$.

The K\"unneth theorem gives
$$H_2(P^*\times P^*)\cong H_1(P^*)\otimes H_1(P^*)\cong \ZZ^9.$$
 Furthermore, Poincar\'e duality implies
 $$H^2(P^*\times P^*)=H_2((P,C)\times (P,C))=H_1(P,C)\otimes H_1(P,C),$$
where $(P,C)\times (P,C)=(P\times P, C\times P \cup P\times C)$. The next proposition follows easily from these calculations and standard algebraic topology.
 
\begin{proposition}\label{homology} \hspace{1cm}
\begin{enumerate}
\item  
The nine homology cross products $\displaystyle \l_{i}\times \l_j, ~i,j\in \{0,1,2\},$ \\form a basis of $H_2(P^*\times P^*)\cong \ZZ^9$.

\item
The nine relative homology cross products $\displaystyle A_{i}\times A_j, ~i,j\in \{0,1,2\},$ \\
form    a basis for $H_2((P,C)\times (P,C))$. 

\item
The bases above are dual with respect to the intersection pairing $H_2((P,C)\times (P,C))\times H_2(P^*\times P^*)\to \ZZ$, namely
 $$(A_i\times A_j)\cdot (\l_\ell\times \l_m)=\delta_{(ij),(\ell m)},~ i,j,l,m\in\{0,1,2\}.$$  
Moreover, since $\l_3=-(\l_0+\l_1+\l_2)$,
 $$
 (A_i\times A_j)\cdot(\l_3\times \l_3)=1\text{ for all } i,j\in\{0,1,2\}.
 $$
\end{enumerate}
\end{proposition}

The cross products $\l_i\times \l_j$ for $i,j\in \{0,1,2,3\}$ are represented 
by the sixteen smoothly embedded (Lagrangian) linking tori in $P\times P\setminus (C\times P\cup P\times C)$.  More precisely, for any $\delta>0$ the product of   two circles of radius $\delta$ centered at the $i$th and $j$th corners is the embedded {\em linking torus} representing $\l_i\times \l_j$.   Taking $i,j\in \{0,1,2\}$ yields a basis of $H_2(P^*\times P^*)$.  Similarly, the nine smoothly embedded squares $A_i\times A_j, ~ j=0,1,2$ form a basis for $H_2((P,C)\times (P,C))$.

When $s\ne 0$, $\coreMap_s(\genusThreeSurface)\subset P^*\times P^*$.  We now restrict further and assume $s>0$.  Then, for all $0<s<\ep$, the homotopy class of $\coreMap_s$, and hence the homology class $\coreMap_s([\genusThreeSurface])\in H_2(P^*\times P^*)$ is independent of $s$ and hence well-defined (up to sign, which can be fixed by orienting $\genusThreeSurface$).

\begin{lemma}
 \label{nis}  There exist four integers $n_0, n_1, n_2, n_3$ so that for all $0<s<\ep$,
 $$[\coreMap_s(\genusThreeSurface)]=n_0(\l_0\times \l_0)+n_1(\l_1\times \l_1)+n_2(\l_2\times \l_2)+n_3(\l_3\times \l_3) \text{ in } H_ 2(P^*\times P^*). $$
\end{lemma}
\begin{proof}
Given a small  neighborhood $U(\triangle_P )$ of the diagonal $\triangle_P  \subset P\times P$,  $$\coreMap_s(\genusThreeSurface)\subset U(\triangle_P )\cap (P^*\times P^*).$$
Such a neighborhood deformation retracts to the union of the diagonal with the corners cut off and the four linking tori 
$$ U(\triangle_P ) \simeq \triangle^\delta_P \cup \{\l_i\times \l_i ~|~ i=0,1,2,3\},$$ 
such that $\triangle^\delta_P \cap (\l_i\times \l_i) $ is the $(1,1)$ curve in the linking torus $\l_i\times \l_i$.  The lemma follows.  
\end{proof}

We introduce the streamlined notation:
\begin{equation}\label{dotE}\cdot_\genusThreeSurface:H_1(P,C)\times H_1(P,C)\to \ZZ, ~A\cdot_\genusThreeSurface B\coloneqq  (A\times B)\cdot \coreMap_s(\genusThreeSurface).\end{equation}
Since $\{A_i\times A_j\}$ forms a dual basis to $\{\l_i\times \l_j\}$,  Proposition~\ref{homology} and Lemma~\ref{nis} imply that
 \begin{equation}
\label{bilinear}
A_i\cdot_\genusThreeSurface A_j=\begin{cases}   n_3&\text{ if } i\ne j,\\
                                                  n_i+ n_3&\text{ if } i= j.\end{cases}
\end{equation}
In particular, the bilinear form $\cdot_\genusThreeSurface$ is symmetric.
\subsection{Intersections and perturbed character varieties}

Given a point $(\rho^\rout,\rho^\rin)\in P\times P$, the preimage $\coreMap_s^{-1}(\rho^\rout,\rho^\rin)\subset \genusThreeSurface$ is analytically isomorphic to the subset of conjugacy classes of  traceless perturbed representations $\rho\in \NAT_s$ whose restrictions to the two boundaries equal $ \rho^\rout$ and $\rho^\rin$.

\begin{lemma}\label{lem8.2} The surface $\genusThreeSurface$ may be oriented so that 
 $n_3=A_i\cdot_\genusThreeSurface A_j=1$ for $i\ne j$.
\end{lemma}
\begin{proof}
Figure~\ref{A0XA1fig} depicts a tangle decomposition  of the unknot equipped with an earring  along a pair of parallel 4-punctured 2-spheres. The outer tangle $(D^3, \link^\rout)$, as well as the inner tangle $(D^3,\link^\rin)$ are trivial.   The earring tangle lies between them.  

\begin{figure}[ht] 
\begin{center}
\includegraphics[angle=0,origin=c,width=2in]{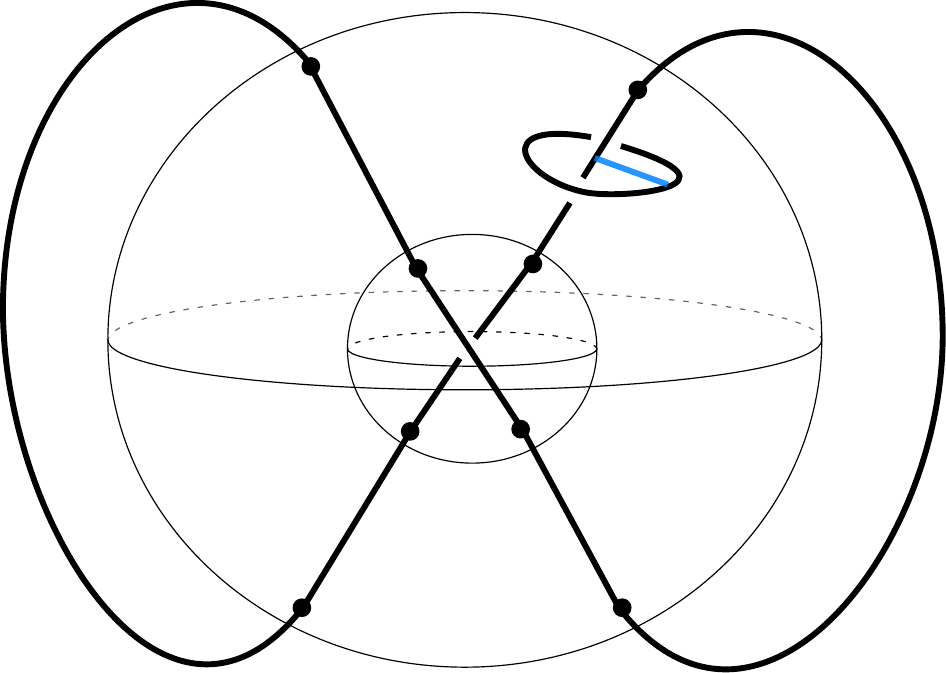}
\caption{ \label{A0XA1fig}}
\end{center}
\end{figure} 

The outer tangle is trivial, and its traceless character variety is an arc (see~\cite{HHK1}), whose restriction to the pillowcase is given (in the coordinates used in the present article) by 
\begin{align*}
\alpha^\rout:[0,\pi]& \to P\\ 
t &\mapsto [t,0]
\end{align*}
Indeed, one sees from Figure~\ref{A0XA1fig} that in the outer tangle, $a=f$, so that $\gamma=0$. Similarly the inner tangle is trivial, and from Figure~\ref{A0XA1fig} one sees that $b=f$ and hence $\gamma=\theta$.  Therefore its traceless character variety is the arc
\begin{align*}
\alpha^\rin:[0,\pi]& \to P\\ 
t &\mapsto [t,t]
\end{align*}
Clearly $[\alpha^\rout]=A_0$ and $[\alpha^\rin]=A_1$ in $H_1(P,C)$.  Hence
$$A_0\cdot_\genusThreeSurface A_1=
\alpha^\rout\cdot_\genusThreeSurface \alpha^\rin.$$

As noted by Kronheimer-Mrowka~\cite[Proposition~4.1]{KM},  the {\em unperturbed}  traceless character variety of the unknot with an earring is a single, regular point (which generates the reduced instanton homology of the unknot $I^\nat(U)\cong\ZZ$). It follows that the same is true  for small holonomy perturbations.  Thus for all small enough $s$,  the perturbed traceless variety of the unknot with an earring, perturbed along the two curves $Q_p$ and $Q_q$,  is again a single, regular point.  One concludes that for small $s>0$, $\coreMap_s^{-1}(\alpha^\rout\times \alpha^\rin)$ is one regular point and, in particular,  
$\coreMap_s$ is transverse to $\alpha^\rout\times \alpha^\rin$.

 Hence 
$$A_0\cdot_\genusThreeSurface A_1 = (A_0\times A_1)\cdot \coreMap_s([\genusThreeSurface])=\pm1.$$
Since $\genusThreeSurface$ is connected,   one of the two orientations satisfies $A_0\cdot_\genusThreeSurface A_1=1$ and hence $n_3=1$. Equation~\eqref{bilinear} now shows that $A_i\cdot_\genusThreeSurface A_j=1$ for all $i\ne j$.
\end{proof}

The following lemma (at least when $i=0$ or $1$) is a consequence of the fact that for the 2-component unlink $U_2$, the reduced singular instanton homology $I^\nat(U_2)\cong\ZZ^2$ has two generators with even relative grading.  There is a slight wrinkle in the case of $A_2\cdot_\genusThreeSurface A_2$, and so we present a more elementary argument, which leverages the fact that the Hopf link, in contrast to $U_2$, has a regular unperturbed traceless moduli space consisting of two points, distinguished by whether the two meridians are equal or negatives.  These two points define generators with even relative grading in instanton homology.  This permits us to argue as in the proof of Lemma~\ref{lem8.2} for the unknot.
  
\begin{lemma}\label{lem8.3} With the orientation of $\genusThreeSurface$ as above,
 $A_i\cdot_\genusThreeSurface A_i= 2$ for $i=0,1,2$. Equivalently, we have $n_0=n_1=n_2=1$.
\end{lemma}

 \begin{proof} 
 Figure~\ref{lem82fig} illustrates several tangle decompositions of the Hopf link into three  pieces: the earring in the middle, and trivial tangles  outside and inside.  The rightmost link also has two additional $w_2$-arcs.
 
 \begin{figure}[ht] 
\begin{center}
\includegraphics[angle=0,origin=c,width=6in]{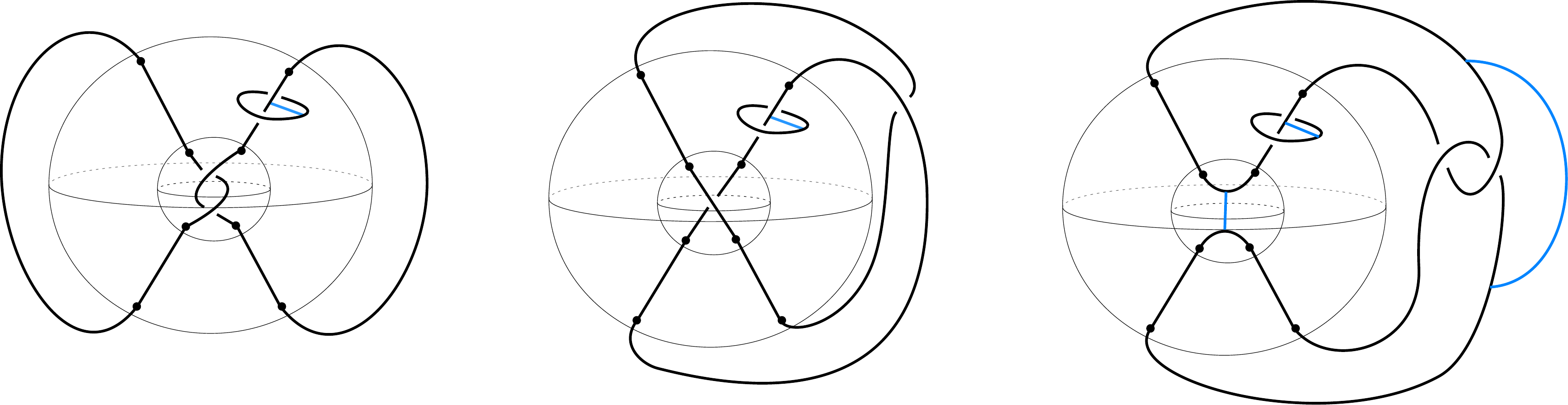}
\caption{ \label{lem82fig}}
\end{center}
\end{figure} 
\noindent In the leftmost decomposition, the outside tangle has traceless character variety the embedded arc corresponding to $\theta=0$, since $a=f$.  In other words, its traceless character variety is the arc
\begin{align*}
\alpha^\rout:[0,\pi]& \to P\\ 
t &\mapsto [t,0]
\end{align*}
Clearly $[\alpha^\rout]=A_0\in H_2(P,C)$. 

The inner tangle has traceless character variety corresponding to the arc
\begin{align*}
\alpha^\rin:[0,\pi]& \to P\\ 
t &\mapsto  [t,-2t]
\end{align*}
 This is an arc in the pillowcase of slope $-2$ which joins the bottom corners $[0,0]$ and $[\pi,0]$.

Hence 
 $$[\alpha^\rin]=A_0\in H_2(P,C).$$
 Since the unperturbed traceless character variety of the Hopf link consists of precisely two regular points,
 $$ A_0\cdot_\genusThreeSurface A_0  =\alpha^\rout\cdot_\genusThreeSurface \alpha^\rin=\pm 2 \text{ or } 0.$$

 Taubes~\cite{Taubes} showed that the relative $\ZZ/2$ grading in instanton homology coincides with the relative intersection number of character varieties (for further details in the present context see~\cite{MR2063666, HHK2}.)  For the Hopf link, this relative grading is even, as one can easily compute, e.g., using Kronheimer-Mrowka's skein exact sequences~\cite{KM}, or by more elementary means (e.g., ~\cite{HHK1}).
 Hence $$A_0\cdot_\genusThreeSurface A_0 =\pm 2.$$

 Similarly, the decomposition of the Hopf link illustrated in the center of Figure~\ref{lem82fig} gives rise to the two embedded arcs
 \begin{align*}
\alpha^\rout:[0,\pi]& \to P & \alpha^\rin:[0,\pi]& \to P \\ 
t &\mapsto [t, -t] & t &\mapsto [t,t]
\end{align*}
 Then $\alpha^\rin=A_1$ and $\alpha^\rout$ is homologous to $A_1$. Thus, just as above,  
 $$A_1\cdot_\genusThreeSurface A_1 =\pm 2.$$

 One argues similarly for the rightmost tangle, obtaining two regular generators for the intersection $A_2\cdot_\genusThreeSurface A_2$, and thus
 $$A_2\cdot_\genusThreeSurface A_2 = \pm 2 \text{ or } 0.$$
However, the presence of the two extra $w_2$-arcs makes this case different from the Hopf link with an earring. We may nevertheless reduce to the previous case as follows. In Figure~\ref{lem82bfig}, on the left we redraw the link with three $w_2$-arcs from Figure~\ref{lem82fig}. In the middle of Figure~\ref{lem82bfig} a different decomposition of this link is shown, and redrawn on the right, after an isotopy. Now 
  \begin{figure}[ht] 
\begin{center}
\labellist 
\pinlabel $\substack{\text{different} \\ \text{decomposition}}$ at 700 193
\endlabellist
\includegraphics[width=6in]{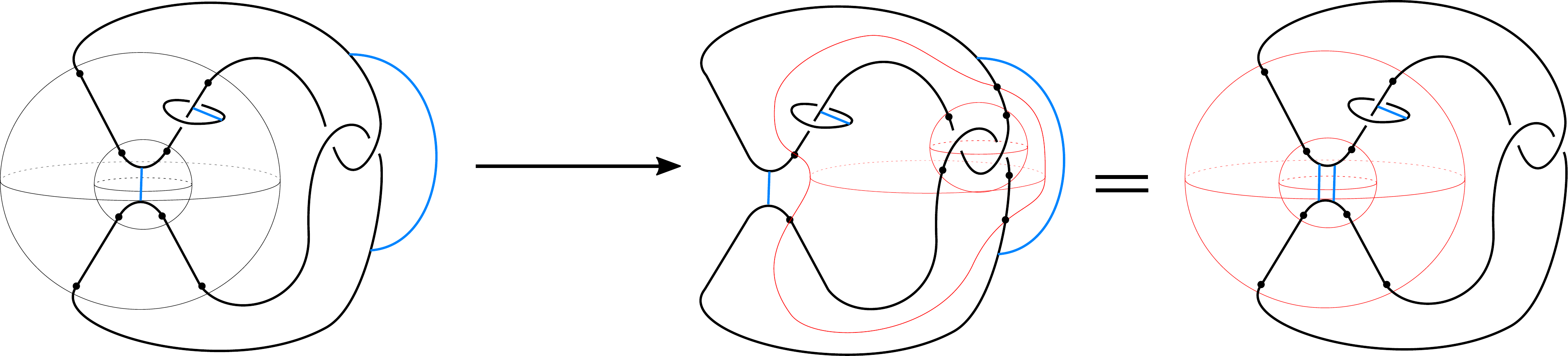}
\caption{ \label{lem82bfig}}
\end{center}
\end{figure} 
the inner tangle contains two parallel $w_2$-arcs, and its traceless character variety is identical to the corresponding tangle with no $w_2$-arcs.     Hence the intersection number is unchanged by  removing the two arcs, which is again a tangle decomposition of the Hopf link with one earring. Therefore, 
$$A_2\cdot_\genusThreeSurface A_2=\pm 2.$$  
 
   \begin{figure}[ht] 
\begin{center}
\includegraphics[angle=0,origin=c,width=6in]{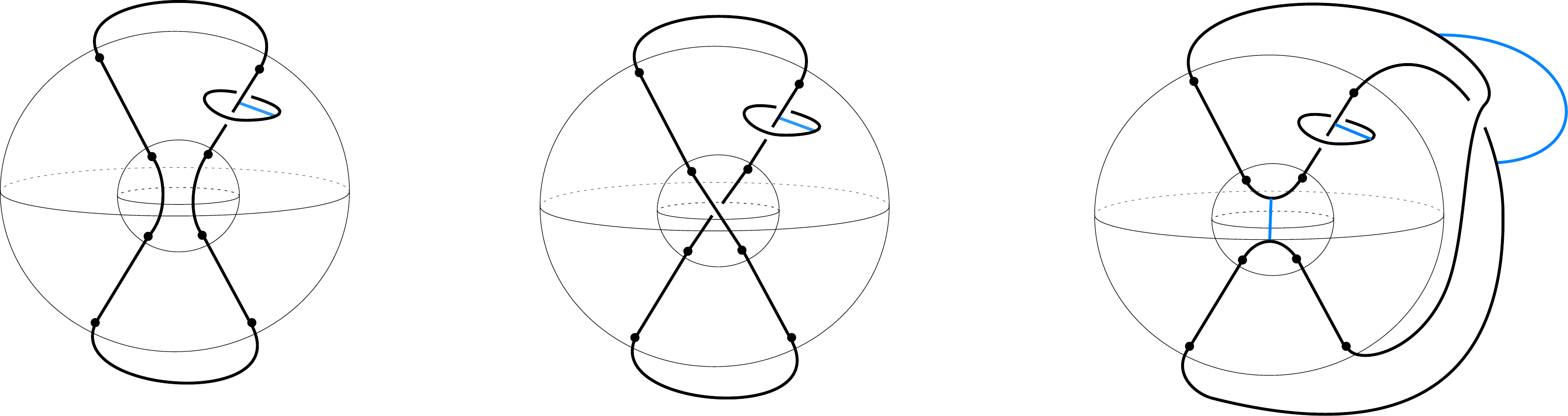}
\caption{} \label{lem82cfig}
\end{center}
\end{figure} 

At last, we prove that the intersection number is $\mathbf{+}2$. The three links illustrated in Figure~\ref{lem82cfig} represent (from left to right, the first factor always corresponding to the outer tangle) $ (A_0-A_1)\cdot_\genusThreeSurface A_0$, $A_1 \cdot_\genusThreeSurface (A_1-A_0)$, and $(A_2-A_0)\cdot_\genusThreeSurface A_2$. These all represent the unknot with an earring (after canceling two of the arcs in the rightmost diagram, as before) and hence  these products all equal $\pm 1$.

Orient $\genusThreeSurface$ as in Lemma~\ref{lem8.2}  so that $A_1\cdot_\genusThreeSurface A_0=1$.   Then
$$\pm 1=(A_0-A_1)\cdot_\genusThreeSurface A_0=A_0\cdot_\genusThreeSurface A_0 -1=\pm 2-1,$$ and hence $A_0\cdot_\genusThreeSurface A_0=2$.   Similarly $A_1\cdot_\genusThreeSurface A_1=2$ and $A_2\cdot_\genusThreeSurface A_2=2$, completing the proof. 
\end{proof}

\begin{corollary} The genus 3 surface $\genusThreeSurface$ may be oriented so that 
 \begin{align*}[\coreMap_s(\genusThreeSurface)]&=\l_0\times \l_0+\l_1\times \l_1+\l_2\times \l_2+\l_3\times \l_3\\
 &=\l_0\times \l_0+\l_1\times \l_1+\l_2\times \l_2+ (\l_0+\l_1+\l_2)\times(\l_0+\l_1+\l_2)\end{align*}
 in $H_2(P^*\times P^*)$.\qed
\end{corollary}

From now on we orient $\genusThreeSurface$ as in the corollary above.

\subsection{\texorpdfstring{The homotopy class of $\coreMap_s$}{The homotopy class of u-s}}

We next construct explicitly a family of continuous maps
$$\modelMap{\delta}:\genusThreeSurface   \to (P^*)^-\times P^*, ~0<\delta<\tfrac{\pi}{4}$$
which satisfy: 
\begin{enumerate}
\item \label{cond1} $\modelMap{\delta}$ agrees with $\coreMap_0$ on $\genusThreeSurface^\delta$, 
\item \label{cond2} the image of $\modelMap{\delta}$ lies in $P^{\delta/2}\times P^{\delta/2}$ and
\item \label{cond3} There exists an $\epsilon_1>0$ depending on $\delta$,  so that  if    $0<s<\ep_1$, then $\modelMap{\delta}$ and $\coreMap_s$ are  homotopic by a homotopy whose restriction to $\Sigma^\delta$  moves points a distance less than $\delta/100$.
\end{enumerate}
with $\genusThreeSurface^\delta, P^\delta$ as in Equation (\ref{deltas}).  Doing so proves the final statement in Theorem \ref{thmA}.  The map  $\modelMap{\delta}$ induces a correspondence from a neighborhood of each corner of the first pillowcase to a neighborhood of that corner in the second pillowcase, illustrated in Figure~\ref{correspondtwfig}; see also Figure~\ref{fig:description_of_the_map}.

\subsubsection{The local construction} 
Let $h:[-1,1]\to [0,2\pi]$ be a smooth non-decreasing function satisfying
$h(t)=0$ for $t\in[-1,-\tfrac 1 2]$, $h(t)=2\pi$ for  $t\in[ \tfrac 1 2, 1]$, $h(0)=\pi$ and $h'(0)>0$.  Let $g:[-1,1]\to [\tfrac 1 2,1]$ be a smooth function  so that 
$g(t)=|t|$ when $|t|\ge \tfrac 2 3  $, $g(0)=\frac 1 2$, and $g$ has one critical point, a minimum, at $0$.  The function $h$ induces a model Dehn twist 
\[
\tau_{S^1}\colon [-1,1]\times S^1 \to [-1,1]\times S^1
\]
given by $\tau_{S^1} (t, e^{\nu\bbi}) = (t, e^{(\nu+h(t) )\bbi})$.  Combined with $g$, one gets a ``folded twist correspondence''  
\[
\DehnTwistMap:[-1,1]\times S^1 \to \CC^*\times \CC^*,~\DehnTwistMap(t, e^{\nu\bbi})=g(t)(e^{\nu\bbi}, e^{(\nu+h(t))\bbi}).
 \] 
  A direct calculation shows $\DehnTwistMap$ is a Lagrangian immersion.   
 
A homotopy rel endpoints of $g$ to the constant map (at $1$) determines a homotopy  rel boundary  of  $\DehnTwistMap$ to $(t, e^{\nu\bbi})\mapsto (e^{\nu\bbi}, e^{(\nu+h(t))\bbi})$. Since $h(1)=h(0)+ 2\pi$, this latter map factors through a homeomorphism with the linking torus:
 $$S^1\times S^1\cong [-1,1]\times S^1/_\sim\to S^1\times S^1\subset \CC^2, ~ (-1,e^{\nu\bbi})\sim  (1,e^{\nu\bbi}).
$$

\subsubsection{\texorpdfstring{The map $\modelMap{\delta}$}{The map v-delta}}\label{sec:v_d}

Fix $\delta>0$.
Define $\modelMap{\delta}$ on $\genusThreeSurface^\delta$ to equal $\coreMap_0$.  In symbols:
$$\modelMap{\delta}(\sigma_\pm[\gamma,\theta])=([\gamma,\theta], [\gamma,\theta]) \in \triangle_P^\delta \subset P^\delta \times P^\delta.$$
It remains to define $\modelMap{\delta}$ on $\genusThreeSurface\setminus \genusThreeSurface^\delta$, a disjoint union of four annuli forming the tubular neighborhoods of the four circles.

 Orient $\genusThreeSurface$ and, identify each of these annuli    with $[-1,1]\times S^1$ in such a way 
that $\sigma_\pm(\delta\cos\nu, \delta\sin\nu)$ is identified with 
$(\pm 1, e^{\nu\bbi})$.  Now define $\modelMap{\delta}$ in each annulus to be the map $\DehnTwistMap$ defined above, scaled down by $\delta$  so that its image lies in
 $\{(z,w)~|~ \tfrac\delta 2\leq |z|, |w|\leq \delta\}$.

 
The family $\modelMap{\delta}:\genusThreeSurface\to P^*\times P^*$ varies smoothly in $\delta$, and is   $\delta$ close to $\coreMap_0$ (as a map to $P\times P$) in the $C^0$ norm. It satisfies conditions~\eqref{cond1} and~\eqref{cond2} above, and clearly 
$$\lim_{\delta\to 0} \modelMap{\delta}=\coreMap_0$$
in the $C^0$ norm.
By appropriately choosing a  positive or negative Dehn twist at each corner, one can ensure that 
  $$\modelMap{\delta} ([\genusThreeSurface])= \l_0\times \l_0+\l_1\times \l_1+\l_2\times \l_2+\l_3\times \l_3 \text { in  }H_2(P^*\times P^*).$$

\begin{figure}[ht] 
\labellist 
\pinlabel $\DehnTwistMap$ at 450 603
\pinlabel $\text{\small vertical fold}$ at 280 403
\pinlabel $\text{\small vertical fold}$ at 855 403
\endlabellist
\centering
\includegraphics[scale=0.25]{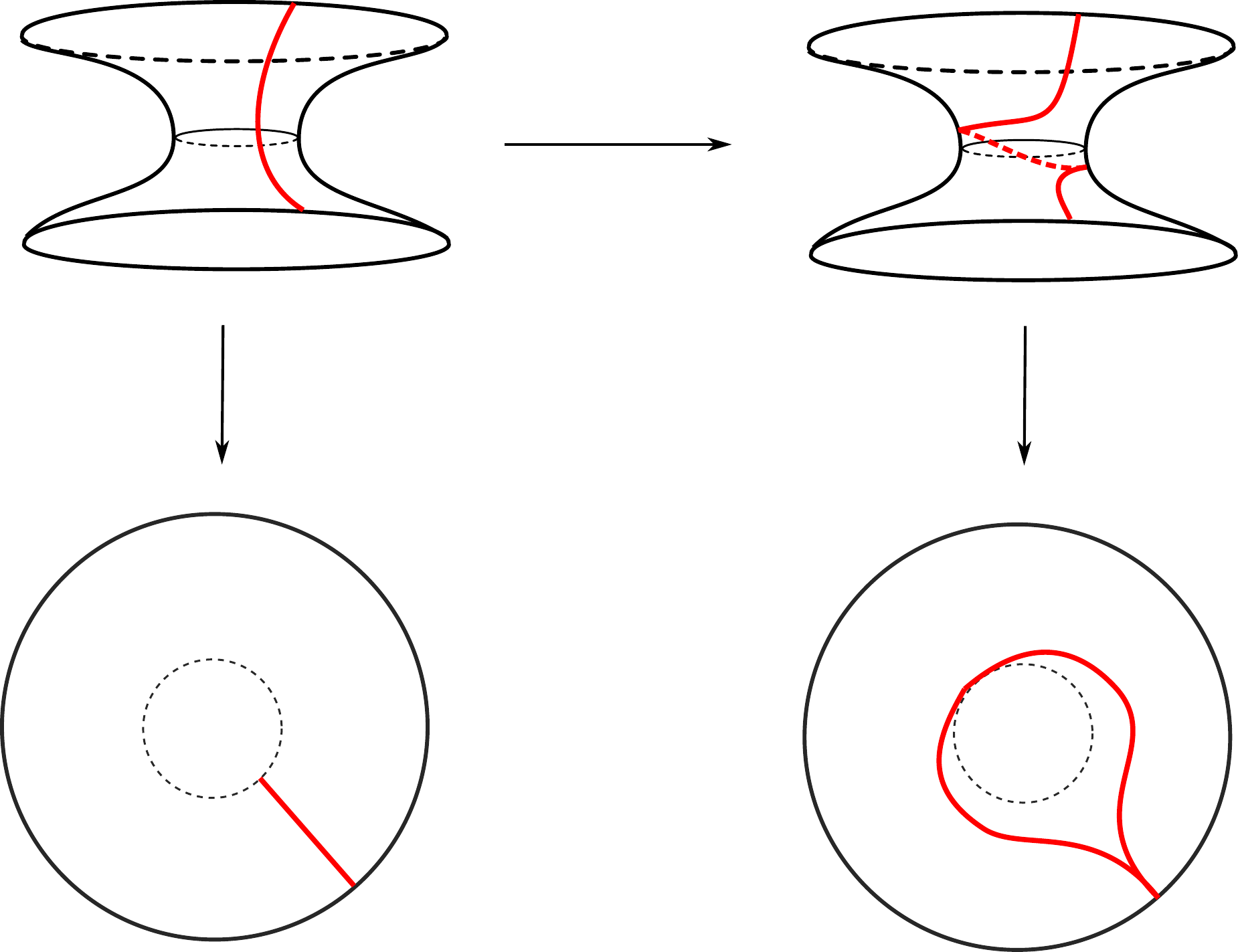}
\caption{The correspondence induced by $v_\delta$. }\label{correspondtwfig}
\end{figure}

 We  address the condition~\eqref{cond3} in the following theorem.
\begin{theorem}\label{thm:close}
 \label{main3} For any $0<\delta<\pi/4$ there exists $\epsilon>0$ such that for all $0<s<\epsilon$  
 $$ \text{the maps }
 \coreMap_s, \modelMap{\delta}:\genusThreeSurface\to P^*\times P^*$$
 are  homotopic by a homotopy whose restriction to $\Sigma^\delta$  moves points at distance less than $\delta/100$.
\end{theorem}
\begin{proof}  Fix $0<\delta<\tfrac \pi 4$.  Choose $\ep >0$ so that each $0<s\leq \ep $ satisfies the  conclusions  of Theorem~\ref{thm1}.

Write  $\genusThreeSurface=\genusThreeSurface^\delta\cup Z^\delta$, where $Z^\delta$ is the disjoint union of four annuli, each a tubular neighborhood of the preimage of a corner by $\coreMap_0$.  Since $ \coreMap_s\to \coreMap_0$ as maps $\genusThreeSurface\to P\times P$ and  the restrictions of $\coreMap_0$ and $\modelMap{\delta}$ to (the compact surface) $\genusThreeSurface^\delta$ coincide, it follows that
the restriction of the family $\coreMap_s, s\in [0,\epsilon]$ to $\genusThreeSurface^\delta$ gives a homotopy  
$$H_s:\coreMap_\ep|_{\genusThreeSurface^\delta}\sim \modelMap{\delta}|_{\genusThreeSurface^\delta}:\genusThreeSurface^\delta\to P^*\times P^*, ~s\in[0,\ep].
$$
By making $\ep$ smaller if needed, we may assume this homotopy moves points in $\genusThreeSurface^\delta$ at most at distance $\delta/101$.

 It remains to show that 
$H_s$ extends to a homotopy of  $\genusThreeSurface$, that is, over the four annuli
$$Z^\delta=\sqcup_{i=0}^3 Z^\delta_i.$$

Lemma \ref{notat} implies that at the $i$th corner, 
 the two boundary components of $Z^\delta_i$ are sent by $\coreMap_0$ (and hence also $\modelMap{\delta}$) to the diagonal  in the linking torus $$\triangle_{\l_i}\subset \l_i\times \l_i\subset P^\delta\times P^\delta,$$
 with $\l_i$ as illustrated in Figure~\ref{pillowgensfig}.
Use restrictions to the three annuli to construct a map with domain a torus:
 $$k_i: \partial(Z^\delta_i\times [0,\ep])=\big(Z^\delta_i\times\{0\} \big)\cup \big(\partial(Z^\delta_i)\times [0,\ep]\big)\cup \big( Z^\delta_i\times\{\epsilon\}\big)\xrightarrow{\modelMap{\delta}\cup H_s\cup 
 \coreMap_s} P^*\times P^*.$$

 The restriction of the  homotopy $H_s$ to the boundary of $Z_i^\delta$  has image in a $\delta/50$ tubular neighborhood of $\l_i\times \l_i$.  Since $\coreMap_0$ takes $Z_i^\delta$ into a $\delta$ neighborhood of $(\l_i,\l_i)$, by making $\ep>0$ smaller if needed, we may assume that $\coreMap_\ep$ sends $Z_i^\delta$ into the $2\delta$ deleted neighborhood:
 $$W_{2\delta}^*\coloneqq  \big(U_{2\delta}(\l_i)\setminus\{\l_i\}\big)\times \big(U_{2\delta}(\l_i)\setminus\{\l_i\}\big) .$$

 The homotopy $H_s$ can be extended over $\genusThreeSurface$ if, for each $i$,  the map $k_i$ is nullhomotopic.    The inclusion of the linking torus
 $\l_i\times \l_i\subset W_{2\delta}^*$ is a deformation retract, since $\delta<\tfrac\pi 4$.  By construction, the map $k_i$ takes the two circles $\partial Z_i^\delta\times\{0\}$ to an essential curve, namely the $(1,1)$ curve on the linking torus $\l_i\times \l_i$. Thus $k_i$ is nullhomotopic if it  has degree 0, which holds if and only if the intersection numbers    $\coreMap_\ep\cdot (A_i\times A_i)$ and $\modelMap{\delta}\cdot (A_i\times A_i)$ coincide (with $A_i$ illustrated in  Figure~\ref{pillowgensfig}.)
 
 By construction,  $[\modelMap{\delta}]$ and $ [\coreMap_\ep]$ both equal $\l_0\times \l_0 +\l_1\times \l_1+\l_2\times \l_2+\l_3\times \l_3$ in $H_2(P^*\times P^*)$, and so the intersection numbers coincide, and hence $k_i$ is nullhomotopic, proving the theorem when $s=\ep.$  This implies the statement for any $0<s<\ep$,  completing the argument.

 Lastly, to complete the  proof of Theorem~\ref{thm:close} and Theorem~\ref{thmA},   consider arcs $A_i$ from Figure~\ref{pillowgensfig}, and their images, figure eight Lagrangians $L^8_i$, under the correspondence $\modelMap{\delta}$. Intersections of those arcs and figure eights are 
 $$A_i \cdot L^8_j = \begin{cases} \pm 1&\text{ if } i \ne j,\\
 \pm 2 & \text{ if } i= j.\end{cases}$$
 depending on the sign of the Dehn twist in the definition of $\modelMap{\delta}$. We now choose the sign of the Dehn twists so that the intersections are $+1$ and $+2$. This proves that the homology class $[\modelMap{\delta}(\genusThreeSurface)]$ satisfies $n_0=n_1=n_2=n_3=1$ (in the terminology of Lemma~\ref{nis}).

 \end{proof}
 
 Theorem~\ref{thm:close} does {\em not} assert that $\coreMap_s$ and $\modelMap{\delta}$ are {\em Lagrangian} (or Hamiltonian) regular homotopic.   Smale-Hirsch theory \cite{MR119214}   permits us to alter  $\modelMap{\delta}$ near some point in $\genusThreeSurface$ so that these two Lagrangian immersions become regularly homotopic, but more subtle methods, such as the Gromov-Lees theorem \cite{MR410764},  may be  required to determine if they are Lagrangian regular homotopic.   
 
 However,  neither $\coreMap_s$ nor $\modelMap{\delta}$ are Lagrangian regular homotopic to an embedding.  This is because the restriction map $H^2(P\times P)\to H^2(P^*\times P^*)$ is zero and hence the Euler classes of the normal bundles of these immersions are both zero.  But the Darboux-Weinstein theorem \cite{MR1698616} implies that the normal bundle of an embedded Lagrangian genus three surface in a 4-manifold has Euler class $-4$.

\section{\texorpdfstring{The action by composition of the Lagrangian correspondences $\coreMap_s$ and $\modelMap{\delta}$ on immersed curves in the pillowcase}{The action by composition of the Lagrangian correspondences u-s and v-delta on immersed curves in the pillowcase}
}\label{sec:corr}

Composition with the Lagrangian immersion  $\coreMap_s:\Sigma\to P^*\times P^*$ induces  functions $\Lag(P^*)^{\pitchfork {u_s}}  \to \Lag(P^*)$, where $\Lag(P^*)$ denotes the set of 
Lagrangian immersions in $P$ and $\Lag(P^*)^{\pitchfork {u_s}} \subset \Lag(P^*)$ denotes the subset of those immersions which are composable with $\coreMap_s$. The latter immersions are of the form $i:S\to P$ such that $i\times \id_P:S\times P\to P\times P$ is transverse to 
$\coreMap_s$ (see Definition \ref{def:compos} and Equation   (\ref{corronlag})).
A similar comment applies to $\modelMap{\delta}:\Sigma\to P^*\times P^*$.  In this section we describe these functions when restricted to loop-type and arc-type Lagrangians for small $s,\delta$.

\subsection{The action on compact Lagrangians}

Part  (3)    of Lemma~\ref{notat} implies that for small enough $s\ne0$ any immersed Lagrangian circle  $L:S^1\to P^{\delta}$ has two disjoint cross sections $L_\pm$ which are each mapped by $\coreMap_s^\rin$ close to   $L$: 
 $$\begin{tikzcd}
&  \genusThreeSurface \arrow[d,"\coreMap_s^{\rout}"]\arrow[r,"\coreMap^\rin_s"]&P^{\frac{1}{2}\delta}\\
S^1\arrow[r,"L"] \arrow[ru, "L_\pm"]                                & P^{\delta}& \end{tikzcd}$$
The range of $\coreMap_s^\rin$ is taken to be $P^{\frac{1}{2}\delta}$ in this diagram because  applying the correspondence $\genusThreeSurface_s$ might move $L$ out of $P^\delta$, but only slightly, and hence its image is inside $P^{\frac{1}{2}\delta}$ for small enough non-zero $s$.  
 
In particular, every $\delta$-loop-type Lagrangian (see Definition~\ref{loops/arcs}) is composable with $u_s$ for all small enough positive $s$. 
This determines the {\em Lagrangian  correspondence on loops}  induced by the earring cobordism,  via its traceless character variety  $\genusThreeSurface_s=\NAT_s$ and the restriction $\NAT_s\to P^*\times P^*$:

\begin{equation}
\begin{split}
\label{loops} (-\circ \NAT_s) : \allLoopImmersions{\delta}&\to\allLoopImmersions{\frac{1}{2}\delta}
, \\
L &\mapsto L\circ \NAT_s   
\end{split}
\end{equation}

We summarize this in the following theorem which refers to the $\delta$-loop-type Lagrangians of Definition~\ref{loops/arcs}.

\begin{theorem} \label{looptype} For small enough $s \ne 0$, the correspondence $\NAT_s \cong \genusThreeSurface_s\to P^\rout\times P^\rin$ associated to the earring cobordism induces the function 
$$ (- \circ \NAT_s ) : \allLoopImmersionsout{\delta}\to\allLoopImmersionsin{\frac{1}{2}\delta}$$
which sends a loop-type immersed Lagrangian to two nearby  isotopic copies of $L$.
  \qed
\end{theorem}

Theorem~\ref{looptype}  implies the first assertion  of Theorem~\ref{thmB}, part (4), namely the doubling action of the Lagrangian correspondence $\coreMap_s(\genusThreeSurface)$ on loop-type Lagrangians. By letting $\delta$  and $s$     approach zero, the earring cobordism induces the trivial doubling operation on the compact objects of the Fukaya category  of the pillowcase.   Notice that Theorem~\ref{looptype} applies equally to $\modelMap{\delta}$, which, away from the corners, is also a trivial 2-fold covering space.

\subsection{The action on  arc-type Lagrangians}  \label{arc}

Here we address the second half of part (4) of Theorem~\ref{thmB}. Suppose $A:[-\tfrac 1 2, \tfrac 1 2]\to  P$ is a $\delta$-arc-type Lagrangian.  The same argument as for the  $\delta$-loop-types shows that outside of the neighborhood $U_\delta(C)$, the correspondences $\coreMap_s$  and $\modelMap{\delta}$ double the part of $A$ that lies in $P^\delta$.  Hence, to complete the description of the action of $\coreMap_s$ on $\allArcImmersions{\delta}$, it remains to describe  what happens in each neighborhood of a corner.

\medskip

 In the following definition, $\mathbb D$ denotes the unit disk in the complex plane, $\alpha_+\subset \mathbb D$ denotes the oriented arc on the real axis  from $ \tfrac 1 2 $ to $ 1$, $\alpha_-\subset \mathbb D$ denotes the oriented arc on the real axis  from $ -\tfrac 1 2 $ to $ -1$, and $\beta$ denotes the intersection of the imaginary axis with $\mathbb D$.

\begin{definition}\label{HF8} A {\em homology figure eight curve} in the twice-punctured 2-disk
 is an immersion $F:L\to  \mathbb D\setminus \{-\tfrac 1 2,\tfrac 1 2\}$ of an oriented compact 1-manifold without boundary which intersects each arc $\alpha_\pm$ transversely  in one point with  intersection number $\pm 1$, and intersects $\beta$ transversely in two points (and hence with algebraic intersection number $0$). See Figure~\ref{HF8fig} for an illustration.

Given   an arc-type Lagrangian in the pillowcase $A:[-\tfrac 1 2,\tfrac 1 2]\to P$ joining corners $c_1,c_2$, a {\em homology figure eight curve supported near $A$} is a map $F:L\to P^*$ obtained as a composite of a homology figure eight curve in $\mathbb D\setminus\{-\tfrac 1 2,\tfrac 1 2\}$ with a continuous map
$(\mathbb D,-\tfrac 1 2,\tfrac 1 2)\to  (P, c_1,c_2)$ which satisfies:
\begin{enumerate}
\item immerses $\mathbb D\setminus \{-\tfrac 1 2,\tfrac 1 2\}$    into $P^*$;
\item its restriction to the line segment   $[-\tfrac 1 2,\tfrac 1 2]$ equals $A$.
\end{enumerate}
When $L$ is  connected, we call $F$ a {\em figure eight curve}.
\end{definition}

\begin{figure}[t] 
\labellist 
\pinlabel $\alpha_+$ at 315 76
\pinlabel $\alpha_-$ at 035 80
\pinlabel $\beta$ at 163 130
\endlabellist
\centering
\includegraphics[width=0.4\textwidth]{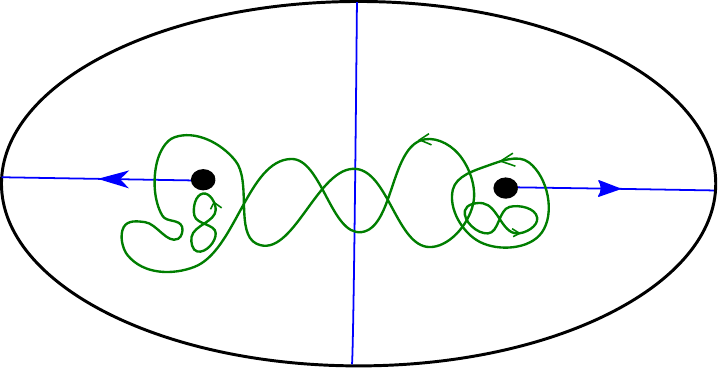}
\caption{A homology figure eight curve in the twice-punctured disk. }\label{HF8fig}
\end{figure} 
Illustrations  of two arc-type Lagrangians  and figure eight curves supported near them are given in Figure~\ref{example1fig}.
Figure~\ref{corrfig} is a computer plot of the action of the correspondence $\coreMap_s$ for $s=0.19$  on the three arcs corresponding to the three trivial tangles which appear in the unoriented skein triple.  Each of these arcs is transformed into a nearby figure eight curve.
 In Figure~\ref{corrfig}, the pillowcase is considered as a subset of $\RR^3$ via the embedding
 $i:P\subset \RR^3$ given by $[\gamma,\theta]\mapsto(\cos(\gamma), \cos(\theta), \cos(\gamma-
 \theta))$, with image $i(P)=\{(x,y,z)~|~ x^2+y^2+z^2-2xyz=1\}$.  The  red arc corresponds to the bottom edge $A_0$ in Figure~\ref{pillowgensfig}. The green arc corresponds to the diagonal edge $A_1$, and the blue arc corresponds to the right edge. This is the same triple illustrated in Figure  \ref{fig:basic_lagrangians}.
\begin{figure}[ht] 
\begin{center}
\includegraphics[angle=0,origin=c,width=3in]{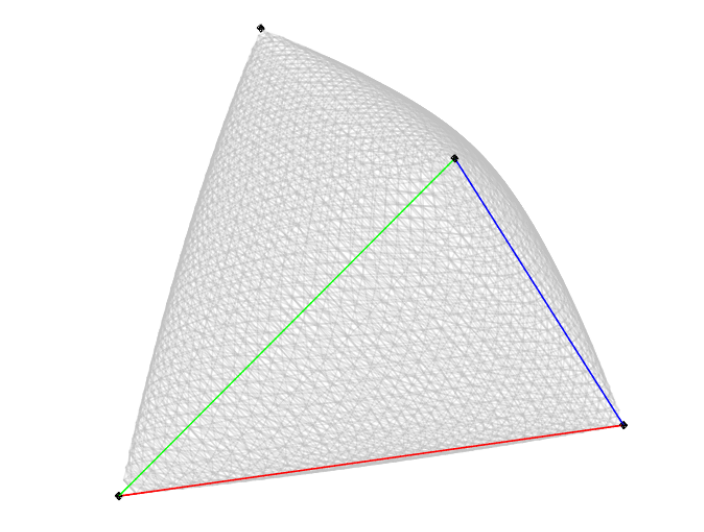}
\includegraphics[angle=0,origin=c,width=3.2in]{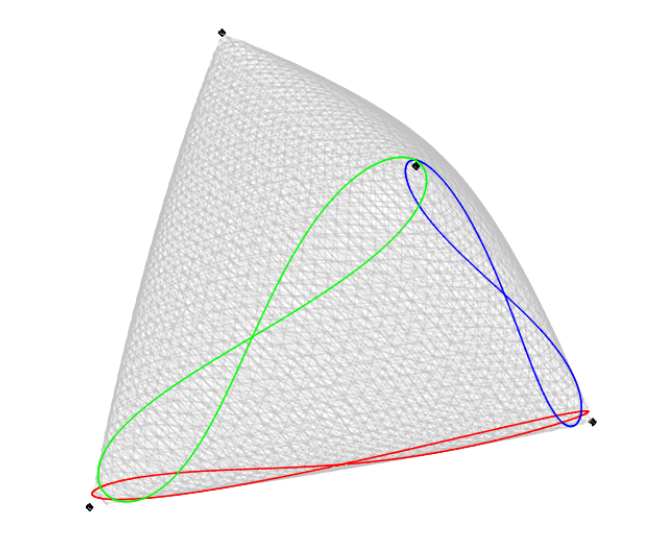}
\caption{} \label{corrfig}
\end{center}
\end{figure}

 \medskip

The correspondence induced by $\modelMap{\delta}$  near each corner is explicitly described by  Figure~\ref{correspondtwfig}.  
 Indeed, from that figure  one sees that a radial arc $A(t)$ with $A(0)\in\partial U_\delta(c)$ and $A(1)$ the  center (corner) $c$ is sent by the correspondence to a loop which wraps once around the corner $c$, both of whose endpoints are  sent to $A(0)$.   
Hence every $\delta$-arc-type Lagrangian $A$ is sent by the correspondence to a figure eight curve supported near $A$, composed of   four pieces: the two arcs which wrap around the corners, and the trivial double of the part of the Lagrangian that lies outside $U_\delta(C)$. This (connected) curve is uniquely oriented by requiring the algebraic intersection numbers with $\alpha_\pm$ to be $\pm 1$. Figure~\ref{vdeltacorfig} illustrates an example.
 
 \begin{figure}[hb] 
\labellist 
\pinlabel $(\modelMap{\delta})_*$ at 330 125
\endlabellist
\centering
\includegraphics[scale=0.5]{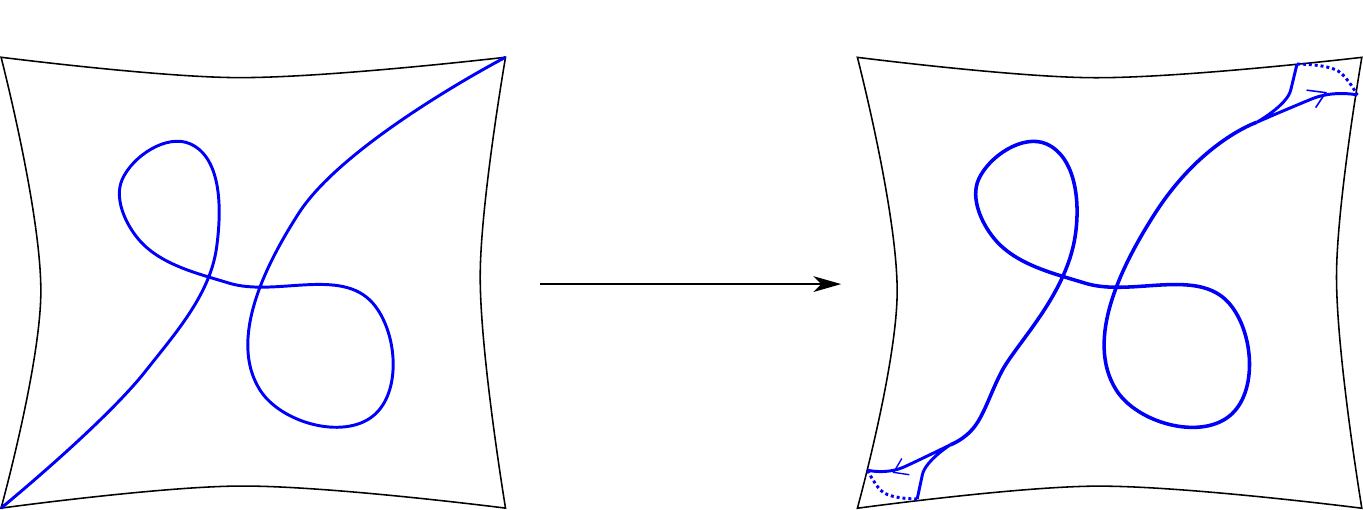}
\caption{The action of the correspondence $\modelMap{\delta}$. }\label{vdeltacorfig}
\end{figure}

  A precise description of the action of  $\coreMap_s$ near each corner is difficult to write down, and, in particular, we do not know if it sends every $\delta$-arc-type curve to a connected curve. Since $\coreMap_s$ is homotopic in $P^*\times P^*$ to $\modelMap{\delta}$,  a  homotopy transversality argument shows that if  $\coreMap_s$ is transverse to $A \times \id:[0,1]\times P^\rin \to P^\rout \times P^\rin$ for a a $\delta$-arc-type curve $A$, then   the image of    $A$ 
  by the action of $\coreMap_s$ is oriented.  The argument of Lemma \ref{lem8.2} shows that this image meets $\alpha_\pm$ transversely once,  and the fact that $r_s^\rout$ is a trivial 2-fold cover away from the corners (Lemma \ref{notat}) shows that it meets $\beta$ transversely in two points.  Hence $\coreMap_s$ sends  a $\delta$-arc-type curve to a homology figure eight curve.

Given any arc-type Lagrangian $A$ in the pillowcase, there exists arbitrarily small  holonomy perturbations supported in a collar of $(S^2,4)$ so that after perturbation, $A\times \id$ is transverse to  $\coreMap_s$ \cite{Herald, HK}.  This is summarized in the following proposition.

\begin{proposition} For $\delta>0$ small enough, the correspondence induced by $\modelMap{\delta}$ sends every $\delta$-arc-type Lagrangian $A$ to a figure eight curve supported near $A$.  For $s>0$ small enough, the correspondence induced by $\coreMap_s$ sends a $\delta$-arc-type Lagrangian $A$, with $A\times\id$ transverse  to $\coreMap_s$, to a homology figure eight curve supported near $A$. \qed
\end{proposition}

   For completeness, we state the following conjecture, which says that $\NAT_s$ takes arcs to figure eight curves, rather than (the weaker notion of) homology figure eight curves.

\begin{conjectureABC}  \label{truefig8} For $\delta$ small enough, the Lagrangian immersions $\coreMap_s$ and $\modelMap{\delta}$ are Hamiltonian regular homotopic in $(P^*)^-\times P^*$ for all $s>0$ small enough.  Moreover,  $\coreMap_s$ is transverse to every $\delta$-arc-type curve and sends each such curve to a figure eight curve.
\end{conjectureABC}

\section{Context and Implications} \label{sec:PillHom_n_quilts}

In the remainder of this article we  outline a few consequences of casting  Theorem~\ref{thmB} 
as a  calculation of an endomorphism in a Floer field theory.

We use freely the language of symplectic geometry and its related algebra, including the Lagrangian Floer complex,  Fukaya $A_\infty$ categories, quilted Floer homology, and bounding cochains. We refer the reader to the original sources~\cite{Flo-lagr,FOOO1,FOOO2,Seidel,WW}  and to the expository articles~\cite{Aur3, Smith}. 
For the benefit of the reader, we review these topics in Appendix~\ref{Appendix} at a level sufficient to follow the remainder of this article.

\subsection{Pillowcase homology}\label{sec:PilHom}

The articles~\cite{HHK1, HHK2,HHHK} lay out a program aimed at investigating the Atiyah-Floer  conjecture for reduced instanton homology $I^\natural(\link)$ for   links $\link$ in $S^3$, starting with a decomposition along a Conway sphere (to begin with, without an earring):
\begin{equation*} (S^3,\link)=(D^3,T_1) \cup_{(S^2,4)} (D^3,T_2) 
\quad \text{(See Figure~\ref{arc_svk})}
\end{equation*}
Such a decomposition gives rise  to the pair of Lagrangian immersions (omitting $D^3$ from the notation):
\begin{equation}\label{pillowlag}R_\pi(T_1) \looparrowright P^*\looparrowleft R_\pi(T_2).\end{equation}
  It is natural to consider the associated Lagrangian Floer theory of these two Lagrangians, in particular because $P^*$ is a surface and therefore some of the subtle arguments  about moduli spaces of $J$-holomorphic disks in its Lagrangian Floer theory  can be replaced by simpler differential topology arguments  \cite{Abouzaid_Fuk_surface,SRS}.

The presence of at least one arc-type  component in $R_\pi(T)$  for any 2-tangle $T$ in the 3-ball motivates the use of wrapped Floer homology $HW(R_\pi(T_1), R_\pi(T_2))$ inside $P^*$. 
See Section~\ref{sec:future_directions} for a further discussion about the wrapped Fukaya category of the pillowcase.

In \cite{HHK1, HHK2} the constructions were designed to match Kronheimer-Mrowka's reduced singular instanton homology $I^\natural(\link)$, and thus the (proper, exact) Lagrangian $R_\pi(T_2)$ was replaced by its compact image under the correspondence $\NAT_s$:
\begin{equation}
R^\nat_\pi(T_2)\coloneqq \NAT_s \circ R_\pi(T_2)\end{equation}
In other words, the tangle $T_2$ is modified by placing an earring on it, resulting in a traceless character variety consisting entirely of loop-type Lagrangians.
The associated Lagrangian Floer homology of the pair \begin{equation}\label{pillowlagnat}R_\pi(T_1) \looparrowright P^*\looparrowleft R^\nat_\pi(T_2)\end{equation}
in $P^*$ was given the moniker \emph{pillowcase homology} by M.~Hedden.

The Lagrangians $R_\pi(T_1)$ and $R^\nat_\pi(T_2)$ can  be made transverse by a small holonomy perturbation  (see e.g., ~\cite{Herald,HK}), resulting in a finite set of intersection points which corresponds bijectively to the generators of Kronheimer-Mrowka's instanton chain complex defining $I^\nat(\link)$ (see ~\cite{HHK1, HHK2}).  As such, this framework provides a concrete way to perturb the Chern-Simons functional used in the definition of $I^\nat(\link)$,  and thus single out transverse critical points. For nontrivial knots such perturbations are unavoidable, because the critical set of the relevant unperturbed Chern-Simons functional is not regular~\cite{KM}.

 \medskip

The articles~\cite{HHK1,HHK2} give much supporting evidence that pillowcase homology
is isomorphic to $I^\natural(\link)$. 
The $\ZZ/4$ gradings agree, and in many examples, including when $\link$ is a 2-bridge knot,   all differentials vanish in both complexes.  More complicated calculations for tangle decompositions of torus knots provide additional evidence \cite{FKP}.  The following discussion, centered on the example of the $(4,5)$ torus knot, brings bounding cochains into the picture. Conjecture~\ref{conj:bounding_cochains} below   below describes our  guess, using bounding cochains, on  how to construct a pillowcase homology equivalent to  $I^\natural(\link)$. 

\subsection{Dependence on the earring location}

To determine  what can be captured by pillowcase homology  requires an understanding of its dependence on a choice of decompositions of links along Conway spheres, or more generally,  surfaces with any number of marked points. In particular one wishes the Floer homology groups \begin{equation}\label{eq:earring_two_sides}\HF(R_\pi(T_1),R^\natural_\pi(T_2)) \quad \text{and} \quad \HF(R^\natural_\pi(T_1),R_\pi(T_2))\end{equation}
to be isomorphic.

We  turn our attention to the $(4,5)$ torus knot,    henceforth denoted by $K_{4,5}$. This is a knot for which Kronheimer-Mrowka's spectral sequence from Khovanov to instanton homology has a nontrivial higher differential~\cite{KM-higher-diff}. We highlight below another remarkable phenomenon.  For a particular tangle decomposition of $K_{4,5}$, the Floer homology groups~\eqref{eq:earring_two_sides} are not isomorphic; a certain differential, present in one of the chain complexes, is absent in the other. Following Bottman and Wehrheim~\cite{BW}, we then argue that figure eight bubbling (in the context of immersed Lagrangian correspondence $\NAT_s$) must occur to account for the lost differential, and we give an approximate description of the figure eight bubble and the corresponding bounding cochain.  

The following  informal exposition is meant to bolster the narrative and support the Bottman-Wehrheim  approach to immersed quilted Floer theory with an explicit flat moduli space example; a more formal treatment (with attention paid to analytical and $A_\infty$ aspects) is beyond the scope of the present article.

\subsection{
\texorpdfstring{Character varieties for a tangle decomposition of $K_{4,5}$.}
{Character varieties for a tangle decomposition of K(4,5).}
}\label{sec:fig8bubble} 

In \cite[Section~11.8]{HHK2}, a particular tangle decomposition $$(S^3,K_{4,5})=(D^3,T_1)\cup_{(S^2,4)}(D^3,T_2)$$   is described, with $T_2$ the trivial tangle. 
  This tangle decomposition is similar to the one illustrated for the trefoil knot in Figure~\ref{arc_svk}. The trivial tangle $(D^3,T_2)$ is the neighborhood of an arc on the (unknotted) torus containing $K_{4,5}$.

 For an appropriate holonomy perturbation $\pi$, the Lagrangian curve $R_\pi(T_1)$ is calculated    to be a disjoint union of a (red) arc-type Lagrangian and a (green) loop-type Lagrangian,  illustrated on the left of Figure~\ref{45torus2fig}. The curve $R(T_2)$ is the linear (blue) arc illustrated on the right. 
\footnote{In Figure~\ref{45torus2fig}, the pillowcase is stereographically projected to the plane, whereas in~\cite[Figure 23]{HHK2}, the pillowcase is depicted by its fundamental domain 
$0\leq \gamma\leq \pi, ~0\leq \theta\leq 2\pi$.  The two perspectives are identified in  Figure~\ref{fig:pillowcase_coordinates}.  }

\begin{figure}[t]
\centering
 \includegraphics[width=.7\textwidth]{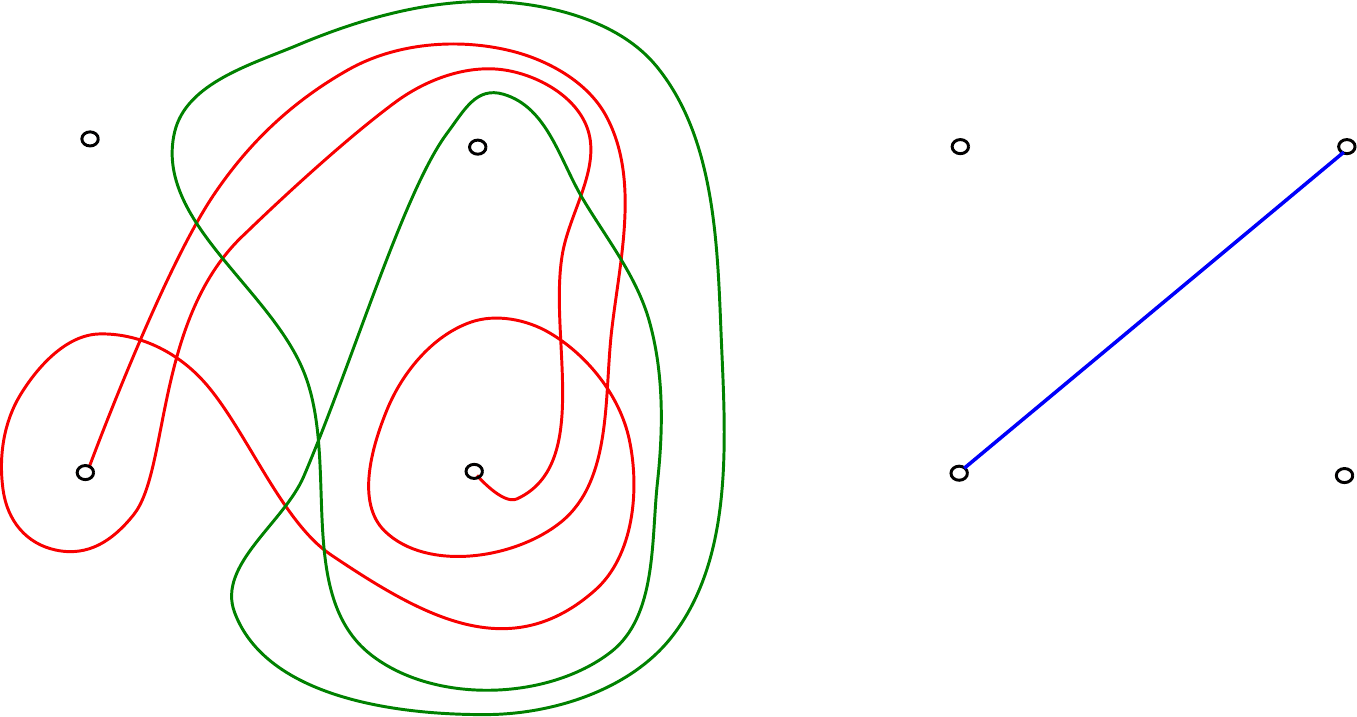}
 \caption{$R_\pi(T_1)$ (on the left) and $R(T_2)$ (on the right) for a 2-tangle decomposition of $K_{4,5}$.}\label{45torus2fig}
\end{figure}

Inserting an earring between the two tangles $T_1$ and $T_2$ yields the decomposition into three tangles, along a disjoint union of two $(S^2,4)$'s:
\begin{equation}
\label{tripledecomp}
(S^3, K_{4,5} \cup E, W)=(D^3, T_1)\cup_{(S^2,4)^\rout}(S^2\times I, 4\text{pt} \times I \cup E, W)\cup_{(S^2,4)^\rin}(D^3, T_2)
\end{equation}
where $E$ and $W$ are the circle and the arc forming the earring. 
Passing to holonomy perturbed traceless character varieties and applying Theorem~\ref{thmA} the topological decomposition~\eqref{tripledecomp}  produces the generalized immersed Lagrangian correspondence (see Definition~\ref{def:glag}) of length three:
\begin{equation}\label{triple_maps}
(R_\pi(T_1),\NAT_s, R(T_2))\coloneqq \begin{tikzcd}[row sep=.8cm, column sep=.8cm]
& {R_\pi(T_1)} \arrow[rd,"r^\rout"left] \arrow[ld]  & & \NAT_s \arrow[ld,"r_s^\rout" left]\arrow[rd,"r_s^\rin" right] & & {R(T_2)}\arrow[ld,"r^\rin"right] \arrow[rd]   &\\
\text{1pt} & & P^\rout & & P^\rin & & \text{1pt}
\end{tikzcd}
\end{equation}

\begin{remark*} The computer calculation which produces Figure~\ref{corrfig} shows that the action of the correspondence  $\coreMap_s$ on the linear arc $R(T_2)$ is a true Figure eight curve, rather than just a homology Figure eight curve, as predicted by Conjecture \ref{truefig8}.     In addition,  the identification of $R_\pi(T_1)$ given in Figure \ref{45torus2fig} is also the result of a computer calculation.     
Hence the following discussion  implicitly relies on two machine computations. 
\end{remark*} 

\medskip

Taking $\HF(R_\pi(T_1),R_\pi^\nat(T_2))$ is one way to compute the pillowcase homology,  and this is shown to be isomorphic to $I^\nat(K_{4,5})\cong \FF^7$ in \cite{HHK2}. We now outline why
\begin{equation}\label{badboy}
\HF(R_\pi(T_1),\NAT_s, R(T_2))\cong \HF(R_\pi(T_1),R^\nat_\pi(T_2))\cong \FF^7\not\cong \HF(R_\pi^\nat(T_1),R(T_2))\cong \FF^9,
\end{equation}
where the left-most group is the quilted Floer homology of the triple (see Section~\ref{sec:appendix_quilts}).
In words, considering three Lagrangian correspondences~\eqref{triple_maps}, composition of the second pair preserves the quilted Floer homology, but composition of the first pair does not. 

In particular, the conclusion of the Wehrheim-Woodward embedded composition theorem~\eqref{embeddedcomp} does not extend, without modifications,   to composition of  immersed Lagrangians {\em in the context  of perturbed flat moduli spaces.}

Implicit in Equation~\eqref{badboy} is the choice of the zero bounding cochain (see Section~\ref{sec:appendix_quilts}) for each Lagrangian. We argue that the discrepancy~\eqref{badboy} is rectified by the introduction of the bounding cochain $b\in \CF(R^\nat_\pi(T_1))$ illustrated in the middle of Figure~\ref{bigon1fig}:  
\begin{equation}\label{goodboy}
 \HF((R_\pi^\nat(T_1),b),(R(T_2),0))\cong \HF((R_\pi(T_1),0),(R^\nat(T_2),0)) 
\cong   \FF^7.
\end{equation}
Moreover, we outline why {\em strip-shrinking}, the mechanism which Wehrheim-Woodward introduced in their proof of the embedded composition theorem, and which Bottman-Wehrheim use to define the  bounding  cochain $8(-,-)$ associated to composition, produces this very bounding cochain $b$.

The rest of this section explains these three assertions.  As a quick overview,  Figures \ref{bigon1fig} explains Equations (\ref{badboy}) and (\ref{goodboy}),  Figure~\ref{fig:quilt} exhibits the differential in the quilted complex, and Figures~\ref{fig:before_bubble} and~\ref{fig:bubble}  describe the appearance of a figure eight bubble. 

\medskip

To ease eyestrain we denote
\begin{align*}
A&\coloneqq  \:\text{the arc-type  component of $R_\pi(T_1)$ (red  in Figure~\ref{45torus2fig})},\\
A^\natural&\coloneqq A\circ \NAT_s    \text{ (the corresponding figure eight, as described by Theorem~\ref{thmB})},\\
L&\coloneqq  \:\text{the loop-type  component of $R_\pi(T_1)$ (green  in Figure~\ref{45torus2fig})},\\
L^\natural&\coloneqq L\circ \NAT_s  \text{ (two copies of $L$, as described by Theorem~\ref{thmB})},\\
B& \coloneqq R (T_2)\text{ (blue in Figure~\ref{45torus2fig}),  and }\\ 
B^\nat& \coloneqq \NAT_s\circ  R (T_2) \text{ (the corresponding figure eight, as described by Theorem~\ref{thmB})}.
\end{align*}
Using Figure~\ref{45torus2fig}, we obtain:
\begin{align*}
\CF(R_\pi(T_1), R^\nat_\pi(T_2)) &\cong  \CF(A,B^\nat)\oplus \CF(L, B^\nat) \cong \FF^5 \oplus \FF^4 \\
\CF(R^\nat_\pi(T_1), R_\pi(T_2)) &\cong  \CF(A^\nat,B)\oplus \CF(L^\nat, B) \cong \FF^5 \oplus \FF^4 
\end{align*}
  The right side of Figure~\ref{bigon1fig} reveals that the chain complex $\CF(A, B^\nat)$  has one bigon $y_+ \to x_-$, while the middle picture, which by  Theorem~\ref{thmB} corresponds to $\CF(A^\nat,B)$,  has none. One can check quickly, using Figure~\ref{45torus2fig}, that there are no other bigons in the Floer chain complexes, and hence
  $$\HF(R^\nat_\pi(T_1), R(T_2))\cong \FF^9 \quad \text{ and } \quad \HF(R_\pi(T_1), R^\nat(T_2))\cong \FF^7 .$$
  We remark that a local picture similar to Figure~\ref{bigon1fig} can also be found for a tangle decomposition of the $(4,7)$-torus knot, and more torus knot examples can be obtained using the methods of~\cite{FKP,HHK2}.

 \begin{figure}[t] 
\labellist
\pinlabel $A\subset R_\pi(T_1)$ at 40 5
\pinlabel $x$ at 70 70
\pinlabel $y$ at 115 108
\pinlabel $B\subset R_\pi(T_2)$ at 200 157
\pinlabel $A^\nat$ at 190 5
\pinlabel $x_-$ at 290 80
\pinlabel $x_+$ at 247 70
\pinlabel $y_+$ at 320 103
\pinlabel $y_-$ at 320 133
\pinlabel $b$ at 270 123
\pinlabel $B$ at 365 150
\pinlabel $B^\nat$ at 590 150
\pinlabel $y_-$ at 578 116
\pinlabel $y_+$ at 518 116
\pinlabel $x_-$ at 505 83
\pinlabel $x_+$ at 505 40
\pinlabel $A$ at 480 5
\endlabellist
\centering
\includegraphics[width=1\textwidth]{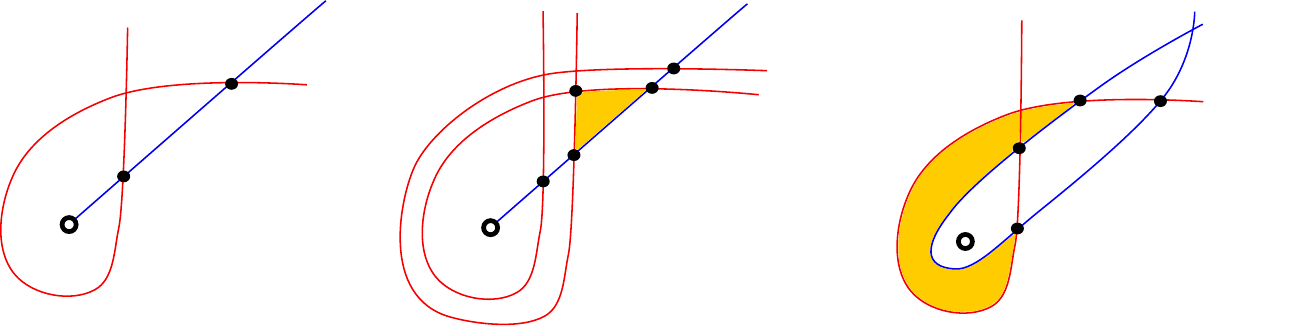}
\caption{ The pairs $(A,B)$, $(A^\nat,B)$ and $(A, B^\nat)$ }\label{bigon1fig}
\end{figure}

Because $I^\nat(K_{4,5})\cong \FF^7$, any correct construction of pillowcase homology   \emph{must}   count the differential $y_+ \to x_-$, if it expects to produce  singular instanton homology. Thus  an enhancement of the middle picture of Figure~\ref{bigon1fig} is needed to revive the differential $y_+ \to x_-$.
This is precisely what equipping the immersed Lagrangian $A^\nat$ with the indicated bounding cochain $b$ accomplishes: the shaded triangle contributes to $\mu^2(b,y_+)=x_-$. 
We expect that $b$ satisfies the $A_\infty$ Maurer-Cartan equation~\eqref{eq:MC}, and that other terms $\mu^k(b,\ldots , b,-)$ contributing to the differential $\partial_{b,0}$ (see Equation~\eqref{eq:deformed_differential})  vanish. In this case the homology of the deformed Floer complex $\CF((A^\nat,b), (B,0))=(\CF(A^\nat,B), \partial+\mu^2(b,-))$ (with $\partial=0$) is isomorphic to that of $\CF(A, (B^\nat,0))$, 
 which results in
$$\HF((R^\nat_\pi(T_1),b), (R(T_2),0))\cong \FF^7.$$

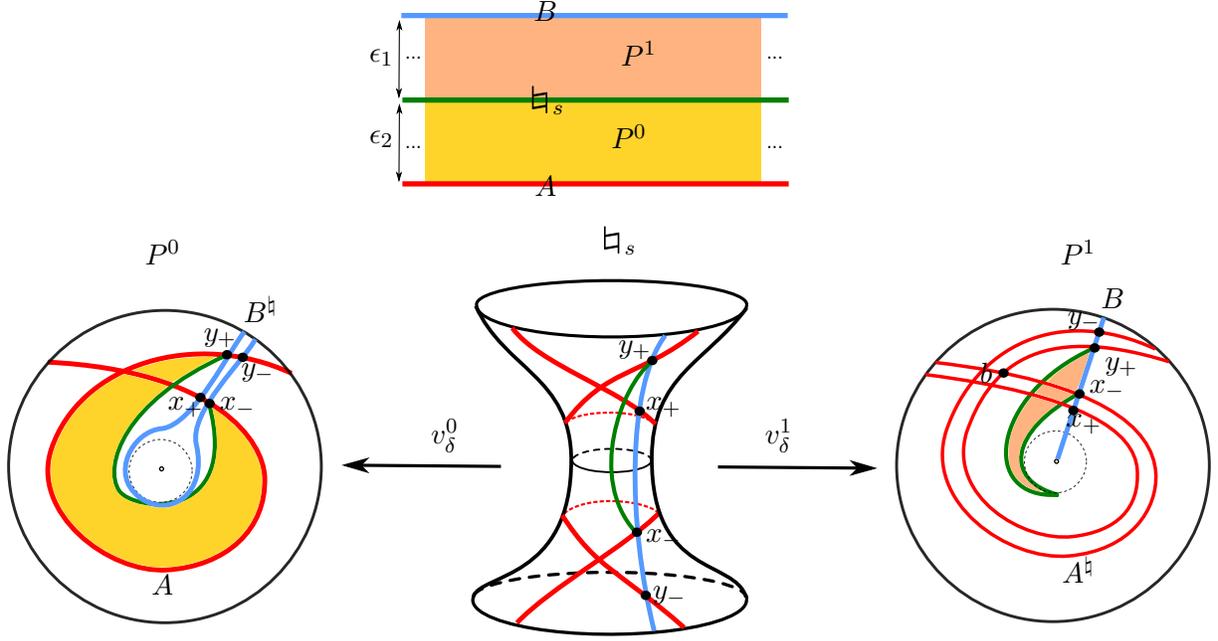
\begin{figure}[t] 
\labellist
\pinlabel $\epsilon_1$ at 650 1010
\pinlabel $\epsilon_2$ at 650 868
\pinlabel $A$ at 935 786
\pinlabel $B$ at 935 1091
\pinlabel $\NAT_s$ at 935 935
\pinlabel $P^\rin$ at 1097 1017
\pinlabel $P^\rout$ at 1080 872
\pinlabel $\NAT_s$ at 1060 690
\pinlabel $\modelMap{\delta}^\rout$ at 760 350
\pinlabel $\modelMap{\delta}^\rin$ at 1340 350
\pinlabel $B^\nat$ at 440 570
\pinlabel $y_+$ at 370 520
\pinlabel $y_-$ at 435 460
\pinlabel $x_-$ at 400 400
\pinlabel $x_+$ at 310 400
\pinlabel $y_+$ at 1090 500
\pinlabel $x_+$ at 1140 400
\pinlabel $x_-$ at 1140 170
\pinlabel $y_-$ at 1150 70
\pinlabel $y_-$ at 1870 550 
\pinlabel $y_+$ at 1935 470
\pinlabel $x_+$ at 1870 370
\pinlabel $x_-$ at 1910 430
\pinlabel $b$ at 1700 460
\pinlabel $A$ at 270 90
\pinlabel $A^\nat$ at 1860 115
\pinlabel $P^\rin$ at 1860 670
\pinlabel $P^\rout$ at 270 670
\pinlabel $B$ at 1920 590 
\endlabellist
\centering
\includegraphics[width=1\textwidth]{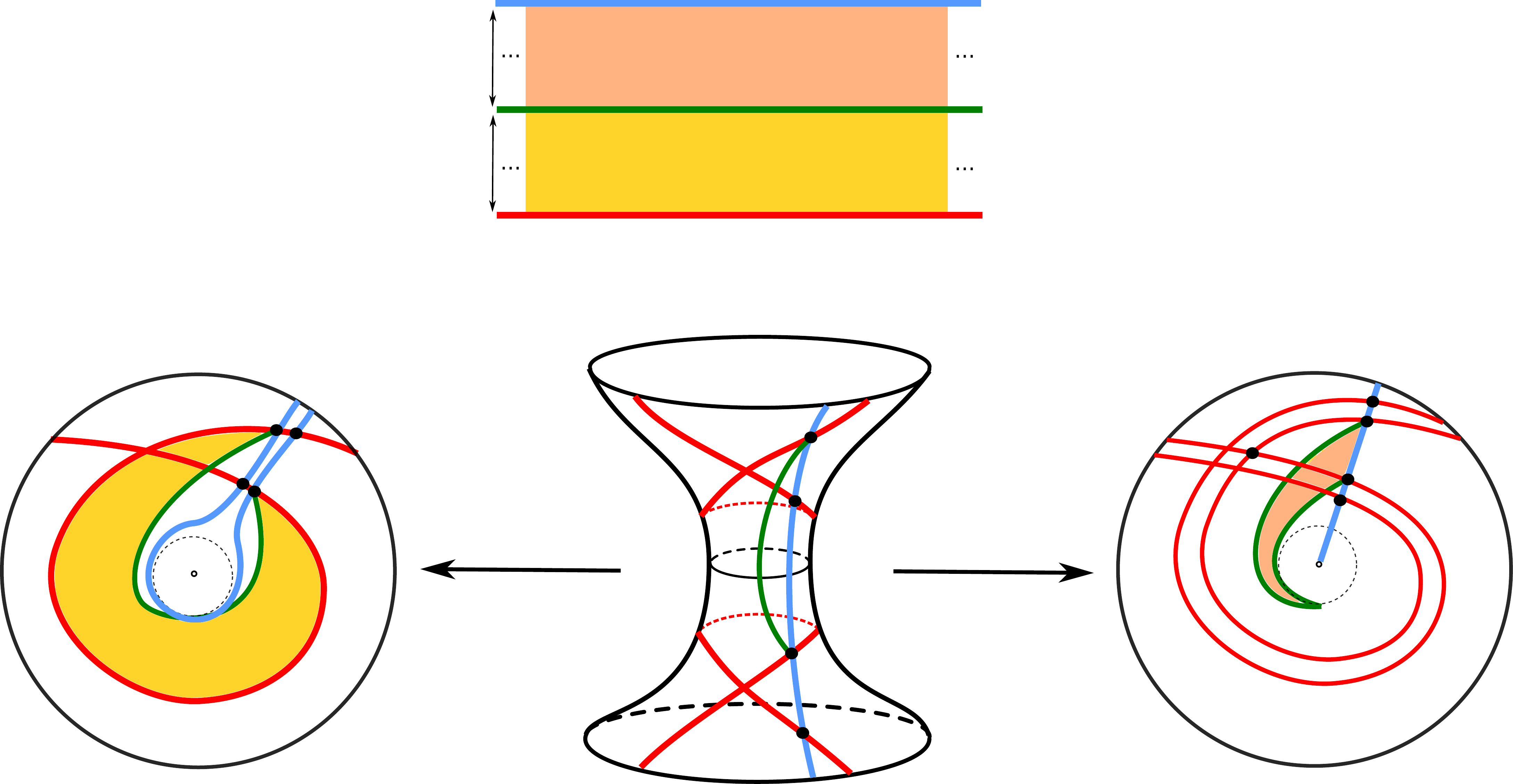}
\caption{Quilted Floer homology of $A\subset R_\pi(T_1) $, $\NAT_s$, and $B=R_\pi(T_2) $ from Figure~\ref{bigon1fig}. \label{fig:quilt} }
\end{figure}

\subsection{The quilt and the figure eight bubble}
We take these notions one step further by invoking quilts (discussed in Section~\ref{sec:appendix_quilts}) and arguing that the strip-shrinking process produces a {\em figure eight bubble with output $b$} (The term figure eight bubble should not be confused with Definition \ref{HF8} of a figure eight curve in the pillowcase; see Conjecture~\ref{conj:BW_immersed}). The loop-type component $L$ does not enter into the calculation, and hence we use the streamlined notation $A\subset R_\pi(T_1)$ and $B=R(T_2)$, which we think of as immersions of the interval $[0,1]$ into the pillowcase $P$.

The diagram~\eqref{triple_maps} determines  the quilted Floer chain complex  $\CF(A, \NAT_s,B)$
whose generators are naturally in bijection with the generators of the two composed complexes $\CF(A,B^\nat)$ and $\CF(A^\nat,B)$.  

\begin{figure}[ht!] 
\begin{subfigure}{1\textwidth}
\labellist
\pinlabel $y_+$ at 370 500
\pinlabel $x_-$ at 400 390
\pinlabel $y_+$ at 1060 430
\pinlabel $x_-$ at 1110 180
\pinlabel $y_+$ at 1920 460
\pinlabel $x_-$ at 1910 410
\pinlabel $b$ at 1710 460
\endlabellist
\centering
\includegraphics[width=1\textwidth]{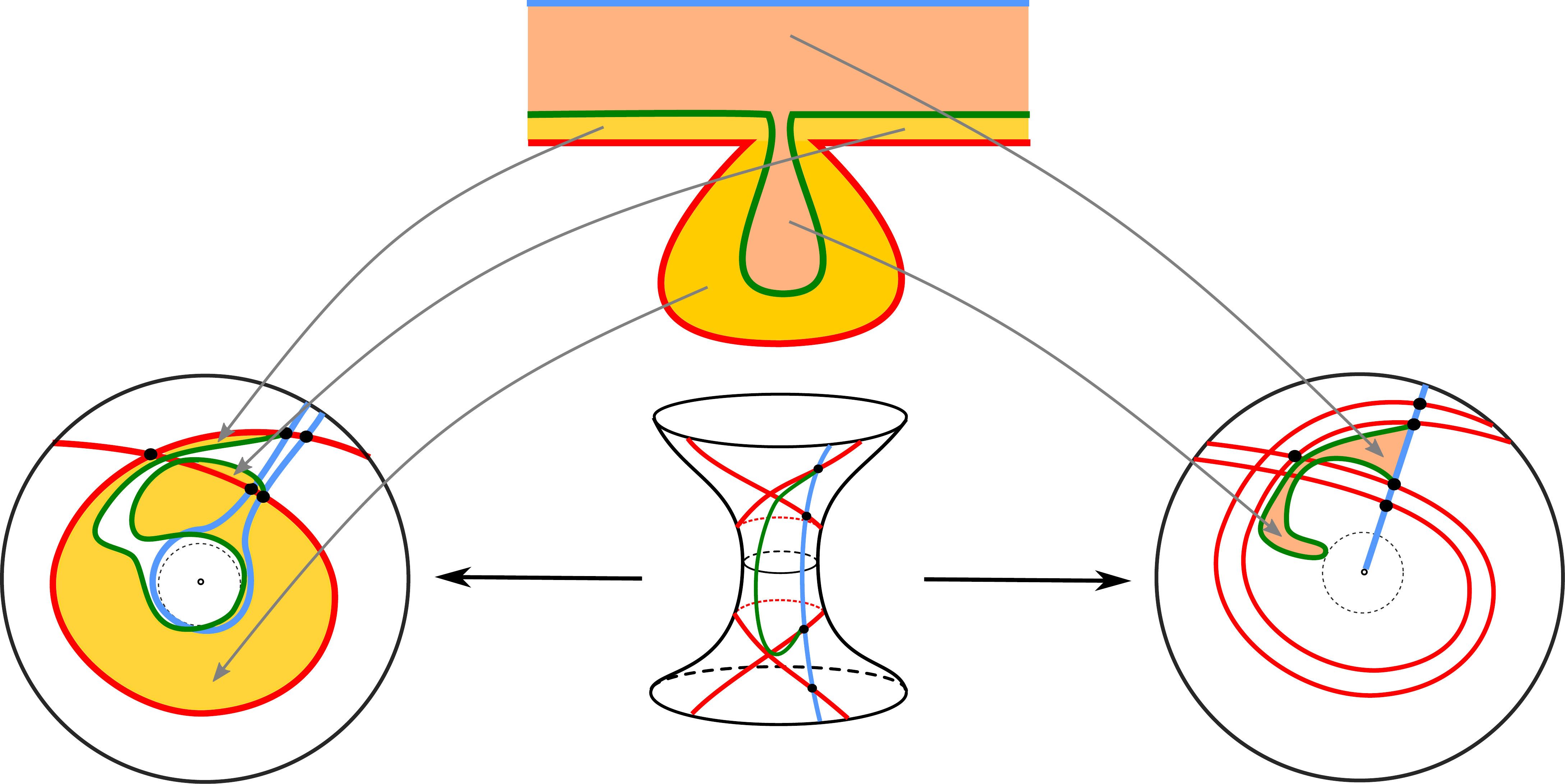}
\caption{Picture of the quilt right before the bubbling occurs. \label{fig:before_bubble} }
\end{subfigure}

\bigskip

\hrule

\bigskip

\begin{subfigure}{1\textwidth}
\labellist
\pinlabel $\text{triangle contributing to }\mu^2(b,y_+)=x_-$ at 1150 940
\pinlabel $\text{figure eight bubble}$ at 650 700
\pinlabel $b$ at 1040 880
\pinlabel $b$ at 220 480
\pinlabel $y_+$ at 370 500
\pinlabel $x_-$ at 400 390
\pinlabel $y_+$ at 1060 400
\pinlabel $x_-$ at 1110 150
\pinlabel $y_+$ at 1920 460
\pinlabel $x_-$ at 1910 410
\pinlabel $b$ at 1710 460
\endlabellist
\centering
\includegraphics[width=1\textwidth]{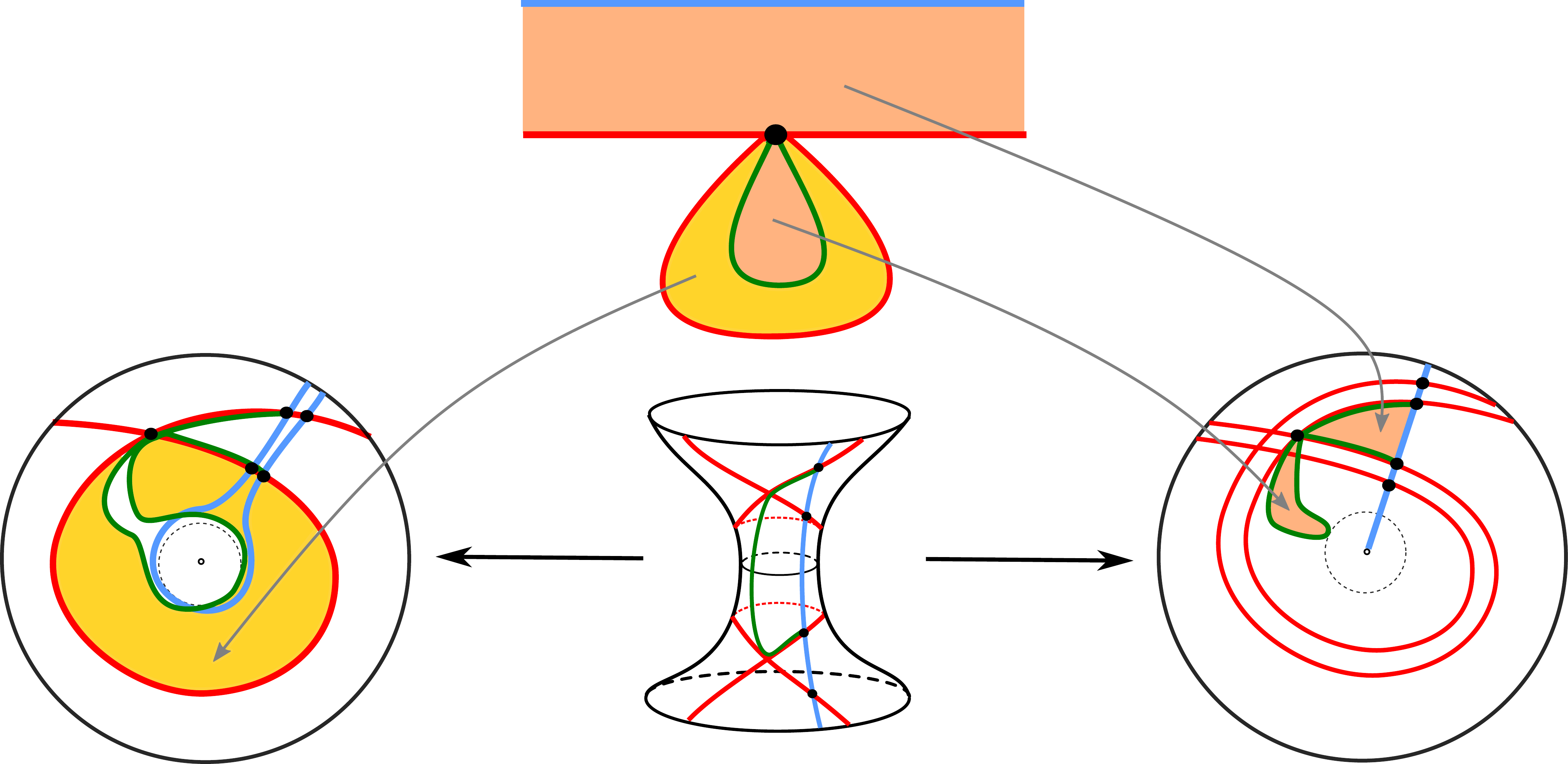}
\caption{Degeneration of the quilt into the figure eight bubble, and the triangle.\label{fig:bubble} }
\end{subfigure}
\caption{}\label{fig:b}
\end{figure} 

In the absence of bounding cochains, the differential in $\CF(A, \NAT_s,B)$ is defined by counting  $J$-holomorphic quilts of the form illustrated on the top of Figure~\ref{fig:quilt}. The labels $B,~\NAT_s,~A$ on the domain indicate the boundary conditions, and the labels $P^\rin,~P^\rout$ indicate the target spaces of the strips. Illustrated on the bottom of Figure~\ref{fig:quilt} are the images of two strips (left and right), together with the image of the seam inside $\NAT_s$  (in green, the seam is first mapped to $\NAT_s$, and by Theorem~\ref{thmA}, to both surfaces $P^0, P^1$  by $v_\delta$)\footnote{We remind the reader that the seam is allowed to vary in a 2-dimensional immersed Lagrangian of $(P^0)^-\times P^1$, as opposed to the other boundaries that are mapped to 1-dimensional Lagrangians.} , resulting in  an {\em immersed smooth quilt}  joining $y_+ \to x_-$.  

 A more careful study is required to prove that there are no further non-zero differentials.  We conjecture that this is the only one, which implies that the quilted Floer complex $\CF(R_\pi(T_1),\NAT_s, R(T_2))$ is chain isomorphic to $\CF((R_\pi(T_1), R^\nat(T_2))$.

We now outline a possible scenario of how a figure eight bubble might form in the strip-shrinking process, leading to the bounding cochain $b=8(0,0)\in \CF(A^\nat,A^\nat)$. 
The strip-shrinking argument with which Wehrheim-Woodward establish their embedded composition theorem proceeds, in our situation, as follows. 
One deforms the quilted strips (see the top of Figure~\ref{fig:quilt})
by letting the width $\epsilon_1$ (resp. $\epsilon_2$) tend to zero to obtain the differential of  $\CF(A^\nat,B)$ (resp. $\CF(A,B^\nat)$).  In the strip-shrinking process, a priori, certain degenerations called figure eight bubbles can arise, obstructing $\partial^2=0$ in the Floer complex for the composition.  This is the point investigated in analytic detail in \cite{BW}.

In our case (see Figure~\ref{fig:quilt}), in the limit $\epsilon_1 \to 0$ the seam limits to $B$, and so the indicated quilt limits to the bigon $y_+ \to x_-$ in $\CF(A,B^\nat)$ (see the right side of Figure~\ref{bigon1fig}).

 In Figures~\ref{fig:before_bubble}  (slightly before bubbling occurs)  and~\ref{fig:bubble} below, we sketch the bubbling process 
for the limit $\epsilon_2 \to 0$. The quilt  appears to tend to a triangle in $P^{\rout}$, with a figure eight bubble attached to it,  outputting the desired bounding cochain $b$ (the same as in the middle of Figure~\ref{bigon1fig}).

\medskip

There are serious technical obstacles to making the above discussion rigorous, in particular determining precisely the image of the seam,  not to mention that there may be quilts or polygons which we have overlooked.

\section{Future directions}
\label{sec:future_directions}
 
\subsection{New understanding}
 
In the light of discussion in the previous section, we propose a modification of pillowcase homology   and its relation to instanton homology.
 
\begin{conjectureABC} \label{conj:bounding_cochains}  \hspace{1cm} \\
There exists an assignment that associates to every 2-tangle $T$ and its holonomy perturbed traceless character variety $R_\pi(T) \looparrowright P^*$ a bounding cochain $b\in \CF(R_\pi(T),R_\pi(T))$, satisfying
\begin{itemize}
\item $(R_\pi(T),b)$ is a well-defined tangle invariant, as an object of the wrapped Fukaya category $\wrFuk(P^*)$;
\item This assignment of bounding cochains extends to tangles modified by the earring, which results in a tangle invariant $(R^\nat_\pi(T),b)$;
\item Given a 2-tangle decomposition $(S^3,\link)=(D^3,T_1) \cup_{(S^2,4)} (D^3,T_2)$, the corresponding Lagrangian Floer homology recovers reduced singular instanton homology:
$$\HF\big((R_\pi(T_1),b_1),(R^\nat_\pi(T_2), b_2)\big)\cong I^\nat(\link).$$ 
\end{itemize}
\end{conjectureABC}
 
In the previous section we   provided evidence   that adding an earring to a tangle  can require  a non-zero bounding cochain $b$.   We shall indicate in the next section why we expect that nontrivial bounding cochains are needed for a general tangle, whether or not it contains an earring. Furthermore, we outline a possible way of constructing the desired bounding cochain $b\in \CF(R_\pi(T),R_\pi(T))$.

\subsection{Approach via Floer Field theory}
\label{sec:fft}

\begin{figure}[ht]
\begin{tikzcd}[row sep=0.7]
 & \text{\bf{Floer field theory}}&
  \\
 \boxed{\substack{(F,k) \\ \text{Punctured surface}}} \arrow[r]  
 & \boxed{\substack{R(F,k) \\\text{Symplectic manifold} }} &\\
  \boxed{\substack{(F_1,k_1)\xrightarrow{(Y,T)}(F_2,k_1) \\ \text{Tangle cobordism}} }\arrow[r]  &
  \boxed{\substack{ R(F_1,k_1)\xrightarrow{\underline{L}(Y,T)}R(F_2,k_1) \\\text{Generalized Lagrangian correspondence}}} &\\
   \boxed{\substack{\text{Tangle }(Y,T) \\ \partial (Y,T)= (F,k)}} \arrow[r]  &
  \boxed{\substack{ 1 \text{pt} \xrightarrow{\underline{L}(Y,T)}R(F,k) \\\text{Generalized Lagrangian}} } &\\
  \boxed{\substack{(Y,\link) \\ \text{Link in a 3-manifold}} }\arrow[r]  &
  \boxed{\substack{ 1 \text{pt} \xrightarrow{\underline{L}(Y,\link)} 1 \text{pt} \\\text{Generalized Lagrangian correspondence}} } \arrow[r]& \boxed{\substack{ H_* (Y,\link)= \HF(\underline{L}) \\ \text{Homology theory}} }
 \end{tikzcd}
 \caption{}
 \label{fig:fft}
\end{figure}

In~\cite{FF_coprime,FF_tangles,Wehr_exp}, Wehrheim and Woodward developed \emph{Floer Field theory}, a systematic framework to define invariants of 3-manifolds and links inside them. Given a Cerf decomposition of a 3-manifold along a union of surfaces, the main idea is to construct a generalized Lagrangian correspondence from the point to itself, and then take its quilted Floer homology. Their embedded composition theorem~\eqref{embeddedcomp}, in the context of local Cerf moves, gives a method   to establish that the resulting Floer homology is independent of   the Cerf decomposition. As we shall see in (\ref{more_bounding_cochains}), combined with a conjecture of Bottman and Wehrheim, this should also provide us with a systematic procedure to produce the bounding cochain needed to correct the definition of pillowcase homology.

In more detail, as illustrated in Figure~\ref{fig:fft}, a Floer field theory is a functor from the tangle cobordism category  to the \emph{symplectic category} $\Symp$\footnote{This is Wehrheim-Woodward's version~\cite{FF_coprime} of Weinstein's symplectic category~\cite{Weinstein}.}, where objects are symplectic manifolds (with any extra set of assumptions that makes Floer homology well-defined), and morphisms are generalized Lagrangian correspondences up to a certain equivalence. The latter equivalence identifies those Lagrangian correspondences that are related by Equation~\eqref{eq:gen_corr_comp}. Composition of morphisms in the symplectic category is defined via concatenating generalized Lagrangian correspondences.

The relevant Floer Field theory for the present article would be a functor that assigns the traceless character variety $R(F, k)$ to a surface with $k$ punctures, and to a tangle cobordism $(F_1,k_1)\xrightarrow{(Y,T)}(F_2,k_1)$ the generalized Lagrangian correspondence 
\[
\underline{L}(Y,T)=(R(Y_1,T_1), \cdots, R(Y_\ell,T_\ell)),
\]
where $(Y,T) = (Y_1,Y_1) \cup \cdots \cup (Y_\ell,T_\ell)$ is a decomposition along punctured surfaces.

Wehrheim  and Woodward construct such a  Floer field theory in \cite{FF_tangles} for a restricted class of tangles that are such that in a Cerf decomposition, all the moduli spaces associated with the splitting surfaces are smooth, and all the ones associated with the elementary pieces are smooth and  {\em embedded} inside the moduli spaces of their boundary.  This condition holds  for elementary tangles, as illustrated in Figure~\ref{fig:elem_tangle}.   A decomposition into elementary tangles is easy to obtain via a bridge presentation of the tangle. 
To show that this procedure induces a Floer field theory, they show that the morphism in $\Symp$ associated with $(Y,T)$ is independent of the decomposition into elementary tangles, i.e. that the generalized Lagrangian correspondences associated to two different decompositions are equivalent, i.e., related by a sequence of compositions and decompositions (Equation~\eqref{eq:gen_corr_comp}). This is  done via Cerf theory by checking a few elementary \emph{Cerf moves}, see~\cite[Theorem~2.2.13, Sections~3.3 and~3.4]{FF_coprime}. Together with the embedded composition theorem~\eqref{embeddedcomp}, they then show that the groups $\HF(\underline{L}(Y,T))$ are topological invariants.

However, in our setting,  neither the   $R(F, 2k)$'s nor the  $R(Y_i,T_i)$'s are smooth manifolds, due to the even number of punctures. Nevertheless, it may be possible to develop Floer theory in $R(S^2, 2k)$'s, via equivariant techniques, or by only considering the smooth stratum (see~\cite{HK} for the description of these spaces in general, and~\cite{Paul} for a detailed study of $R(S^2,6)$). The Floer field theory framework would result in a well-defined generalized Lagrangian invariant inside the pillowcase 
$$\text{1pt}\xrightarrow{\underline{L}(D^3,T)} R(S^2,4)=P.$$

 We expect that  applying the Bottman-Wehrheim  framework (Conjecture~\ref{conj:BW_immersed}) gives   a well-defined Lagrangian with a bounding cochain:
\begin{equation}\label{more_bounding_cochains}\underline{L}(D^3,T)=(R(Y_1,T_1), \cdots, R(Y_\ell,T_\ell)) \qquad \xmapsto{\text{Composing all } R(Y_i,T_i)\text{'s}} \qquad (R_\pi(T),b).
\end{equation} 
Indeed, since the tangles $(Y_i,T_i)$ are elementary, one can reasonably hope that their associated Lagrangian correspondences $R(Y_i,T_i)$, after  suitable holonomy perturbations, should be smoothly embedded in the smooth loci of the $R(S^2, 2k)$'s, and equipped with the trivial bounding cochains $b_i = 0$. Compositions of consecutive correspondences, following Bottman-Wehrheim's conjecture, should then be equipped with bounding cochains $ 8(b_i, b_{i+1})$, which might now be nontrivial. Composing them all, we should then get $R_\pi(T)$, equipped with for example

$$
b = 8( \ldots 8(8(b_1, b_2), b_3), \ldots , b_l ) ,
$$
or any other way of bracketing the $8(-,-)$'s.  
 
\begin{figure}[ht!]
  \centering
  \begin{subfigure}{0.49\textwidth}
    \centering
    \includegraphics[width=1\textwidth]{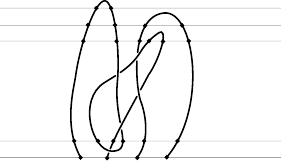} 
    \caption{Decomposition of a 2-tangle $(D^3,T)$ into elementary tangles.}\label{fig:hack}
  \end{subfigure}
  \begin{subfigure}{0.49\textwidth}
    \centering
    \includegraphics[width=0.6\textwidth]{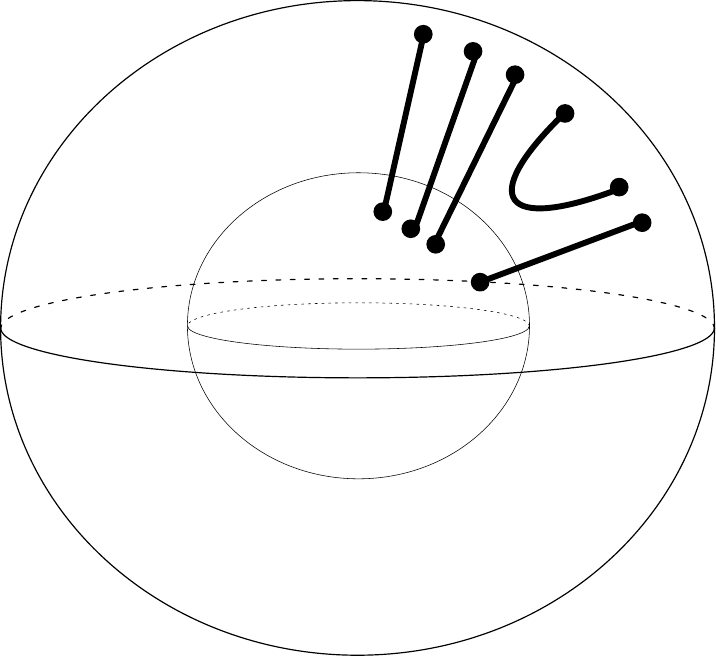}
    \caption{An example of an elementary tangle.}\label{fig:elem_tangle} 
  \end{subfigure}
  \caption{}
\end{figure}

\subsection{Approach via Bordered theory}
\emph{Bordered theory} is another effective cut-and-paste technique for studying homological invariants of links. It is a functor from the tangle cobordism category to the category of \emph{$A_\infty$ algebras} and \emph{$A_\infty$ bimodules} over them. These $A_\infty$ algebraic structures are generalizations of regular algebras and bimodules, and we refer the reader to \cite[Section~2]{LOT-bim} for the definitions and properties. Algebras are considered up to quasi-isomorphism, and bimodules up to homotopy equivalence. Composition of bimodules is defined via derived tensor product. On the left of Figure~\ref{fig:bordered+fft} we list the algebraic objects associated to the basic building blocks of tangle cobordisms category. 

The means by which the bordered theory is constructed can be different.  The first bordered theory was discovered by Khovanov~\cite{Khovanov2} in the context of Khovanov homology, and his construction is carried out in the context of algebra and combinatorics of tangle diagrams. The work of Lipshitz, Ozsv\'ath and Thurston on Heegaard Floer homology of 3-manifolds with boundary~\cite{LOT-main,LOT-bim} established the first symplectic geometric bordered theory, in which the $A_\infty$  aspects played a crucial role. They also invented the term ``bordered theory,''  the etymology of which comes from the fact that each boundary component of a 3-dimensional cobordism should be \emph{bordered}, that is, parameterized up to isotopy.  

\medskip

To an  $A_\infty$ category $\cA$ with objects $O_i$, one can define the  $A_\infty$ algebra $\oplus_{i,j}\Hom(O_i,O_j)$.  In favorable circumstances one is able to pick a small set of objects $G_\ell$ in $\cA$ such that the corresponding  subcategory $\cB$ generates $\cA$ in an appropriate sense, (roughly, $\Tw \cB\cong \Tw \cA$). The advantage is that this results in a   more tractable $A_\infty$ algebra $\oplus_{\ell,m}\Hom(G_\ell,G_m)$.

Figure~\ref{fig:bordered+fft} illustrates that given a Floer field theory, there is a way to construct a corresponding bordered theory. Namely, to a punctured surface $(F,k),$
the  $A_\infty$ category we associate is the Fukaya category of (the smooth top stratum of) $R(F,k)$.

On the level of morphisms, the fact that generalized Lagrangian correspondences induce $A_\infty$ homomorphisms (which, in turn, can be viewed as $A_\infty$ bimodules~\cite[Definition~2.2.48]{LOT-bim}) is a topic of active research~\cite{Functor_MWW,BW,Fuk_Ymaps,Gao}. This bridge between bordered theory and Floer field theory was partially realized for bordered Heegaard Floer homology~\cite{Auroux1,Auroux2,Perutz_LMI,Perutz_LMII,LP}, and realized for Khovanov bordered theory~\cite{Seidel-Smith,Abo_Smith_0,Abo_Smith}.

\begin{figure}[ht!]
\begin{tikzcd}
 \text{\bf{Bordered $A_\infty$ theory}} & & \text{\bf{Floer field theory}}
  \\
 \boxed{\substack{ \mathcal A(F,k) \\ A_\infty \text{ algebra} }}& \boxed{\substack{(F,k) \\ \text{Punctured surface}}} \arrow[r]  \arrow[l]  
 & \boxed{\substack{R(F,k) \\\text{Symplectic manifold} }} \arrow[ll,bend right=13,"\text{The Fukaya category}" above]\\
 \boxed{\substack{ _{\mathcal A(F_1,k_1)}{\mathcal X(Y,T)}_{\mathcal A(F_2,k_1)} \\ A_\infty \text{ bimodule} } }& \boxed{\substack{(F_1,k_1)\xrightarrow{(Y,T)}(F_2,k_1) \\ \text{Tangle cobordism}} }\arrow[r]  \arrow[l]  &
  \boxed{\substack{ R(F_1,k_1)\xrightarrow{\underline{L}(Y,T)}R(F_2,k_1) \\\text{Generalized Lagrangian correspondence}}} \arrow[ll,bend right=10, dashed]\\
  \boxed{\substack{ \mathcal X(Y,T)_{\mathcal A(F,k)} \\ A_\infty \text{ module over } \mathcal A(F,k) } }& \boxed{\substack{\text{Tangle }(Y,T) \\ \partial (Y,T)= (F,k)}} \arrow[r]  \arrow[l]  &
  \boxed{\substack{ 1 \text{pt} \xrightarrow{\underline{L}(Y,T)}R(F,k) \\\text{Generalized Lagrangian}} }\arrow[ll,bend left=10, dashed] \\
  \boxed{\substack{ H_* (Y,\link) \\ \text{Homology theory}} }& \boxed{\substack{(Y,\link) \\ \text{Link in a 3-manifold}} }\arrow[r]  \arrow[l]  &
  \boxed{\substack{ 1 \text{pt} \xrightarrow{\underline{L}(Y,\link)} 1 \text{pt} \\\text{Generalized Lagrangian correspondence}} }\arrow[ll,bend left=13,"\text{quilted Floer homology}" below]
 \end{tikzcd}
 \caption{}
 \label{fig:bordered+fft}
\end{figure}

\medskip 

Constructing a \emph{gauge theoretic} bordered theory for reduced singular instanton homology $I^\nat(\link)$ is an open problem. As sutured versions are being developed~\cite{KM_sut,Street,XZ}, upgrading them to the full bordered theory is conceivable, at least in the case of 2-tangles $(D^3,T)$. The resulting hypothetical $A_\infty$ module invariant $\mathcal X(D^3,T)$ would be an algebraic counter-part to the invariant $(R_\pi(T),b)$ from Conjecture~\ref{conj:bounding_cochains}. In the next section we explain how  $\mathcal X(D^3,T)$ may result in an immersed curve invariant of a 4-ended tangle.

\subsection{Instanton immersed curves}\label{sec:inst_curves} Given a 3-manifold $M$ with a parameterized boundary $F$, bordered Heegaard Floer theory outputs not only an $A_\infty$ module $\CFAhat(M)_{\mathcal A(F)}$, but also a \emph{type D structure}~$\CFDhat(M)^{\mathcal A(F)}$. There is a  duality between these two algebraic structures, making them equivalent, see~\cite[Proposition~2.3.18]{LOT-bim}. Furthermore, type D structures are equivalent to \emph{twisted complexes}~\cite{Bondal_Kapranov,Kontsevich_HMS}, if $\mathcal A(F)$ is viewed as an $A_\infty$ category and twisted complexes are not required to be bounded; see \cite[Proposition~2.13]{KWZ}. The unbounded twisted complexes, in turn, are formally equivalent to Lagrangians with bounding cochains. 

In what follows we   use the terms {\em type D structure}, {\em Lagrangian with a bounding cochain}, and {\em twisted complex} interchangeably.   These all  refer to the following algebraic structure: a collection of objects $L=\oplus L_i$ of an $A_\infty$ category (equivalently, idempotents of an $A_\infty$ algebra), together with a collection of morphisms $b_{i,j} \in \CF(L_i,L_j)$ (equivalently, elements of an $A_\infty$ algebra with left and right idempotents $i$ and $j$), such that $\sum b_{i,j}=b \in \CF(L,L)$ satisfies the Maurer-Cartan equation~\eqref{eq:MC}. We  stick to the term ``twisted complex'' for the majority of discussion, and use the superscripts to indicate the algebra, as in $\CFDhat(M)^{\mathcal A(F)}$.  See Appendix \ref{Appendix} for some discussion.

\medskip

In  the torus boundary case, Hanselman, Rasmussen and Watson interpreted the bordered Heegaard Floer type D structure invariant as an immersed curve $\HFhat(M)$ in a once-punctured torus~\cite{HRW}:
$$(M, \partial M=T^2) \quad \mapsto \quad \CFDhat(M)^{\mathcal A(T^2)}  \quad \mapsto \quad \HFhat(M) \looparrowright T^2\setminus 1\pt$$
Immersed curves in this context may contain several components, and each component comes with a local system.
By appropriately identifying the  torus $T^2 $ with the boundary $\partial M$, they show that the immersed curve $\HFhat(M)$ does not depend on the parameterization of $T=\partial M$. Thus $\HFhat(M)$ fits into the Floer field theory philosophy, that is, the invariant of $M$ should be an object of the Fukaya category of the moduli space associated to $\partial M$, and in this case the Heegaard Floer moduli space is equal to $\mathit{Sym}^1(T^2\setminus 1\pt)=T^2\setminus 1\pt$. Another key result, which also supports the Floer-field-theoretic perspective, is the gluing theorem for immersed curves:
\begin{equation}\label{hf_gluing}\HFhat(M \cup_{T^2} M')=\HF(\HFhat(M),\HFhat(M'))\end{equation}

The construction of $\HFhat(M)$ relies on the following two points. The first is  that the bordered Heegaard Floer theoretic torus algebra $\mathcal A(T^2)$ is equal to the partially wrapped Fukaya category of $T^2\setminus 1\pt$. The second is a classification result, stating that any twisted complex (up to homotopy equivalence) over the Fukaya category of a surface is equivalent to a collection of immersed curves in that surface, decorated with local systems. This was first proved by Haiden-Katzarkov-Kontsevich~\cite{HKK} by representation theoretic techniques, but a more geometric proof was given in~\cite{HRW}, which provides a practical algorithm of obtaining an immersed curve $\HFhat(M)$ from the twisted complex $\CFDhat(M)^{\mathcal A(F)}$. Generalizations of this geometric approach to the fully wrapped case and fields other than characteristic 2 were obtained in~\cite{KWZ}. 
 
\medskip

In the context of 2-tangles $(D^3,T)$, similar immersed curve invariants were recently constructed, with values in the wrapped Fukaya category of the boundary 4-punctured sphere $(S^2,4)=\partial(D^3,T)$. Heegaard Floer theoretic curves, which recover knot Floer homology, were developed by Zibrowius in~\cite{pqMod}. In symbols: 
$$(D^3,T)\quad \mapsto \quad \HFT(T)\looparrowright (S^2,4), \qquad \HFKhat(T\cup_{(S^2,4)} T')=\HF(\HFT(T),\HFT(T')).$$
A Khovanov-theoretic curve invariant of 2-tangles was developed in \cite{HHHK}, now immersed in the pillowcase $\Khr(T)\looparrowright P^*$. In \cite{KWZ}, the authors expanded this invariant, by introducing two more curve invariants $\BNr(T),~\Kh(T)$, and also noting that the pillowcase can be canonically identified with the boundary $\FPS$, as long as one of the tangle ends is distinguished (the distinguished tangle end in the context of character varieties is the one containing  the basepoint). As a result, Khovanov theory produces the following immersed curve invariants of \emph{pointed} (i.e., with one tangle end distinguished) 2-tangles:
$$
(D^3,T)\quad \mapsto  \quad\BNr(T),\Khr(T),\Kh(T)\looparrowright (S^2,4)=\partial (D^3,T)$$
satisfying the gluing formulas

\begin{equation}\label{eq:kh_gluing}
\begin{split}
 \BNr(T\cup T')&=\HF(\BNr(T),\BNr(T')) \quad  \text{(reduced Bar-Natan homology~\cite{BarNatanKhT})} \\
 \Khr(T\cup T')&=\HF(\BNr(T),\Khr(T')) \quad \text{(reduced Khovanov homology) } \\
 \Kh(T\cup T')&=\HF(\BNr(T),\Kh(T')) \quad \text{(unreduced Khovanov homology) }
\end{split}
\end{equation}

\medskip

Based on computations and similarities with Khovanov theory, we envision the existence of the following three \emph{instanton} theoretic immersed curves associated to pointed 2-tangles.
 
\begin{conjectureABC}
There exists an assignment that associates to every 2-tangle $T$ three instanton theoretic immersed curve invariants  
$$
(D^3,T)\quad \mapsto  \quad I(T),I^\nat(T),I^\sharp(T)\looparrowright P^* \cong \partial(D^3,T),$$
such that given a 2-tangle decomposition $(S^3,\link)=(D^3,T) \cup_{(S^2,4)} (D^3,T')$, the instanton invariants are recovered via wrapped Lagrangian Floer homology:
\begin{align*}
 \widehat{I}(\link)&=\HF(I(T),I(T')) \quad  \text{(``minus'' version of instanton homology~\cite{Daemi_Scaduto}) } \\
 I^\nat(\link)&=\HF(I(T),I^\nat(T')) \quad \text{(reduced instanton homology) } \\
 I^\sharp(\link)&=\HF(I(T),I^\sharp(T')) \quad \text{(unreduced instanton homology) }
\end{align*}
\end{conjectureABC}

The invariant $\widehat{I}(\link)$ above can be substituted by $I^\nat(\link,\Gamma)$~\cite{KM_deforms}, or possibly by $\underline{\mathit{KHI}}^-(\link)$~\cite{Zhenkun}. The relationship to Conjecture~\ref{conj:bounding_cochains} is as follows: under the one-to-one correspondence between immersed curves and twisted complexes~\cite{HKK}, the curve $I(T)$ should be the one associated to $(R_\pi(T),b)$, the curve $I^\nat(T)$ should be the one associated to $(R^\nat_\pi(T),b)$, and the curve $I^\sharp(T)$ should be the one associated to yet another hypothetical twisted complex $(R^\sharp_\pi(T),b)$. The latter representation variety $R^\sharp_\pi(T)$ is constructed by taking a \emph{disjoint sum} with a Hopf link, as opposed to adding an earring, which is equivalent to a \emph{connected sum} with a Hopf link.

To give a sample hypothetical computation of $\widehat{I}$ for trefoil $K_{2,3}$, we use the decomposition of the trefoil from Figure~\ref{fig:tref_dec}. Since both of the tangles are rational, perturbations are not needed, and we also guess that the bounding cochains are zero. In this case we would have $I(T)=R(T)$ and $I(T')=R(T')$, two Lagrangian arcs depicted in Figure~\ref{fig:tref_pairing}. Since we pair an arc with arc, we need to work in an appropriate setting, namely \emph{wrapped} Lagrangian Floer homology (see the next section for more details). In this setting one of the arcs should be infinitely wrapped around the punctures where it ends, resulting in the infinite number of generators for $\widehat{I}(K_{2,3})$ in Figure~\ref{fig:tref_pairing}. In terms of dimension, this picture agrees with the computation of $\widehat{I}(K_{2,3})$ given in \cite[Section~9.2.3]{Daemi_Scaduto} --- the generators $\mathbb F[x]=\langle 1,x,x^2,\ldots\rangle_{\mathbb F}$ correspond to the generators coming from the wrapping.

\begin{figure}[H]
\centering
\begin{subfigure}{0.45\textwidth}
\centering
\labellist 
\pinlabel $T$ at 150 190
\pinlabel $T'$ at 50 148
\endlabellist
\includegraphics[width=0.35\textwidth]{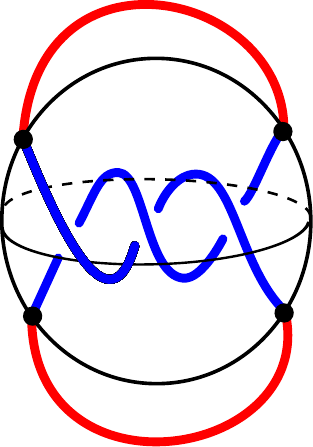}
\caption{Decomposition of the trefoil}\label{fig:tref_dec}
\end{subfigure}
\begin{subfigure}{0.45\textwidth}
\centering
\labellist 
\pinlabel $I(T)$ at 105 68
\pinlabel $I(T')$ at 10 48
\endlabellist
\includegraphics[width=0.44\textwidth]{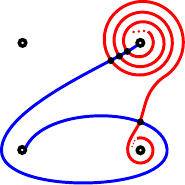}
\caption{The corresponding intersection picture computing $\widehat{I}(K_{2,3})$}\label{fig:tref_pairing}
\end{subfigure}
\caption{}\label{fig:tref}
\end{figure}

In addition to the connection of instanton curves with representation varieties and bounding cochains, another potential path towards defining $I(T)$ is via bordered theory discussed in the previous section. That is, obtaining first a twisted complex invariant $\mathcal{X}(D^3,T)$ over some algebra $\mathcal A$ via gauge theory, realizing $\mathcal A$ as a subcategory of $\wrFuk(S^2,4)$, and then getting a curve $I(T)$ by invoking the one-to-one correspondence. Due to the low number of tangle ends (four), we predict that the algebra $\mathcal A$ is the same as in the Khovanov case, i.e., $\mathcal A=\B$ the algebra from \cite{KWZ}; see Section~\ref{sec:wrapped} for more on this algebra. 

We end this section with a diagram summarizing the potential 2-tangle invariants. The dashed arrows are the ones not yet constructed.
\[
\begin{tikzcd}[row sep =1cm, column sep=0.8cm]
 \boxed{\substack{(D^3,T) \\ \text{Pointed 2-tangle}}} \arrow[rr,dashed,"\text{Floer field theory}" above] \arrow[d,dashed,"\text{Bordered theory}" right]  && \boxed{\substack{1\pt\xrightarrow{\underline{L}(T)} P^* \\ \text{Generalized Lagrangian} }} \arrow[d,dashed,"\text{Conjecture~\ref{conj:BW_immersed} (\cite{BW})}\\" left]\\
\boxed{\substack{ \X(T)^{\mathcal A} \\ \text{Type D structure} \\ \text{(Twisted complex)}}} \arrow[rd, "\text{\cite{HKK,KWZ}}" left]  && \boxed{\substack{ (L,b) \\ \text{Lagrangian with a bounding cochain}\\\text{(Twisted complex)} }} \arrow[ld, "\text{\cite{HKK,KWZ}}" right] \arrow[ll,no head,dashed,"\text{as objects in }\wrFuk(S^2\setminus 4)"below, "\simeq" above]  \\
&\boxed{\substack{I(T) \\ \text{Instanton immersed curve}}}&
 \end{tikzcd}
 \]

\subsection{The wrapped Fukaya category of the pillowcase and traceless character varieties of 2-tangles}\label{sec:wrapped}
We now tie together the general discussion above with the main result of the paper. Namely, we expand on the fact that, according to the Floer field philosophy discussed in Section~\ref{sec:fft}, the Lagrangian correspondence $\NAT_s$ is expected to induce an $A_\infty$  endofunctor on the Fukaya category of the pillowcase.

The choice of which Fukaya category to work in is not uniquely determined.  
For the character varieties of 2-tangles in a ball, one possible approach  is to leverage the fact that both loop-type and arc-type Lagrangians lift to smooth closed immersed equivariant curves in the 2-fold branched cover $T\to T/\iota=P$ given by the elliptic involution. This would suggest working in the $\ZZ/2$ equivariant Fukaya category associated to $T$.  Such a theory would typically be equipped with a $H^*(B\ZZ/2)=H^*(\mathbb R P^\infty)=\FF[x]$ action.
However, carefully developing the $\ZZ/2$ equivariant $A_\infty$ structures  is a challenge, and associating $A_\infty$ functors to Lagrangian correspondences would be even more difficult.

In the interest of working in a better understood context, we propose that the wrapped Fukaya category $\wrFuk(P^*)$ of the pillowcase (or, more precisely, the specific $A_\infty$ subcategory $\wrFuk^\star(P^*)$ defined below) provides a robust substitute for the $\ZZ/2$-equivariant Lagrangian Floer homology in the torus. 

This choice is motivated, among other reasons, by the fact that two arc-type Lagrangians $L_1,~L_2$ meeting at a corner of the pillowcase contribute a tower of generators $\{1,x,x^2,\ldots\}=\FF[x]$ in both the wrapped theory of $P^*$ and  $\ZZ/2$-equivariant theory of $T\to P$. A second important motivation for this choice comes from the calculation in Theorems~\ref{thmA} and \ref{thmB},    which identifies a Dehn twist near each corner in the action induced by the correspondence $\NAT_s$.     A last, but 
critical reason to work in the wrapped setting comes from the consequences of the theorem of Haiden-Katzarkov-Kontsevich, which  provides an identification of homotopy classes of twisted complexes over the wrapped Fukaya category of a punctured surface (such as the pillowcase) with immersed curves, as we next explain.

\begin{figure}[ht]
\centering
\begin{subfigure}{0.45\textwidth}
\centering
\labellist 
\pinlabel $S_1$ at 56 29
\pinlabel $S_2$ at 83 0
\pinlabel $\arcVer$ at 75 40
\pinlabel $\arcHor$ at 40 17
\pinlabel $D_1$ at -6 15
\pinlabel $D_2$ at 87 65
\endlabellist
\includegraphics[width=0.44\textwidth]{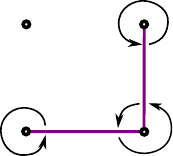}
\caption{}\label{fig:wrap_subcat}
\end{subfigure}
\begin{subfigure}{0.45\textwidth}
\centering
\labellist 
\pinlabel $\text{\tiny$D_2$}$ at 84 64
\pinlabel $\text{\tiny$S_1 S_2 S_1 S_2 S_1 S_2 S_1$}$ at -8 58  
\pinlabel $\text{\tiny$S_1 S_2 S_1 S_2 S_1$}$ at -1 48
\pinlabel $\text{\tiny$S_1 S_2 S_1 $}$ at 7 39
\pinlabel $\text{\tiny$S_1 $}$ at 15 31
\pinlabel $\text{\tiny$\id^\circ$}$ at 85 56
\pinlabel $\text{\tiny$S_1S_2$}$ at 87 44
\pinlabel $\text{\tiny$S_1S_2S_1S_2$}$ at 95 36
\pinlabel $\text{\tiny$S_1S_2S_1S_2S_1S_2$}$ at 103 28
\endlabellist
\includegraphics[width=0.44\textwidth]{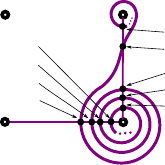}
\caption{}\label{fig:Fuk_pairing}
\end{subfigure}
\caption{}\label{fig:big_wrFuk}
\end{figure}

In Section~\ref{ssec:fukayacat} of the appendix we briefly recall the construction of the Fukaya category, and in particular, the wrapped Fukaya category of the pillowcase $\wrFuk(P^*)$. In what follows it is useful to restrict   to a particular full subcategory. 
  \begin{proposition}\label{missTLC}
Given any 2-tangle $T$ in a ($\ZZ/2$-homology) 3-ball and an appropriate holonomy perturbation $\pi$, the immersed curves  $R_\pi(T)$ and $R^\nat_\pi(T)$  miss the top left corner of the pillowcase, that is, the point with pillowcase coordinates $[\gamma,\theta]=[0,\pi]$.  
\end{proposition}
\begin{proof}Theorem~\ref{thm:immersed_curve_invariant} 
states that only abelian traceless representations are sent to corners for generic perturbations. An abelian representation is determined by its value on meridians. As such,  it is limited to  three possibilities  according to how the endpoints are joined by the tangle, and none of these possibilities correspond to $[\gamma,\theta]=[\pi,0]$. 
\end{proof}

Motivated by this proposition, we consider the full $A_\infty$ subcategory $\wrFuk^\star(P^*)\subset \wrFuk(P^*)$ consisting of unobstructed (see~\cite[Definition 2.1]{Abouzaid_Fuk_surface}) immersed curves that \emph{do not have endpoints at  the top left  corner}. In addition, it is convenient to include in $\wrFuk^\star(P^*)$ immersed curves equipped with nontrivial \emph{local systems}, that is, flat vector bundles over the field of two elements (see~\cite[Remark~2.11]{Aur3}).

We now describe the subcategory $\wrFuk^\star(P^*)$ algebraically, using~\cite[Theorem~7.6]{Bocklandt} and constructions from~\cite[Section~3]{HKK} (ignoring gradings). Let $\B$ be the full $A_\infty$ subcategory of the wrapped Fukaya category of the pillowcase \(\wrFuk(P^*)\) with objects the two arcs $\arcVer$ and $\arcHor$ in Figure~\ref{fig:wrap_subcat}. As an $A_\infty$ algebra,
\begin{equation}\label{eq:FukB}\cB=\CF({\arcHor}, {\arcHor}) \oplus \CF({\arcVer}, {\arcVer}) \oplus  \CF({\arcVer}, {\arcHor})\oplus \CF({\arcHor}, {\arcVer})\end{equation}
As explained in loc.cit., there is a particular set of Hamiltonian perturbations of $\arcVer$ and $\arcHor$, resulting in $\mu^1$ and $\{\mu^k\}^{k\geq 3}$ all being zero. Hence the product $\mu^2$ is associative and gives $\cB$ the structure of an associative algebra.

We now explain how to view $\cB$ as a quiver algebra. The chords $S_i$ and $D_j$ from Figure~\ref{fig:wrap_subcat}  are used to label the elements in $\cB$. Looking at Figure~\ref{fig:Fuk_pairing}, notice that the basis of $\CF(\arcVer,\arcHor)$, which by definition is the set of intersection points of the infinitely wrapped version of $\arcVer$ with $\arcHor$,
is  in one-to-one correspondence with the set of counter-clockwise chords between $\arcVer$ and $\arcHor$ around the bottom right puncture.    Any such chord can be written as a composite of $S_1$ and $S_2$, and hence \begin{equation}\label{eq:chords1}
\CF(\arcVer,\arcHor)=\langle S_1,~S_1S_2S_1,~S_1S_2S_1S_2S_1,\ldots\rangle_\FF\end{equation}
Similarly,  the other generators can also be labeled by chords, yielding
\begin{equation}\label{eq:chords2}
\begin{split}
\CF(\arcVer,\arcVer)&=\langle \ldots,S_1 S_2 S_1 S_2,~S_1 S_2,~\id^\circ,~D_2,~D_2 D_2,\ldots\rangle_\FF\\
\CF(\arcHor,\arcVer)&=\langle S_2,S_2S_1S_2,~S_2S_1S_2S_1S_2,\ldots\rangle_\FF\\
\CF(\arcVer,\arcVer)&=\langle \ldots,S_2 S_1 S_2 S_1,~S_2 S_1,~\id^\bullet,~D_1,~D_1 D_1,\ldots\rangle_\FF
\end{split}
\end{equation}
where $\id^\circ,~\id^\bullet$ denote the length zero chords. Note that $\id^\circ+\id^\bullet=1$ is a unit for $\cB$.

With this description, the product $\mu^2(x,y)$ is equal to concatenation of the chord descriptions of $x$ and $y$ from~\eqref{eq:chords1} and~\eqref{eq:chords2}; if concatenation does not result in another basis element in ~\eqref{eq:chords1} and~\eqref{eq:chords2}, $\mu^2(x,y)=0$. In other words, the category $\B$, viewed as an $A_\infty$ algebra, is a unital associative algebra  and has the following quiver description: 
\[\qquad \qquad \qquad \qquad 
\B =
\mathbb{F}_2\Big[
\begin{tikzcd}[row sep=2cm, column sep=1.5cm]
{\arcHor}
\arrow[leftarrow,in=145, out=-145,looseness=5]{rl}[description]{D_1}
\arrow[leftarrow,bend left]{r}[description]{S_2}
&
{\arcVer}
\arrow[leftarrow,bend left]{l}[description]{S_1}
\arrow[leftarrow,in=35, out=-35,looseness=5]{rl}[description]{D_2}
\end{tikzcd}
\Big]\Big/ (D_j S_i=0=S_i D_j)
\hspace{0.22\textwidth}
\]

Next we consider the category of twisted complexes $\Tw \cB$ (see Section~\ref{ssec:fukayacat} for a brief recollection of this notion).
The already mentioned classification result~\cite{HKK} (reproved geometrically in \cite{HRW,KWZ}) implies that the homotopy equivalence classes of twisted complexes over $\cB$ are in one-to-one correspondence with the regular homotopy classes of unobstructed, decorated with local systems  immersed curves  not touching the top left corner of the pillowcase. In other words, the arcs $\arcVer$ and $\arcHor$ not only generate $\wrFuk^\star(P^*)$, which  means that there is a quasi-equivalence of $A_\infty$ categories $\Tw \wrFuk^\star(P^*) \simeq\Tw \B $, but furthermore, the category of twisted complexes over $\wrFuk^\star(P^*)$ does not have any new objects compared to $\wrFuk^\star(P^*)$, that is, $\Tw \wrFuk^\star(P^*) \simeq \Tw \B \simeq \wrFuk^\star(P^*)$.

As such, immersed curves coming from the traceless character varieties $R_\pi(T)$ and $ R^\natural_\pi(T)$ define\footnote{We assume that $R_\pi(T)$ and $ R^\natural_\pi(T)$  are unobstructed. Alternatively, we consider  $R_\pi(T)$ and $ R^\natural_\pi(T)$ as homotopy classes. The latter is possible because the set of unobstructed immersed curves up to regular homotopy is the same as the set of all immersed curves up to continuous homotopy.} objects in $\wrFuk^\star(P^*)\simeq \Tw \B$, and we denote the corresponding twisted complexes by $\N(R_\pi(T))$ and $\N(R^\natural_\pi(T))$. We now give an example of a twisted complex associated to a curve.

\begin{example} Consider the three-twist tangle $T_3$ depicted in Figure~\ref{fig:three-twists}. In our pillowcase coordinates, the traceless character variety $R(T_3)$ is a linear arc of slope $\tfrac 1 3$ depicted in Figure~\ref{fig:quiverFuk_and_curve}.
The corresponding twisted complex is equal to 
\begin{align*}\N(R(T_3))=~&(\arcVer_1 \oplus \arcHor_2 \oplus \arcHor_3 \oplus \arcHor_4, ~ b_{12}+b_{23}+b_{34}),\\
&\text{where } b_{12}=S_1,~ b_{23}=D_1,~ b_{34}=S_2 S_1
\end{align*}
This twisted complex can be represented compactly as in Figure~\ref{fig:3-twist-tangle-D-str}, and we use this abbreviated graph notation subsequently. To see why  $R(T_3)$ corresponds to the twisted complex $\N(R(T_3))$, first notice that $R(T_3)$ is isotopic to the grey curve depicted  in Figure~\ref{fig:quiverFuk_and_curve}.    This grey curve can be straightforwardly reconstructed from the complex $\N(R(T_3))$, as explained in  see~\cite[Section~4.1]{HKK}.
\begin{figure}[ht!]
  \centering
  \begin{subfigure}[b]{0.25\textwidth}
  \centering
  \includegraphics[scale=0.5]{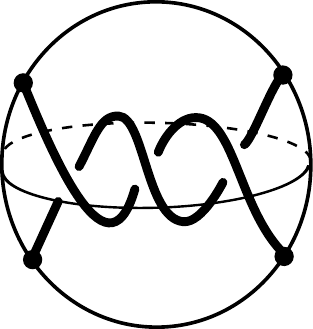}
  \caption{Tangle $T_3$.}
  \label{fig:three-twists}
  \end{subfigure}
  \begin{subfigure}[b]{0.25\textwidth}
  \centering
  \includegraphics[scale=0.9]{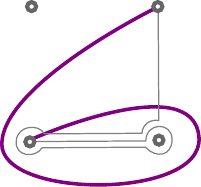}
  \caption{$R(T_3)$}
  \label{fig:quiverFuk_and_curve}
  \end{subfigure}
  \begin{subfigure}[b]{0.2\textwidth}
    \centering
    $$
    \begin{tikzcd}[row sep=0.4cm]
    \arcVer_1
    \arrow{d}{S_1}
    \\
    \arcHor_2
    \arrow{d}{D_1}
    \\
    \arcHor_3
    \arrow{d}{S_2 S_1}
    \\
    \arcHor_4
    \end{tikzcd} 
    $$
    \caption{$\N(R(T_3))$}
    \label{fig:3-twist-tangle-D-str}
  \end{subfigure}
  \begin{subfigure}[b]{0.25\textwidth}
  \centering
  \includegraphics[scale=0.8]{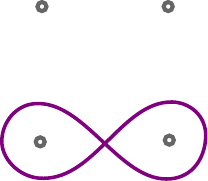}
  \caption{$R_\pi^\nat(||)$}
  \label{fig:fig_eight}
  \end{subfigure}
  \caption{}
\end{figure}

\end{example}
  Conjecture~\ref{conj:action} below proposes a  candidate for an $A_\infty$ functor induced by the Lagrangian correspondence $\NAT_s$ on $\wrFuk^\star(P^*) \simeq \Tw \cB$. To explain its statement, we use the following technical construction. 
  
  Any twisted complex $\N \in \Tw\cB$ possesses a closed ($\Leftrightarrow \mu^1_{\Tw\cB}(\id)=0$) identity morphism $\oplus_i \id_i=\id  \in \CF(\N,\N)$ consisting of $\id^{\bullet(\circ)}_i\in \CF (\mathbf{a}^{\bullet(\circ)}_i,\mathbf{a}^{\bullet(\circ)}_i)$ for every generator $\mathbf{a}^{\bullet(\circ)}_i \in \N$. Denote the central element 
$$H\coloneqq D_1 + D_2 + S_1S_2 + S_2S_1 \in \cB.$$ 
This defines a different  closed morphism 
$$H\cdot \id  \in \CF(\N,\N),$$ 
by
$$H  \cdot\id^\circ= D_2+S_2 S_1 \in \CF(\arcVer_i,\arcVer_i) \quad \text{ and }  \quad  H \cdot\id^\bullet= D_1+S_1 S_2 \in \CF(\arcHor_j,\arcHor_j)$$
for all generators $\arcVer_i, \arcHor_j \in \N$.

\subsection{\texorpdfstring{The conjectural $A_\infty$ functor}{The conjectural A-infinity functor}}\label{sec:functor}
The algebra $\B$ plays a central role in~\cite{KWZ}. Using the fact that the algebra $\B$ has both cobordism theoretic and symplectic geometric origins, the authors developed Khovanov-theoretic analogues of $R_\pi(T)$ and $R^\natural_\pi(T)$, which we already touched on in Section~\ref{sec:inst_curves}. Namely, using the cube-of-resolution approach, they constructed immersed curves $\BNr(T)$ and $\Khr(T)$, satisfying the gluing formulas~\eqref{eq:kh_gluing}. These new curves are also objects of the category $\wrFuk^\star(P^*) \simeq \Tw\B$, i.e., they do not meet the top left puncture of $P^*$. Moreover, $\BNr(T),~\Khr(T)$  coincide with $R_\pi(T), ~ R^\natural_\pi(T)$ for rational tangles, but differ from $R_\pi(T), ~ R^\natural_\pi(T)$ for more complicated tangles.

Just as the curve $R^\natural_\pi(T)$ is obtained from the curve $R_\pi(T)$ by applying the correspondence $\NAT_s$, the invariant $\widetilde{\text{Kh}}(T)$ is obtained from $\widetilde{\text{BN}}(T)$ by applying a certain functor to $\Tw\B$. This functor was studied in~\cite[Section~3 and Lemma~20]{comparison}, where it was denoted by $\mcH$. Essentially, it takes any object $\N\in \text{Tw}\:\B$, and outputs a mapping cone:
$$\mcH(\N)=[\N\xrightarrow{H\cdot \id} \N],$$
where $H\cdot \id$ is the closed morphism defined above. For the readers acquainted with~\cite{Daemi_Scaduto}, one should think of the element $H$ as a cousin of an element $x$ in the context of the $\ZZ[x]$-module instanton invariant $\widehat{I}(\link)$.

The quasi-equivalence $\wrFuk^\star(P^*) \simeq \Tw\B$ implies that one can consider $\mcH$ acting not only on twisted complexes, but also on curves. Let us inspect this latter action. Taking as an input the bottom arc $\arcHor$ (equal to $R(||)$, the traceless character variety of the crossingless vertical trivial tangle), its corresponding twisted complex consists of a single generator $[\arcHor]$. As such, 
$$\mcH([\arcHor])=[\arcHor_1\xrightarrow{H}\arcHor_2]=[\arcHor_1\xrightarrow{D_1 + S_2 S_1}\arcHor_2],$$
and this twisted complex corresponds (in the same way Figure~\ref{fig:3-twist-tangle-D-str} is related to Figure~\ref{fig:quiverFuk_and_curve}) to the figure eight in Figure~\ref{fig:fig_eight}, which is equal to the traceless character variety modified by the earring  $R_\pi^\nat(||)$. In other words, the two bottom right twisted complexes in the diagram below are homotopy equivalent if $L=R(||)=\arcHor$:
\begin{equation}\label{diago}
\begin{tikzcd}[column sep=3cm, row sep =1.2cm]
L 
\arrow[d,"\text{\cite{HKK,KWZ}}" left] 
\arrow[rr,"\NAT_s"above,"\text{Conjecture~\ref{conj:BW_immersed}}" below]
&  &
(L^\nat,b) \arrow[d,"\text{\cite{HKK,KWZ}}"right]
\\
\N(L)
\arrow[r,"\mcH" above] 
& 
\mcH(\N(L)) 
\arrow[r,phantom, "\underset{\hspace{1cm}}{\simeq}" description,"\text{\scriptsize Conjecture~\ref{conj:action}}"  below]
&
\N(L^\nat,b)
\end{tikzcd}
\end{equation}

A more complicated example arises if we take $L=R(T_3)$ from Figure~\ref{fig:quiverFuk_and_curve}. The corresponding twisted complex is $[\arcVer \xrightarrow{S_1} \arcHor \xrightarrow{D_1} \arcHor \xrightarrow{S_2 S_1} \arcHor]$, and the image under $\mcH$ is depicted in Figure~\ref{fig:fig_eight_D_str}. The dashed arrows there indicate a certain chain homotopy (in fact, a change of basis over $\B$), which results in a twisted complex in Figure~\ref{fig:change_of_basis}; see \cite[Lemma~2.17]{KWZ}. This model of $\mcH(\N(R(T_3)))$ can now be represented by an immersed curve (in the same way Figure~\ref{fig:3-twist-tangle-D-str} is related to Figure~\ref{fig:quiverFuk_and_curve}). The resulting curve in Figure~\ref{fig:quiverFuk_and_8curve} is equal to $R^\nat_\pi(T_3)$, which suggests  that the two bottom right twisted complexes in diagram~\eqref{diago} are homotopy equivalent if $L=R(T_3)$. 

\begin{figure}[H]
  \centering
  \begin{subfigure}[b]{0.3\textwidth}
    \centering
    $$
    \begin{tikzcd}[row sep=0.9cm, column sep=1.5cm]
    \arcVer_{1}
    \arrow{d}{S_1}
    \arrow{r}{D_2 + S_1 S_2}
    &
    \arcVer_{5}
    \arrow{d}{S_1}
    \\
    \arcHor_{2}
    \arrow[ru,dashed,"S_2" description]
    \arrow{d}{D_1}
    \arrow{r}{D_1 + S_2 S_1}
    &
    \arcHor_{6}
    \arrow{d}{D_1}
    \\
    \arcHor_{3}
    \arrow[ru,dashed,"\id" description]
    \arrow{d}{S_2 S_1}
    \arrow{r}{D_1 + S_2 S_1}
    &
    \arcHor_{7}
    \arrow{d}{S_2 S_1}
    \\
    \arcHor_4
    \arrow[ru,dashed,"\id" description]
    \arrow{r}{D_1 + S_2 S_1}
    &
    \arcHor_{8}
    \end{tikzcd} 
    $$
    \caption{$\mcH(\N(R(T_3)))$}
    \label{fig:fig_eight_D_str}
  \end{subfigure}
  \begin{subfigure}[b]{0.3\textwidth}
    \centering
    $$
    \begin{tikzcd}[row sep=0.9cm, column sep=0.4cm]
    \arcVer_{1}
    \arrow{d}{S_1}
    \arrow{r}{D_2}
    &
    \arcVer_{5}
    \arrow{d}{S_1}
    \\
    \arcHor_{2} + S_2\otimes \arcVer_{5}
    \arrow{d}{D_1}
    &
    \arcHor_{6}
    \arrow{d}{D_1}
    \\
    \arcHor_{3} + \arcHor_{6}
    \arrow{d}{S_2 S_1}
    &
    \arcHor_{7}
    \arrow{d}{S_2 S_1}
    \\
    \arcHor_4  + \arcHor_{7}
    \arrow{r}{D_1}
    &
    \arcHor_{8}
    \end{tikzcd} 
    $$
    \caption{With the basis changed}
    \label{fig:change_of_basis}
  \end{subfigure}
  \begin{subfigure}[b]{0.37\textwidth}
  \centering
  \includegraphics[scale=0.9]{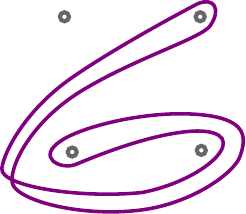}
  \caption{$R_\pi^\nat(T_3)$}
  \label{fig:quiverFuk_and_8curve}
  \end{subfigure}
  \caption{}
\end{figure}

We conclude with an even more complicated example, which shows 
how  the subtle way $\mcH$ acts on curves appears to precisely correspond 
to the bounding cochain appearing in Section~\ref{sec:fig8bubble}. We take $L=A$ to be the arc from Figure~\ref{45torus2fig} in our $K_{4,5}$ example.   The corresponding twisted complex $\N(A)$ is described in Figure~\ref{fig:45curve_typeD}, and its image under $\mcH$ is depicted in Figure~\ref{fig:45curve_typeD_mch}. The dashed arrows there indicate a certain chain homotopy (see \cite[Lemma~2.17]{KWZ}), which results in a homotopy equivalent twisted complex $\mcH(\N(A))'$ in Figure~\ref{fig:45curve_typeD_mch_homotoped}. This representative of the homotopy equivalence class of $\mcH(\N(A))$ is extremely interesting. While it is not in a form that immediately yields a curve on the pillowcase, if we denote by $A^\natural$ the figure eight corresponding to the arc $A$ (see Figure~\ref{fig:45curve_cochains}), the twisted complex $\mcH(\N(A))'$ can be thought of as $\N(A^\natural)$ together with two additional closed self-morphisms, $(\arcHor_3\xrightarrow{D_1} \arcHor_{10})+ (\arcHor_4\xrightarrow{D_1} \arcHor_{11}) $ and $(\arcHor_4\xrightarrow{S_1S_2} \arcHor_{11})+ (\arcVer_5\xrightarrow{S_2S_1} \arcVer_{12})$. It turns out that, under the correspondence of closed morphisms and intersection points in \cite[Theorems~1.5 and~5.22]{KWZ}, these two morphisms correspond to the bounding cochains $b_1$ and $b_2$ respectively, illustrated in Figure~\ref{fig:45curve_cochains}. The composition $A \circ \NAT_s$ is expected to be\footnote{We write ‘‘is expected to be'' because in Theorem~\ref{thmB} we only obtain \emph{homology} figure eights} precisely the figure eight curve $A^\natural$, while the bounding cochain $b_1$ (and $b_2$ too, analogously) is precisely the bounding cochain that arises from the figure eight bubbling, as we argue in Section~\ref{sec:fig8bubble}. Therefore, the functor $\mcH$ on curves appears to be accounting for the bounding cochains arising from the figure eight bubbling, suggesting that the two bottom right twisted complexes in diagram~\eqref{diago} are homotopy equivalent if $L=A$.

For completeness, in Figure~\ref{fig:45curve_typeD_mch_homotoped_further} we describe another representative of the homotopy equivalence class of $\mcH(\N(A))$, which is in the appropriate form to be directly translated into a collection of curves, and we illustrate the latter in Figure~\ref{fig:45curves}. The resulting multicurve can be obtained geometrically from $A^\natural$ by smoothing at bounding cochains.

\begin{figure}[H]
  \centering
  \begin{subfigure}[b]{0.70\textwidth}
    \centering
    $$
    \begin{tikzcd}[row sep=0.7cm, column sep=0.9cm]
    \arcVer_1
    \arrow{r}{D_2}
    &
    \arcVer_2
    \arrow{r}{S_2}
    &
    \arcHor_3
    &
    \arcHor_4
    \arrow[l,"D_1" above]
    \arrow{r}{S_1S_2S_1}
    &
    \arcVer_5
    \arrow{r}{D_2}
    &
    \arcVer_6
    \arrow{r}{S_2}
    &
    \arcHor_7
    \end{tikzcd} 
    $$
    \caption{$\N(A)$.}
    \label{fig:45curve_typeD}
  \end{subfigure}
  \begin{subfigure}[b]{0.25\textwidth}
  \centering
  \includegraphics[scale=0.25]{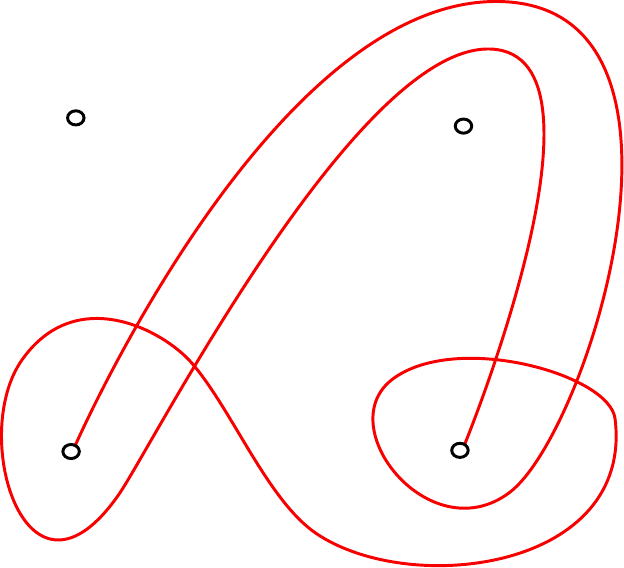}
  \caption{Curve $A$ from Figure~\ref{45torus2fig}.}
  \label{fig:45curve}
  \end{subfigure}
  \\

  \begin{subfigure}[b]{0.70\textwidth}
    \centering
    $$
    \begin{tikzcd}[row sep=0.7cm, column sep=0.9cm]
    \arcVer_1
    \arrow{r}{D_2}
    \arrow{d}{H}
    &
    \arcVer_2
    \arrow{r}{S_2}
    \arrow{d}{H}
    \arrow[dl,"\id" description, dashed]
    &
    \arcHor_3
    \arrow{d}{H}
    \arrow[dl,"S_1" description, dashed]
    &
    \arcHor_4
    \arrow{d}{H}
    \arrow[l,"D_1" above]
    \arrow{r}{S_1S_2S_1}
    &
    \arcVer_5
    \arrow{d}{H}
    \arrow{r}{D_2}
    &
    \arcVer_6
    \arrow{d}{H}
    \arrow[dl,"\id" description, dashed]
    \arrow{r}{S_2}
    &
    \arcHor_7
    \arrow[dl,"S_1" description, dashed]
    \arrow{d}{H}
    \\
    \arcVer_8
    \arrow{r}{D_2}
    &
    \arcVer_9
    \arrow{r}{S_2}
    &
    \arcHor_{10}
    &
    \arcHor_{11}
    \arrow[l,"D_1" above]
    \arrow{r}{S_1S_2S_1}
    &
    \arcVer_{12}
    \arrow{r}{D_2}
    &
    \arcVer_{13}
    \arrow{r}{S_2}
    &
    \arcHor_{14}
    \end{tikzcd} 
    $$
    \caption{$\mcH(\N(A))$.}
    \label{fig:45curve_typeD_mch}
  \end{subfigure}
  \text{\color{white}aaaaaaaaaaaaaaaaaaaaa}
  \\

  \begin{subfigure}[b]{0.70\textwidth}
    \centering
    $$
    \begin{tikzcd}[row sep=1.2cm, column sep=0.9cm]
    \arcVer_1
    \arrow{r}{D_2}
    \arrow{d}{S_2S_1}
    &
    \arcVer_2
    \arrow{r}{S_2}
    &
    \arcHor_3
    \arrow{d}{D_1}
    &
    \arcHor_4
    \arrow[d,bend right, "D_1" left, near start]
    \arrow[d, "S_1S_2" right, near start]
    \arrow[ll, "S_1" description, bend right, dashed]
    \arrow[l,"D_1" above]
    \arrow{r}{S_1S_2S_1}
    &
    \arcVer_5
    \arrow[d, "S_2S_1" right]
    \arrow{r}{D_2}
    &
    \arcVer_6
    \arrow{r}{S_2}
    &
    \arcHor_7
    \arrow{d}{D_1}
    \\
    \arcVer_8
    \arrow{r}{D_2}
    &
    \arcVer_9
    \arrow{r}{S_2}
    &
    \arcHor_{10}
    &
    \arcHor_{11}
    \arrow[lu,"\id" description, dashed, near end]
    \arrow[ru,"S_1" description, dashed, near end]
    \arrow[l,"D_1" above]
    \arrow{r}{S_1S_2S_1}
    &
    \arcVer_{12}
    \arrow{r}{D_2}
    &
    \arcVer_{13}
    \arrow{r}{S_2}
    &
    \arcHor_{14}
    \end{tikzcd} 
    $$
    \caption{$\mcH(\N(A))'$, a homotopy equivalent twisted complex obtained by chain homotopies indicated by dashed arrows in subfigure (c). The twisted complex without the four vertical arrows in the middle corresponds to the curve $A^\natural$ in subfigure (e) and thus can be denoted by $\N(A^\natural)$. Its closed self-morphisms $(\arcHor_3\xrightarrow{D_1} \arcHor_{10})+ (\arcHor_4\xrightarrow{D_1} \arcHor_{11}) $ and  $(\arcHor_4\xrightarrow{S_1S_2} \arcHor_{11})+ (\arcVer_5\xrightarrow{S_2S_1} \arcVer_{12}) $ correspond  to the bounding cochains $b_1$ and $b_2$ in subfigure (e), respectively. 
    Therefore, $\mcH(\N(A))'= \N(A^\natural, b_1+b_2)$.}
    \label{fig:45curve_typeD_mch_homotoped}
  \end{subfigure}
  \begin{subfigure}[b]{0.25\textwidth}
  \centering
  \labellist 
  \pinlabel $b_1$ at 100 60
  \pinlabel $b_2$ at 320 120
  \endlabellist
  \includegraphics[scale=0.25]{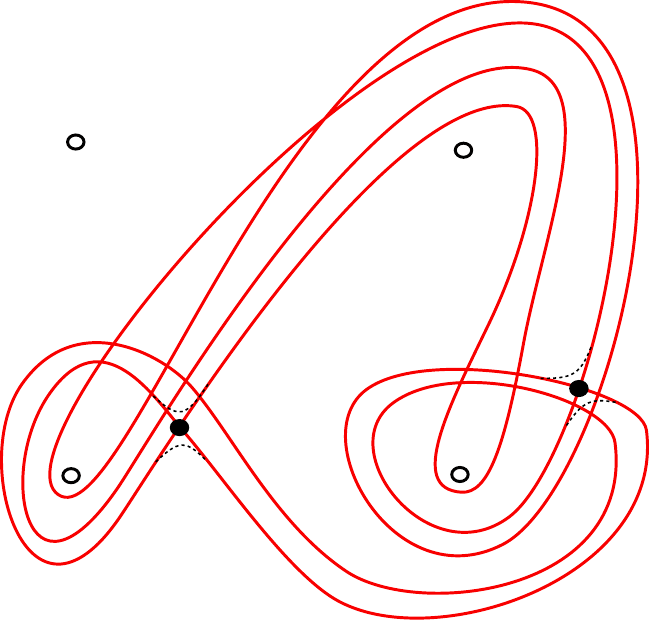}
  \caption{$(A^\natural,b_1+b_2)$.}
  \vspace{2.5cm}
  \label{fig:45curve_cochains}
  \end{subfigure}
  \\

  \begin{subfigure}[b]{0.70\textwidth}
    \centering
    $$
    \begin{tikzcd}[row sep=1.2cm, column sep=0.9cm]
    \arcVer_1
    \arrow{r}{D_2}
    \arrow{d}{S_2S_1}
    &
    \arcVer_2
    \arrow{r}{S_2}
    &
    \arcHor_3
    \arrow{d}{D_1}
    &
    \arcHor_4
    \arrow[d,bend right, "D_1" left, near start]
    \arrow[d, "S_1S_2" right, near start]
    &
    \arcVer_5
    \arrow[d, "S_2S_1" right]
    \arrow{r}{D_2}
    &
    \arcVer_6
    \arrow{r}{S_2}
    &
    \arcHor_7
    \arrow{d}{D_1}
    \\
    \arcVer_8
    \arrow{r}{D_2}
    &
    \arcVer_9
    \arrow{r}{S_2}
    &
    \arcHor_{10}
    &
    \arcHor_{11}
    &
    \arcVer_{12}
    \arrow{r}{D_2}
    &
    \arcVer_{13}
    \arrow{r}{S_2}
    &
    \arcHor_{14}
    \end{tikzcd} 
    $$
    \caption{$\mcH(\N(A))''$, one more homotopy equivalent twisted complex obtained by further chain homotopies indicated by dashed arrows in subfigure (d). The associated multicurve on the right can be obtained by smoothing $A^\natural$ at the bounding cochain.}
    \label{fig:45curve_typeD_mch_homotoped_further}
  \end{subfigure}
  \begin{subfigure}[b]{0.25\textwidth}
  \centering
  \includegraphics[scale=0.25]{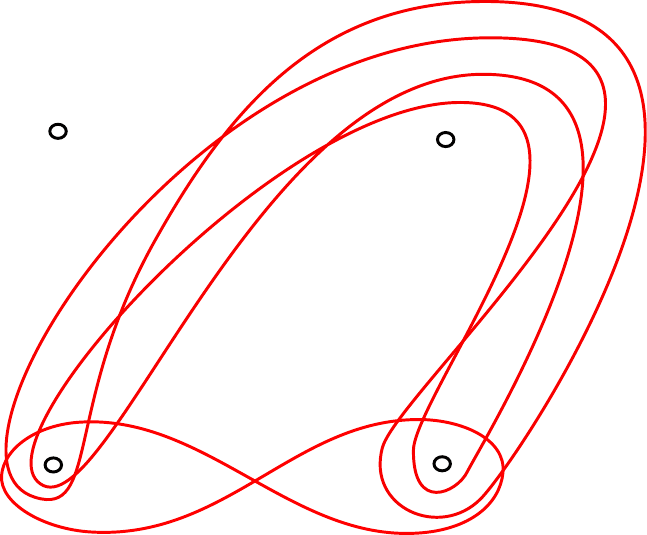}
  \caption{Multicurve associated to $\mcH(\N(A))''$.}
  \vspace{0.8cm}
  \label{fig:45curves}
  \end{subfigure}

  \caption{}
\end{figure}

With the above examples in mind, and given that the Floer field theory predicts the existence of the  $A_\infty$ functor induced by $\NAT_s$, we end with:

\begin{conjectureABC}\label{conj:action} The two bottom right twisted complexes in the diagram~\eqref{diago} are homotopy equivalent for any curve $L$ not touching the top left corner of the pillowcase.  In other words, the hypothetical $A_\infty$ endofunctor induced by immersed Lagrangian correspondence $\NAT_s$
$$\X (\NAT_s) \colon \wrFuk(P^*) \rightarrow \wrFuk(P^*)$$
restricted to the subcategory $\wrFuk^\star(P^*) \simeq \Tw \cB$ is equal to $\mcH$.
\end{conjectureABC}

\appendix
\numberwithin{theorem}{subsection}
\renewcommand{\theequation}{A\arabic{equation}}
\section{Notions from symplectic geometry}\label{Appendix}

\subsection{Fukaya categories, bounding cochains and twisted complexes.}\label{ssec:fukayacat}

We review the notions of the (wrapped) Fukaya  category associated to a symplectic manifold, bounding cochains and twisted complexes. We omit discussion of the many  technicalities (analytic aspects, gradings, perturbations, Novikov rings, etc); these important points are discussed in \cite{FOOO1,FOOO2,Seidel,AJ} (foundations) and \cite{Abo-Sei,Gao} (the origin of wrapped Floer theory, and its adaptation to the immersed setting).

To begin with, we describe the {\em exact}  setting which, is reasonably well-understood.  Assume that $M$ is a \emph{Liouville manifold}, which means that $M$ is non-compact, has cylindrical ends, and satisfies certain convexity assumption at infinity. The objects of the {\em Fukaya category $\Fuk(M)$} are compact embedded exact Lagrangian submanifolds $L\subset M$, perhaps satisfying some extra conditions  (oriented, equipped with a local coefficient system, etc.).

The morphism set between  objects $L_1,L_2$ is  the \emph{Lagrangian Floer complex} $\CF(L_1,L_2)$.  This consists of  the free module over the base ring (in this article, always the field of 2 elements) generated by the transverse intersection points of $L_1'$ and $L_2$ where $L_1'$ is a suitable Hamiltonian deformation of $L_1$. In addition, the morphism sets are equipped with a collection of multilinear \emph{structure maps}:
$$\mu^k \co \CF(L_1,L_2) \otimes \CF(L_2,L_3) \otimes \cdots \CF(L_{k},L_{k+1}) \to \CF(L_1,L_{k+1}),  \ k=1,2,3\ldots $$ defined by counting points in  moduli spaces of rigid holomorphic $(k+1)$-gons in $M$. These satisfy the $A_\infty$ relations 
\begin{equation}\label{eq:A_inf_relations}
\sum_{\substack{1\leq j \leq n \\ 0\leq i \leq n-1}}  \mu^{n-j+1}(a_1,\ldots,a_i,\mu^j(a_{i+1},\ldots, a_{i+j}),a_{i+j+1}\ldots,a_n)
\end{equation}
for each $n\geq 1$. When $n=1$ this simply reads $\mu^1 \circ \mu^1=0$, i.e., $(\CF(L_1,L_2),\mu^1)$ is a chain complex. The map $\mu^2 \co \CF(L_1,L_2)\times \CF(L_2,L_3)\to \CF(L_1,L_3)$ is called composition, and when $n=2$, Equation (\ref{eq:A_inf_relations}) says that associativity for composition fails in general, but is controlled by the further $\mu^k$. In this manner, $\Fuk(M)$ is an \emph{$A_\infty$ category}.
 
The {\em wrapped} Fukaya category $ \wrFuk(M)$  of a Liouville manifold $M$ has objects the set  of  embedded exact Lagrangians $L\to M$, which, if non-compact, are  required to be cylindrical at infinity. The morphism space $\CF(L_1,L_2)$ is generated by the transverse intersection points of $L_1'$ and $L_2$, where $L_1'$ is a deformation of $L_1$ by a specific \emph{quadratic} Hamiltonian. This  ensures that in each end of $M$ the perturbation $L_1'$ wraps infinitely many times; see Figure~\ref{fig:Fuk_pairing} for an example of the wrapping behavior.

There are many Lagrangians in a fixed symplectic manifold $M$, and it is sometimes convenient to study  the full $A_\infty$ subcategory 
corresponding to a finite set of Lagrangians.

\medskip

The $A_\infty$ structure must be weakened, and hence becomes more complicated, in    non-exact settings~\cite{FOOO1,FOOO2},  and when one considers immersed rather than embedded Lagrangians~\cite{AJ,Gao}.    Disk bubbling produces a {\em curvature term} $\mu^0 \in \CF(L,L)$ (its precise definition is involved and may require some completion process). The $A_\infty$ relations~\eqref{eq:A_inf_relations} are modified to include $\mu^0$, which makes the Fukaya category a \emph{curved} $A_\infty$ category. The first relation is
$\mu^1\circ \mu^1 = \mu^2(\mu^0,-) + \mu^2(-,\mu^0)$, so that $\mu^1$ need not be a differential. 

We outline an algebraic procedure used to mitigate the $\mu^1\circ \mu^1\neq 0$ issue. Consider an object $\mathscr L=\oplus_i L_i$ in an additive enlargement of a curved $A_\infty$ category $\Fuk(M)$ (the wrapped case is considered in~\cite{Gao}). Following~\cite[Definition~3.6.4]{FOOO1}, a bounding cochain for $\mathscr L$ is an element in the self-Floer chain complex $b\in \CF(\mathscr L,\mathscr L)$ (consisting of $b_{ij}\in \CF(L_i,L_j)$) which satisfyies the $A_\infty$ Maurer-Cartan equation  
\begin{equation}\label{eq:MC}\sum_{k \geq 0}\mu^k(b,\ldots,b)=0.
\end{equation}
Such a pair   $(\mathscr L, b)$ is  called a \emph{Lagrangian with bounding cochain} or a \emph{twisted complex}\footnote{There are technical differences between these notions in general, but they do not arise in this article.}.

Twisted complexes form an  $A_\infty$ category, denoted by $\Tw \Fuk(M)$. In particular, given two twisted complexes $(\mathscr L,b)$ and $(\mathscr L',b')$, the differential on $\CF((\mathscr L,b),(\mathscr L',b'))$ is defined by  
\begin{equation}\label{eq:deformed_differential}
\begin{split}\partial_{b,b'}=&\mu^1 + \mu^2(b,-)+\mu^2(-,b') + \mu^3(b,b,-)+ \mu^3(b,-,b') + \mu^3(-,b',b') + \\
+&\mu^4(b,b,b,-)+ \mu^4(b,b,-,b') + \mu^4(b,-,b',b') + \mu^4(-,b',b',b') +\cdots
\end{split}
\end{equation}
This differential squares to zero, thanks to Equations~\eqref{eq:MC} and the curved extension of ~\eqref{eq:A_inf_relations}, which includes the $\mu^0$ term.

Even if $\Fuk(M)$ is not curved, the passage $\Fuk(M) \mapsto \Tw \Fuk(M)$ is   useful, because $\Tw \Fuk(M)$ is the smallest $A_\infty$ extension of $\Fuk(M)$ which is closed under the process of taking mapping cones. In addition, if one can find a finite set of objects $\{G_i\}$ (called \emph{generators}) of $\Fuk(M)$ such that the inclusion of the full subcategory on these objects  induces a quasi-equivalence on  twisted complexes,  $\Tw \Fuk(\{G_i\}) \simeq \Tw \Fuk(M)$, then   $\Fuk(\{G_i\})$, and hence $ \Tw \Fuk(\{G_i\})$, is easier to describe than $\Fuk(M)$.

For example, in Section~\ref{sec:wrapped} we explain how two arcs $\arcVer$ and $\arcHor$ generate a particular subcategory of the wrapped Fukaya category of the pillowcase.

 \subsubsection{The pillowcase}
In Section \ref{sec:wrapped} we consider the special case where $M$ is a cylindrical completion of the pillowcase with corners cut:
$$\hat{P}^\delta=P\setminus U_\delta(C)\cup_{\partial U_\delta(C)}\{\text{four copies of }S^1\times[0,+\infty)\}.$$ 
We   use notation $\wrFuk(P^*)$ in place of $\wrFuk(\hat{P}^\delta)$. An example of an object in $\wrFuk(P^*)$ is given in Figure~\ref{fig:big_wrFuk}: the Lagrangian $\arcVer$, illustrated  on the left, defines its wrapped deformation $(\arcVer)'$, illustrated on the right. A basis of $\CF(\arcVer,\arcHor)$ is given by the infinite set of transverse intersection points $\{ S_1, S_1S_2S_1,\ldots \}$ of  $({\arcVer})'$ and $\arcHor$.

Figure~\ref{fig:big_wrFuk} and the surrounding discussion  describes the $A_\infty$ subcategory $\cB \subset \wrFuk(P^*)$ with just two objects, $\arcVer$ and $\arcHor$.

\subsection{Immersed Lagrangian correspondences}\label{sec:appendix_corresp}

We next review some notions of  Wehrheim-Woodward's theory, extrapolated to the immersed case, following Akaho-Joyce \cite{AJ} and Bottman-Wehrheim \cite{BW}. We omit any technical details, as our goal is to review language to illustrate how our calculations can be viewed (philosophically) in the context of the active and still developing subject of immersed Floer theory.  We refer to Wehrheim and Woodward's original papers \cite{WW} for precise constructions in the embedded case.

\begin{definition}\label{def:lagr_corr} 
Given two symplectic manifolds $M$ and $N$, 
an \emph{immersed Lagrangian correspondence  from  $M$ to $N$} is an immersion   $L\looparrowright M^- \times N$ of a smooth manifold $L$  
whose differential takes the tangent spaces   of $L$ onto Lagrangian subspaces of the tangent spaces to $M^-\times N$. 
 We denote it $M\xrightarrow{L} N$. 
If $L$ is embedded it is called an {\em embedded Lagrangian correspondence}. 
\end{definition}

\begin{definition}\label{def:compos}
A {\em composition} of Lagrangian correspondences 
$$M \xrightarrow{L} N \xrightarrow{L'}K \qquad \mapsto \qquad M \xrightarrow{L \circ L'} K  $$
is partially defined in the following way: given the immersions
$$i=i_M\times i_N:L\looparrowright M^-\times N\text{ and } i'=i'_N\times i'_K:L'\looparrowright N^-\times K, $$
  the composition is defined by the {\em fiber product}
$$L\circ L' \coloneqq L\times_N L'  = \{(\ell, \ell')\in L\times L'~|~i_N(\ell)=i_N'(\ell')\}=(i\times i')^{-1}(M\times \triangle_N\times K),
$$
together with the map
\begin{align*}
L\circ L' &\rightarrow M^- \times  K \\
(\ell,\ell') &\mapsto (i_M(\ell), i'_K(\ell')).
\end{align*}
\end{definition} 

In order for $L\circ L'$ to be a smooth manifold of the correct dimension, one must assume that $i\times i'$ is transverse to $M\times \triangle_N\times K$ .  With this assumption, $L\circ L'$ is again an immersed Lagrangian correspondence, and one says {\em the composition is immersed}.  

As a further special case, if $L$ and $L'$ are embedded, $i\times i'$ is transverse to $M\times \triangle_N\times K$, {\em and} the composite $L\circ L':M\to K$ is embedded, we say {\em the composition is embedded}.

Composition with a fixed Lagrangian correspondence $M\xrightarrow{L}N$ induces a function
\begin{equation}\label{corronlag}\Lag^{\pitchfork L}(M)\to \Lag(N) , ~ K\mapsto L\circ K
\end{equation}
where $\Lag(M), \Lag(N)$ denotes the sets of immersed Lagrangians in $M,N$ respectively, and 
$\Lag^{\pitchfork L}(M)\subset \Lag(M)$ denotes the subset of those immersed Lagrangians
 $K\to M$ for which   the composition $L\circ K$ is defined.

The function~\eqref{corronlag} is expected to extend to an $A_\infty$ functor $\Fuk(M)\xrightarrow{L}\Fuk(N)$~\cite{Functor_MWW,BW,Fuk_Ymaps}, where, depending on the setup, $\Fuk(M)$ could be the wrapped Fukaya category~\cite{Gao}.

\begin{definition}
\label{def:glag}
A \emph{generalized immersed Lagrangian correspondence}  
$$ M_0\xrightarrow{ \underline{L}=(L_{01}, L_{12}, \ldots, L_{(k-1)k})} M_k$$
is a sequence of immersed Lagrangian correspondences
$$ M_0 \xrightarrow{ L_{01}} M_1 \xrightarrow{ L_{12}} M_2 \xrightarrow{ L_{23}} \cdots \xrightarrow{ L_{(k-1)k}}M_k$$
\end{definition}
If, for some $i$, the composition of two consecutive correspondences $L_{(i-1)i}$ and $L_{i(i+1)}$  is immersed, 
 then $\underline{L}$ determines the shorter sequence 
\begin{equation}\label{eq:gen_corr_comp}\underline{L}'=(L_{01},\cdots, L_{(i-1)i} \circ L_{i(i+1)} , \cdots, L_{(k-1)k}).\end{equation}

\medskip

In this article (Sections~\ref{sec:PillHom_n_quilts} and~\ref{sec:future_directions} to be precise), the number $k-1$ corresponds the number of surfaces in a decomposition of a link,  $M_0=M_k=\text{pt}$, and we only explicitly consider the cases $k=2,3$:
$$\text{pt} \xrightarrow{ L_{0}} P \xrightarrow{ L_{1}} \text{pt}  \qquad \text{and} \qquad \text{pt} \xrightarrow{ L_{0}} P \xrightarrow{ L_{12}}  P \xrightarrow{ L_{2}}  \text{pt} $$

\subsection{Quilted Floer homology} \label{sec:appendix_quilts}

Wehrheim and Woodward \cite{WW} define the {\em Quilted Floer complex $(\CF(\underline{L}),\partial)$ of a generalized Lagrangian correspondence }
 $\underline{L}\colon M_0\to M_k = M_0$ from $M_0$ to itself.
The  chain group $\CF(\underline{L})$ is generated, under suitable transversality assumptions,  by the set of \emph{generalized intersection points}
 $$\genInt(\underline{L})=
 {i}^{-1}(\triangle)
\subset L_{01}\times\cdots\times L_{(k-1)k} $$
where
$$
i:L_{01}\times L_{12}\times \cdots \times L_{(k-1)k}\to (M_0 \times M_0^-)\times\cdots\times (M_{k-1} \times M_{k-1}^-)$$
 is the immersion induced by assembling the immersions 
 $$L_{01}\looparrowright M_0^-\times M_1, ~ L_{12}\looparrowright M_1^-\times M_2, \ldots , ~ L_{{k-1}k}\looparrowright M_{k-1}^-\times M_0 \qquad (\text{recall }M_0=M_k)$$ and 
 $$\triangle \coloneqq \triangle_{M_0}\times \cdots\times \triangle_{M_{k-1}}.$$
For a generalized intersection point $ x=(x_{01}, x_{12},\ldots, x_{(k-1)k}) \in L_{01}\times\cdots\times L_{(k-1)k}$, define $x_i\in M_i$ to be the image of $x_{i(i+1)}$ under the first coordinate  map $L_{i(i+1)} \to M_i$. Note that it is equal to the image of $x_{(i-1)i}$ under the second coordinate map $L_{(i-1)i} \to M_i$. 
  
Given $x, y \in \cap (\underline{L})$, the coefficient $c_{xy}$ from $\partial(x)=\sum_y c_{xy} \cdot y$ is defined by counting quilted cylinders $\underline{u}\colon \RR \times \RR/ k\ZZ \to (\underline{M},\underline{L})$, i.e., tuples of $J$-holomorphic maps $(u_0, u_1, \cdots, u_k = u_0)$ 
\[
u_i\colon \RR \times [i,i+1]\to M_i
\]
satisfying 
\begin{itemize}
\item[-] the seam condition, that is, an existence of a lift    
$$
\begin{tikzcd}[column sep=5cm]
&
L_{i(i+1)}
\arrow{d}
\\
\RR 
\arrow{r}{(u_{i}(-,i+1),u_{i+1}(-,i+1))}
\arrow[ur,bend left=10, dashed, "\tilde{u}_{i,i+1}"above]
&
M^-_i \times M_{i+1} 
\end{tikzcd}
$$
\item[-] the limiting conditions on the seams
$$\lim_{t\to +\infty} \tilde{u}_{i,i+1}(t) = y_{i (i+1)},~ \lim_{t\to -\infty} \tilde{u}_{i,i+1}(t) = x_{i (i+1)}
$$
\item[-] the limiting conditions on the strips
$$\mbox{ for any $s\in(i,i+1)$, } \lim_{t\to +\infty} u(t,s) = y_i ,~ \lim_{t\to -\infty} u(t,s) = x_i
$$
\end{itemize}

When  $k=2$ and $M_0 = M_2=pt$, $(\CF(\underline{L}),\partial)$ is the usual Lagrangian Floer complex of the pair $L_{01},L_{12}\looparrowright M_1$, as defined in \cite{MR2155230}. Typically,  additional hypotheses are imposed in order to ensure that $\partial^2=0$.

Wehrheim and Woodward prove that if all correspondences in $\underline{L}$ are embedded, suitably transverse, and satisfy extra conditions, such as exactness or monotonicity, then indeed $\partial^2=0$ and the resulting {\em quilted Floer homology}, $\HF(\underline{L})$,  is independent of the choice of almost complex structure,   and is unchanged under Hamiltonian isotopies of the correspondences.

\medskip

Notice that if $\underline{L}'$ is obtained from $\underline{L}$ by composing $L_{(i-1)i}$ with $L_{i(i+1)}$ (as in Equation~\eqref{eq:gen_corr_comp}), then one has a canonical identification
$$\genInt(\underline{L})\simeq \genInt(\underline{L}').$$ 
In particular, the chain groups are preserved by composition.   Wehrheim and Woodward
prove an embedded composition theorem (see \cite[Theorem~5.4.1]{WW} for a precise statement), 
which says that if  the correspondences are embedded, and for some $i$ the composition of $ L_{(i-1)i}$  and $L_{i(i+1)}$ is embedded,  then (with suitable exactness or monotonicity hypotheses)  the obvious bijection on the set of generators induces an isomorphism on homology:
\begin{equation}\label{embeddedcomp}
 \HF(\underline{L})=\HF(\ldots,  L_{(i-1)i} , L_{i(i+1)} ,\ldots )\cong \HF(\ldots,  L_{(i-1)i} \circ L_{i(i+1)} ,\ldots)=\HF(\underline{L}').
\end{equation}

In more general situations, in particular for immersed correspondences and immersed composition as studied in the present article, certain bubbling can occur at ends of the moduli spaces of $J$-holomorphic quilts, which can prevent both $\partial^2=0$ and $\HF(\underline{L}) \cong \HF(\underline{L}')$.

\medskip

Extending  quilted Floer theory to the immersed setting is a topic of intense study. The critical difficulty for upgrading the composition theorem~\eqref{embeddedcomp} lies in the analysis of degeneration in the strip-shrinking process, and in particular {\em  figure eight bubbling}~\cite{BW}. The notion of bounding cochain provides a theoretical means to address the problem. We now briefly describe the geometric approach taken by Bottman-Wehrheim \cite{BW} for producing the appropriate bounding cochains. 

First, generalized immersed Lagrangian correspondences can be enhanced by equipping each Lagrangian  with a bounding cochain:
$$(\underline{L}, {\bf b})=\big((L_{01},b_{01}),\cdots,   (L_{(k-1)k},b_{(k-1)k})\big).$$
An appropriate modification of the quilted Floer differential results in a chain complex $\CF(\underline{L}, {\bf b})$, provided all $b_{i(i+1)}$ satisfy the Maurer-Cartan equation; see~\cite[Section~4.4]{BW}.

With the goal of extending the composition theorem~\eqref{embeddedcomp} to the immersed setting, Bottman and Wehrheim introduce, in terms of counts of Morsified figure eight bubble trees, an extension of composition to a pairing 
$$(L_{01},b_{01})\circ (L_{12},b_{12})=(L_{01}\circ L_{12}, 8(b_{01},b_{12})) \qquad \text{(see \cite[Section~4.4]{BW})}$$
They note that the bounding cochain $8(b_{01},b_{12})$ which they assign to the composition  may be non-zero even when $b_{01}$ and $b_{12}$ are both zero.

\begin{conjecture}[\text{Bottman-Wehrheim \cite{BW}}]
\label{conj:BW_immersed} {\em Let $(\underline{L},{\bf b})$ be a generalized immersed Lagrangian correspondence from $M_0$ to $M_k=M_0$, equipped with bounding cochains. Assume   that for some $i$,  $L_{(i-1)i}$ and $L_{i(i+1)}$ can be composed (Definition \ref{def:lagr_corr}), and let $(\underline{L},'{\bf b}')$  denote the composition
$(\ldots,( L_{(i-1)i} \circ L_{i(i+1)},8(b_{(i-1)i}, b_{i(i+1)})) , \ldots )$. Then,
\[
 \HF(\underline{L},{\bf b})\cong \HF(\underline{L}',{\bf b}').
\]
}
\end{conjecture}

In \cite{Fuk_Ymaps}, Fukaya proves a similar statement using
constructions from \cite{LekiliLipyanskiy}, except that the bounding cochain  associated with $L_{(i-1)i} \circ L_{i(i+1)}$
is defined implicitly.

\newcommand*{\arxivPreprint}[1]{ArXiv preprint \href{http://arxiv.org/abs/#1}{#1}}
\newcommand*{\arxiv}[1]{ArXiv:\ \href{http://arxiv.org/abs/#1}{#1}}
\bibliographystyle{alpha} 
\bibliography{correspondence5}

 \end{document}